\documentclass{amsart}

\usepackage{amssymb,amsmath,amsthm,amsfonts}
\usepackage{enumerate,comment}
\usepackage{mathtools} % to use \coloneqq
\usepackage{todonotes}
\usepackage{xfrac} % for 1/3, 2/3
\usepackage{mleftright} % to use \mleft and \mright instead of \left and \right when the parentheses come after the name of a function: e.g. f\mleft(\bigdot...\mright) (better spacing between f and the left parenthesis)
\usepackage{hyperref}
\usepackage{bbm} % for bb 2, if we want to use it to denote the Sierpinski space
\usepackage{tikz}
\usepackage{stmaryrd}
\usepackage[shortlabels]{enumitem}
\usepackage{tikz-cd}
\usepackage{thmtools} %to use \declaretheorem
\usepackage[capitalise]{cleveref} %to use cref (for references: it adds the words ``Lemma''/''Theorem'' automatically), to be loaded last among similar

\newcommand{\alg}{\mathbf}
\newcommand{\class}{\mathsf}
\newcommand{\cat}{\mathsf}
\newcommand{\functor}{\mathsf}

\newcommand{\abs}[1]{\lvert #1 \rvert}
\newcommand{\pair}[2]{\langle #1, #2 \rangle}
\newcommand{\set}[2]{\{ #1 \mid #2 \}}
\newcommand{\restrict}[2]{#1|_{#2}}
\newcommand{\tuple}[1]{\overline{#1}}

\newcommand{\assign}{\coloneqq}

\newcommand{\iso}{\cong}
\DeclareMathOperator{\dom}{dom}

\renewcommand{\emptyset}{\varnothing}

\newcommand{\into}{\hookrightarrow}
\newcommand{\longto}{\longrightarrow}
\newcommand{\longinto}{\lhook\joinrel\longrightarrow}

\DeclareSymbolFont{forpolishl}{T1}{cmr}{m}{n}
\DeclareMathSymbol{\Luk}{0}{forpolishl}{'212}

\newcommand{\LA}{\alg{L}}
\newcommand{\LS}{\mathbb{L}}
\newcommand{\FS}{\mathbb{F}}
\newcommand{\FA}{\alg{F}}

\newcommand{\X}{\mathbb{X}}
\newcommand{\Y}{\mathbb{Y}}
\newcommand{\Z}{\mathbb{Z}}

\newcommand{\A}{\alg{A}}
\newcommand{\B}{\alg{B}}
\newcommand{\C}{\alg{C}}

\newcommand{\N}{\mathbb{N}}
\newcommand{\Q}{\mathbb{Q}}

\newcommand{\FrameTwo}{\alg{2}}
\newcommand{\Btwo}{\alg{2}_{\cat{BA}}}
\newcommand{\Dtwo}{\alg{2}_{\cat{DL}}}

\newcommand{\MVchain}{[0, 1]}
\newcommand{\PMVchain}{[0, 1]_{+}}

\newcommand{\DiagAlg}{\boldsymbol{\Delta}}
\newcommand{\LeftAlg}{\boldsymbol{\triangleleft}}
\newcommand{\RightAlg}{\boldsymbol{\triangleright}}
\newcommand{\TotAlg}{\boldsymbol{\nabla}}

\newcommand{\HOp}{\mathbb{H}}
\newcommand{\IOp}{\mathbb{I}}
\newcommand{\SOp}{\mathbb{S}}
\newcommand{\POp}{\mathbb{P}}
\newcommand{\PcOp}{\mathbb{P}^{\mathrm{c}}}
\newcommand{\PfrOp}{\mathbb{P}^{\mathrm{fr}}}
\newcommand{\PfinOp}{\mathbb{P}^{\mathrm{fin}}}
\newcommand{\PUOp}{\mathbb{P}_{\mathrm{U}}}
\newcommand{\VOp}{\HOp \SOp \POp}

\DeclareMathOperator{\Con}{Con}
\DeclareMathOperator{\Sg}{Sg}
\DeclareMathOperator{\Sub}{Sub}

\DeclareMathOperator{\Cont}{Cont}
\DeclareMathOperator{\FinRng}{FinRng}

\newcommand{\idmap}{\mathrm{id}}

\newcommand{\op}{\mathrm{op}}
\newcommand{\unit}{\eta}

\DeclareMathOperator{\ev}{ev}

\newcommand{\Alg}{\cat{Alg}_{\LA}}
\newcommand{\Algfv}{\cat{Alg}^{\mathrm{fv}}_{\LA}}
\newcommand{\Set}{\cat{Set}}
\newcommand{\Top}{\cat{Top}}
\newcommand{\LSet}{\cat{Set}_{\LA}}
\newcommand{\LSpa}{\cat{Spa}_{\LA}}
\newcommand{\SepLSpa}{\cat{SepSpa}_{\LA}}
\newcommand{\CRegLSpa}{\cat{CRegSpa}_{\LA}}
\newcommand{\DiscLSpa}{\cat{DiscSpa}_{\LA}}
\newcommand{\CLSpa}{\cat{CSpa}_{\LA}}

\newcommand{\Forget}{\functor{U}}

\newcommand{\UTop}{\Forget_{\cat{Top}}}
\newcommand{\ULSet}{\Forget_{\cat{Set}_{\LA}}}

\newcommand{\Comp}{\mathop{\functor{Comp}}}
\newcommand{\Spec}{\mathop{\functor{Spec}}}
\newcommand{\CComp}{\mathop{\functor{CComp}}}
\newcommand{\Cons}{\mathop{\functor{Cons}}}
\newcommand{\Func}{\mathop{\functor{Func}}}

\newcommand{\Disc}{\mathop{\mathsf{Disc}}}

\newcommand{\Reg}{\mathop{\mathsf{Reg}}}
\newcommand{\Sep}{\mathop{\mathsf{Sep}}}

\newcommand{\Singleton}{\boldsymbol{\ast}}
\newcommand{\equalizer}[2]{\llbracket #1 = #2 \rrbracket}
\newcommand{\maj}{\mathrm{m}}

\newcommand{\I}{\mathcal{I}}
\newcommand{\U}{\mathcal{F}} % ultrafilter
\DeclareMathOperator{\Uf}{Uf}

\DeclareMathOperator{\CS}{CS}

\newcommand{\Sier}{\mathbb{S}}

\makeatletter 
\def\l@subsection{\@tocline{2}{0pt}{2pc}{6pc}{}} 
\makeatother

\newenvironment{enumerateroman}{\begin{enumerate}[label = (\roman*), ref = \roman*]}{\end{enumerate}}

\newtheorem{theorem}{Theorem}[section]
\newtheorem{lemma}[theorem]{Lemma}

\newtheorem{corollary}[theorem]{Corollary}

\newtheorem{fact}[theorem]{Fact}

\theoremstyle{definition}

\newtheorem{definition}[theorem]{Definition}
\newtheorem{notation}[theorem]{Notation}
\newtheorem{example}[theorem]{Example}

\newtheorem{remark}[theorem]{Remark}

\title{Duality for finitely valued algebras}

\author{Marco Abbadini}
\address{Marco Abbadini,
		School of Computer Science,
		University of Birmingham,
		University Rd W,
		Birmingham B15 2TT,
		UK}
\email{m.abbadini@bham.ac.uk}

\author{Adam P\v{r}enosil}
\address{Adam P\v{r}enosil, Departament de Filosofia, Universitat de Barcelona, C. de Montalegre 6, Barcelona 08001, Spain}
\email{adam.prenosil@gmail.com}


\begin{document}

\begin{abstract}
		The theory of natural dualities provides a well-developed framework for studying Stone-like dualities induced by an algebra $\LA$ which acts as a dualizing object when equipped with suitable topological and relational structure. The development of this theory has, however, largely remained restricted to the case where $\LA$ is finite. Motivated by the desire to provide a universal algebraic formulation of the existing duality of Cignoli and Marra for locally weakly finite MV-algebras and to extend it to a corresponding class of positive MV-algebras, in this paper we investigate Stone-like dualities where the algebra $\LA$ is allowed to be infinite.	This requires restricting our attention from the whole prevariety generated by $\LA$ to the subclass of algebras representable as algebras of $\LA$-valued functions of finite range, a distinction that does not arise in the case of finite $\LA$. Provided some requirements on $\LA$ are met, our main result establishes a categorical duality for this class of algebras, which covers the above cases of MV-algebras and positive MV-algebras.
\end{abstract}

\maketitle

\tableofcontents

\section{Introduction}

  At the heart of Stone's celebrated duality~\cite{Stone1936} for Boolean algebras---that is, for the prevariety $\IOp \SOp \POp(\Btwo)$ or, equivalently, the variety $\VOp(\Btwo)$ generated by the two-element Boolean algebra $\Btwo$ with the underlying set $\{ 0, 1 \}$---lies the idea that an abstract Boolean algebra $\A$ may be represented in a concrete way as the algebra $\Cont(X, \{ 0, 1 \})$ of all continuous $\{ 0, 1 \}$-valued functions on a compact (more precisely, Stone) topological space $X$, where the set $\{ 0, 1 \}$ is given the discrete topology, and the Boolean operations of the algebra $\Cont(X, \{ 0, 1 \})$ are computed pointwise. The compact space in question is the so-called spectrum of~$\A$, which is the set of all homomorphisms $\A \to \Btwo$, suitably topologized.

  In the second half of the 20th century, Stone's duality for Boolean algebras inspired a number of further dualities relying on the same pattern. Frequently, these were stated at a universal algebraic level of generality: if a finite algebra $\LA$ satisfies suitable universal algebraic prerequisites, which take the form of assuming the existence of sufficiently many term functions, then a categorical duality obtains between the prevariety $\IOp \SOp \POp(\LA)$ or equivalently the variety $\VOp(\LA)$ generated by $\LA$ and a category of compact topological spaces equipped with some further structure.

  The first result of this kind was Hu's theorem~\cite{Hu1969,Hu1971} stating that in Stone's duality the two-element Boolean algebra $\Btwo$ can be replaced by any primal finite algebra $\LA$ (such as a finite Post algebra), where every finitary operation is a term function. It was soon followed by the seminal duality of Keimel and Werner~\cite{KeimelWerner1974} for any finite algebra $\LA$ that is quasi-primal, i.e.\ such that every finitary operation preserving subalgebras and partial isomorphisms is a term function (such as finite MV-chains), which was in turn generalized by the duality of Davey and Werner~\cite{DaveyWerner1983} for finite $\LA$ with a near unanimity term (such as finite algebras with a lattice reduct).

  These generalizations of Stone's duality add further structure to compact spaces and represent each abstract algebra in $\IOp \SOp \POp(\LA)$ in a concrete way as the algebra of continuous $\LA$-valued functions on some compact space which preserve this structure, again with the discrete topology on $\LA$ and the algebraic operations computed pointwise. Priestley's duality~\cite{Priestley1970} for bounded distributive lattices, which adds order structure on top of the topological structure and restricts to continuous order-preserving $\{ 0, 1 \}$-valued functions, is an example of this. The duality of Davey and Werner is not at all the end of the story---on the contrary, it is one of the foundation stones of the theory of natural dualities~\cite{ClarkDavey1998}, which was extensively developed in the following decades---but for the purposes of this paper we may end our review of existing universal algebraic generalizations of Stone's duality here.

  The reader can easily notice a common thread in the above dualities, namely that the generating object $\LA$ is assumed to be finite. This is not an absolute restriction: some of the work on natural dualities, including~\cite{DaveyWerner1983}, does consider the case of infinite dualizing objects. However, in such cases attention is typically restricted to the setting where $\LA$ is a compact topological algebra, as illustrated by the following quote from an overview paper of Davey~\cite[p.~18]{Davey2015}:
\newlength{\auxlength}
\setlength{\auxlength}{\leftmargini}
\setlength{\leftmargini}{2em}
\smallskip
\begin{quote}
\emph{In the original 1980 Davey–Werner paper~\cite{DaveyWerner1983}, infinite algebras $\underline{\mathbf{M}}$ with a compatible compact topology were allowed. This brings Pontryagin duality for abelian groups under the natural-duality umbrella. As this forces us into the realm of topological algebra and there is a paucity of natural examples, this direction has been little pursued.}
\end{quote}
\smallskip
  Similarly, Davey, Haviar, and Priestley~\cite[p.~247]{DaveyHaviarPriestley2016} have the following to say on the matter:
\smallskip
\begin{quote}
\emph{The authors of~\cite{ClarkDavey1998} took a deliberate decision to restrict their treatment to finitely generated quasivarieties (of algebras). It was already recognised in~\cite{DaveyWerner1983} that finite generation is not a necessary condition for a natural duality to exist but, 30 years on, little general theory has been developed and non-finitely generated examples remain tantalisingly scarce: abelian groups (Pontryagin~\cite{Pontryagin1934}); Ockham algebras~\cite{Goldberg1981,DaveyWerner1985,DaveyWerner1986}; certain semilattice-based algebras~\cite[Section~8]{DaveyJacksonPitkethlyTalukder2007}.}
\end{quote}
\smallskip
\setlength{\leftmargini}{\auxlength}
  The main contribution of this paper is to extend the duality of Davey and Werner to infinite $\LA$ in a direction not pursued in the existing work on natural dualities, namely one that does not ``\emph{force us into the realm of topological algebra}''. Our motivating examples, notably absent from the above list of ``\emph{tantalisingly scarce}'' dualities with infinite $\LA$, will be the dualities of Cignoli, Dubuc, and Mundici~\cite{CignoliDubucMundici2004} and of Cignoli and Marra~\cite{CignoliMarra2012} for certain classes of MV-algebras.

  (To be more precise, we state our main duality result under a strong assumption that Davey and Werner do not impose, namely that there are no homomorphisms between subalgebras of $\LA$ other than inclusions. This assumption is imposed to make the paper more accessible. It can be significantly relaxed, but imposing it has the advantage of tremendously simplifying the presentation of the duality without losing any of our motivating examples. We leave the task of properly formulating the duality without this assumption to future work.)

  The key difference compared to natural dualities in the case of finite $\LA$ is that our duality does not cover the entire class $\IOp \SOp \POp(\LA)$ of what we call \emph{$\LA$-algebras}, but only the subclass $\IOp \SOp \PfrOp(\LA)$ of what we call \emph{finitely valued} $\LA$-algebras, whose definition replaces the class $\POp(\LA)$ of all powers $\LA^{X}$ of $\LA$ by the class $\PfrOp(\LA)$ of their subalgebras consisting of all functions of finite range:
\begin{align*}
  \FinRng(X, \LA) \assign \set{f \in \LA^{X}}{f\colon X \to L \text{ has finite range}} \leq \LA^{X}.
\end{align*}
  This class can equivalently be described as $\Algfv \assign \IOp \SOp \PcOp(\LA)$, where
\begin{align*}
  \PcOp(\LA) \assign \set{\Cont(X, \LA)}{X \text{ compact space}}.
\end{align*}

  In the case covered by our duality results, the class of finitely valued $\LA$-algebras can equivalently be described as $\HOp \SOp \PfrOp(\LA)$, and moreover it contains all locally finite $\LA$-algebras (Fact~\ref{fact: fv = ffv}). Of course, if $\LA$ is finite, then $\PfrOp(\LA) = \POp(\LA)$ and the distinction between arbitrary $\LA$-algebras and finitely valued $\LA$-algebras disappears.

  The universal algebraic theory of natural dualities has largely been developed under the restriction that $\LA$ is finite. However, variants of Stone duality for some specific infinite $\LA$ have already been investigated in the context of MV-algebras. The class of finitely valued $\LA$-algebras for $\LA \assign [0, 1]_{\Q}$ (the rational MV-chain) is precisely the class of locally finite MV-algebras, for which a Stone duality was formulated by Cignoli, Dubuc and Mundici~\cite{CignoliDubucMundici2004}. Building on their work, Cignoli and Marra~\cite{CignoliMarra2012} investigated the case of $\LA \assign [0, 1]$ (the standard MV-chain), obtaining a Stone duality for the strictly larger class of \emph{locally weakly finite finite} MV-algebras. These are precisely the finitely valued MV-algebras in our terminology.\footnote{These classes of MV-algebras are proper subclasses of the class $\IOp \SOp \POp ([0, 1])$ of semisimple MV-algebras, which in our universal algebraic terminology is the class of $\LA$-algebras for $\LA \assign [0, 1]$. The class of semisimple MV-algebras is covered by the duality of Marra and Spada~\cite{MarraSpada2012}, which departs substantially from the Stone-like character of the duality of Cignoli and Marra.} Instead of saying that we are generalizing the duality of Davey and Werner beyond the case of finite $\LA$, we can alternatively say that we are taking the dualities of Cignoli, Dubuc and Mundici and of Cignoli and Marra and stating them at a universal algebraic level of generality.

  A concrete motivating example for pursuing such a generalization is the extension of the duality for finitely valued MV-algebras to finitely valued positive MV-algebras. Namely, consider the reduct $\PMVchain$ of the standard MV-chain $[0, 1]$ in the signature $\{ \oplus, \odot, \vee, \wedge, 0, 1 \}$. This algebra generates the quasivariety $\IOp \SOp \POp \PUOp (\PMVchain)$ of \emph{positive MV-algebras}, introduced in \cite{CabrerJipsenEtAl2019} and further studied in \cite{AbbadiniJipsenEtAl2022,Poiger2024,AbbadiniAglianoEtAl}. Positive MV-algebras are precisely the reducts of MV-algebras in the above signature. They can be thought of as a common generalization of the varieties of MV-algebras and bounded distributive lattices, even though positive MV-algebras themselves do not form a variety. Our universal algebraic duality specializes to a new duality for finitely valued positive MV-algebras, which is in effect the common generalization of the Stone duality of Cignoli and Marra for finitely valued MV-algebras and Priestley duality for bounded distributive lattices.

  Now that we have described the algebraic side of our duality and offered finitely valued positive MV-algebras as a concrete motivating case, let us explain what the spatial side of the duality looks like. Our duality theorem in fact consists of two pieces: a categorical dual equivalence between finitely valued $\LA$-algebras and a class of structured spaces called $\LA$-spaces, and a categorical isomorphism between $\LA$-spaces and more concrete structures called $\LA$-constrained spaces. The former class of spaces is less convenient when trying to get hold of a particular example but easier to work with in general arguments. The latter class of spaces is, in contrast, more tangible in concrete cases but less elegant to work with in general argument.

  An \emph{$\LA$-space} consists of a topological space $X$ equipped directly with an algebra $\Comp \X \leq \Cont(X, \LA)$ of continuous $\LA$-valued functions called the \emph{compatible functions} of $\X$.
We may think of $\Comp \X$ as an ``$\LA$-topological'' structure on top of the ordinary topological structure of $X$. Indeed, the definition of $\LA$-spaces is very reminiscent of the definition of topological spaces, provided that we phrase it using the two-element frame $\FrameTwo \assign \langle \{ 0, 1 \}, \wedge, \bigvee, 0, 1 \rangle$. Namely, if we identify open sets with their characteristic maps into the two-element Sierpi\'{n}ski space $\Sier$ (the set $\{ 0, 1 \}$ with the opens $\emptyset$, $\{ 1 \}$, $\{ 0, 1 \}$), a topological space is precisely the same thing as a set $X$ equipped with a designated subalgebra of $\FrameTwo^{X}$ specifying which functions $f\colon X \to \FrameTwo$ count as continuous.

  $\LA$-spaces are thus sets equipped with a $\FrameTwo$-topological structure plus a compatible $\LA$-topological structure.
 Accordingly, their morphisms $\phi\colon \X \to \Y$ are maps of sets $\phi\colon X \to Y$ which are continuous with respect to both of these structures in precisely the same sense. Namely, for each $f\colon Y \to \FrameTwo$ and each $g\colon Y \to \LA$
\begin{align*}
  f\colon Y \to \FrameTwo \text{ is continuous} & \implies f \circ \phi\colon X \to \FrameTwo \text{ is continuous,} \\
  g\colon Y \to \LA \text{ is compatible} & \implies g \circ \phi\colon X \to \LA \text{ is compatible.}
\end{align*}
  Natural analogues of other topological notions can also be formulated for $\LA$-spaces. The counterpart of the $T_{0}$ separation axiom is the property of being \emph{separated}: for each pair of distinct points $x \neq y$, there is a compatible function $f$ with $f(x) \neq f(y)$.

  Our first main result (Theorem~\ref{thm: cd duality}) is the Congruence Distributive Duality Theorem, or the CD Duality Theorem for short, which states that, under suitable assumptions on $\LA$, a dual equivalence obtains between finitely valued $\LA$-algebras and compact separated $\LA$-spaces. The main hypothesis of this result is that each finitely valued $\LA$-algebra is relatively congruence distributive with respect to some prevariety containing $\LA$. This occurs for example when $\VOp (\LA)$ is congruence distributive, or when $\IOp \SOp \POp (\LA)$ or $\IOp \SOp \POp \PUOp(\LA)$ are relatively congruence distributive.

  The duality falls under the umbrella of concrete dualities formulated by Porst and Tholen~\cite{PorstTholen1991}. In other words, in the proof of the CD Duality Theorem we equip the underlying set $L$ with the structure of a topologically discrete $\LA$-space $\LS$ and then employ the machinery of concrete dualities. A noteworthy difference compared to many other dualities of this type is that $\LS$ is in general not an object within the scope of the CD Duality Theorem, since the discrete space $\LS$ is not topologically compact unless the algebra $\LA$ is finite.

  Our second main result (Theorem~\ref{thm: bp representation}) is the Baker--Pixley Representation Theorem, which shows that the spatial side of this duality can be given a more tangible representation. It states that if $\LA$ has a near unanimity term of arity $k+1$, then the category of compact separated $\LA$-spaces is isomorphic to the category of $k$-ary $\LA$-Priestley spaces, that is, compact separated $k$-ary $\LA$-constrained spaces with the global extension property. We now explain what these terms mean.

  A \emph{$k$-ary $\LA$-constrained space} $\X$ consists of a topological space $X$ equipped with some relational structure, namely a suitable family of \emph{constraints} $\alg{A}_{I} \leq \Cont(I, \LA)$ indexed by sets $I \subseteq X$ of cardinality at most $k$ (which we write more compactly as $I \subseteq_{k} X$). A \emph{continuous compatible function} on (a subset of) a $k$-ary $\LA$-constrained space $\X$ is then a continuous $\LA$-valued function $f$ such that its restriction to each $I \subseteq_{k} X$ is compatible, in the sense that $\restrict{f}{I} \in \A_{I}$. (In the context of $\LA$-spaces, compatible functions were continuous by definition, but now we need to explicitly restrict to continuous functions.) We call an $\LA$-constrained space \emph{separated} if for each $x \neq y$ there is some function $f \in \A_{\{ x, y \}}$ with $f(x) \neq f(y)$. 

  In a $k$-ary $\LA$-constrained space, some local functions $f \in \Cont(I, \LA)$ for $I \subseteq_{k} X$ may not extend to any global compatible function $g \in \Cont(X, \LA)$. The \emph{global extension property} states that these extensions always exist: for each $I \subseteq_{k} X$ and $f\colon I \to \LA$,
\begin{align*}
  f \in \A_{I} \implies f = \restrict{g}{I} \text{ for some compatible } g \in \Cont(X, \LA).
\end{align*}
  In case $\LA$ is the two-element bounded distributive lattice $\Dtwo$, the global extension property is precisely the Priestley separation axiom, up to the correspondence between clopen upsets and order-preserving functions into the topologically discrete two-element poset $0 < 1$ (cf.~Example~\ref{example: priestley}).

  Composing the CD Duality Theorem with the Baker--Pixley Representation Theorem then yields the Near Unanimity Duality Theorem (Theorem~\ref{thm: nu duality}), or the NU Duality Theorem for short, which subsumes a number of familiar dualities including Priestley duality for bounded distributive lattices.\footnote{It is worth recalling here that a finite algebra $\LA$ which generates a congruence distributive variety admits, in a precise sense, a natural duality if and only if $\LA$ has a near unanimity term~\cite{DaveyHeindorfMcKenzie1995}. In other words, if one's notion of a natural duality is restricted to the setting of $\LA$-constrained spaces and finite $\LA$, there is no gap between the CD and NU conditions.}

  The global extension property is, of course, a brute-force local-to-global principle. Whenever possible, it is desirable to replace it with a more local condition. We do so in two cases: the case of unary constraints and the case of finite spaces.

  The first case is where the family of constraints in an $\LA$-constrained space can be recovered from a family of unary constraints. This happens when $\LA$ has a majority term and moreover each subalgebra of $\LA \times \LA$ is either a product subalgebra ($\alg{C}_{1} \times \alg{C}_{2}$ for some $\alg{C}_{1}, \alg{C}_{2} \leq \LA$) or a subalgebra of the diagonal. In this case, a compact separated $\LA$-constrained space has the global extension property if and only if it is topologically a Stone space. This yields a version of the NU Duality Theorem (Theorem~\ref{thm: nu duality unary case}) where the spatial side features Stone spaces with a subalgebra of $\LA$ associated to each point in a continuous way, thus subsuming the duality of Cignoli, Dubuc and Mundici~\cite{CignoliDubucMundici2004} for locally finite MV-algebras and the duality of Cignoli and Marra~\cite{CignoliMarra2012} for finitely valued MV-algebras.

  The second case is the case of finite $\LA$-constrained spaces, which enables us to replace the global extension property by its local version. We say that an $\LA$-constrained space has the \emph{$n$-ary local extension property} if, for each $I \subseteq_{n} X$, each $j \in X$, and each compatible function $f \in \Cont(I, \LA)$, there is some compatible $g \in \Cont(I \cup \{ j \}, \LA)$ such that $f = \restrict{g}{I}$. In the case of $\LA \assign \Dtwo$, the binary local extension property corresponds precisely to the transitivity of the order relation (cf.~Example~\ref{example: priestley local extension}). Given a near unanimity term on $\LA$ of arity $k+1$, we show that a finite $k$-ary $\LA$-constrained space has the global extension property if and only if it has the $n$-ary local extension property for $n \assign k(k - 1)$. This yields a version of the NU Duality Theorem~\ref{thm: nu duality finite case} which subsumes Birkhoff duality for finite bounded distributive lattices (the finite restriction of Priestley duality).

  We claim that one of the contributions of the paper is to set up a \emph{division of labor} between the CD Duality Theorem and the Baker--Pixley Representation Theorem, where one works directly with $\LA$-spaces as much as possible, only reaching for $\LA$-constrained spaces when necessary. We claim that one gains a good deal of clarity from consistently adhering to this division of labor. Working with $\LA$-spaces allows one to entirely bypass the bureaucracy related to keeping track of families of constraints, which is particularly helpful when proving general theorems at the universal algebraic level, without having a fixed $\LA$ in mind. Moreover, the analogy with ordinary topology naturally guides one's attention to notions like ``complete $\LA$-regularity'', which are more difficult to discern by the naked eye from the perspective of $\LA$-constrained spaces. On the other hand, $\LA$-spaces do not themselves provide very tangible representations of particular $\LA$-algebras for a given choice of $\LA$. That task is handled by $\LA$-constrained spaces. 

  In order to make the paper more accessible, its main results are proved under the assumption that there are no homomorphisms between subalgebras of $\LA$ besides inclusion maps. This allows us to simplify the presentation substantially while still covering the motivating cases of MV-algebras and positive MV-algebras. On the other hand, this restriction excludes for instance the case of De Morgan algebras. We shall relax this simplifying assumption in future work. For the time being, let us note that the Near Unanimity Duality Theorem familiar from the theory of natural dualities~\cite[Theorem~3.4]{ClarkDavey1998} does not feature any such restriction.

  We have also chosen to leave the treatment of concrete applications of our duality results to future work, so as not to unduly extend the length of this paper. These applications include a dual description of injective hulls in the category of finitely valued $\LA$-algebras and a dual description of free MV-extensions of positive MV-algebras (see~\cite{AbbadiniJipsenEtAl2022}).

  The outline of the paper is the following. We introduces $\LA$-algebras and $\LA$-spaces in Section~\ref{sec: l-algebras and l-spaces}, including completely $\LA$-regular, separated, and full $\LA$-spaces. In Section~\ref{sec: duality for l-algebras} we set up a dual adjunction between $\LA$-algebras and $\LA$-spaces which uses $\LA$ as a dualizing object. This specializes to a dual equivalence between $\LA$-algebras and completely $\LA$-regular full separated $\LA$-spaces (Theorem~\ref{thm: duality}). This duality has, on its own, little use. The goal of the next two sections is to restrict it to a useful duality. In Section~\ref{sec: finitely valued} we introduce the classes of finitely valued and canonically finitely valued $\LA$-algebras and obtain a dual equivalence between canonically finitely valued $\LA$-algebras and compact full separated $\LA$-spaces (Theorem~\ref{thm: duality canonically fv}). In Section~\ref{sec: jonsson} we impose further restrictions on $\LA$, assuming in particular that finitely valued $\LA$-algebras are relatively congruence distributive with respect to some prevariety containing $\LA$, and obtain a dual equivalence between finitely valued $\LA$-algebras and compact separated $\LA$-spaces (Theorem~\ref{thm: cd duality}). As a corollary, we obtain a representation of relative congruences of finitely valued $\LA$-algebras (Theorem~\ref{thm: representation of congruences}). Finally, in Section~\ref{sec: nu}, we show that compact separated $\LA$-spaces may be represented in a more concrete way in terms of what we call $\LA$-constrained spaces, under the assumption that $\LA$ has a near unanimity term. This allows us to formulate our duality results in a way which directly specializes to a number of existing variants of Stone duality (Theorems~\ref{thm: nu duality}, \ref{thm: nu duality unary case}, and~\ref{thm: nu duality finite case}).

\section{\texorpdfstring{$\LA$-algebras and $\LA$-spaces}{L-algebras and L-spaces}}\label{sec: l-algebras and l-spaces}

  In this section, we introduce $\LA$-algebras and $\LA$-spaces. These form the ambient classes of objects inside which we shall try to find well-behaved classes admitting a Stone-like duality.

  Throughout the paper, we fix an algebra $\LA$ and restrict to algebras in the signature of $\LA$. Whenever topology is involved, we take $\LA$ to be \emph{topologically discrete}. For example, even in the special case of MV-algebras, the algebra $\LA \assign [0,1]$ will be equipped with the discrete topology rather than the Euclidean topology.

\begin{notation}
  Given a set $X$, projection maps from $\LA^{X}$ onto some component $x \in X$ or onto some set of components $I \subseteq X$ will be denoted as follows:
\begin{align*}
  & \pi_{x}\colon \LA^{X} \to \LA, & & \pi_{I}\colon \LA^{X} \to \LA^{I}.
\end{align*}
  We also write $\restrict{f}{J} \assign \pi_{J}(f)$ for the restriction of a function $f\colon I \to \LA$ to a set $J \subseteq I$. For products of the form $X_{1} \times \dots \times X_{n}$, the projection maps will be written as $\pi_{i}\colon X_{1} \times \dots \times X_{n} \to X_{i}$ for $i \in \{ 1, \dots, n \}$. We write $f_{x} \assign f(x)$ for the value of a function $f\colon X \to \LA$ at some point $x \in X$.
\end{notation}

  \begin{notation}
  	For an algebra $\A$ and a subset $S \subseteq \A$, we let $\Sg^\A(S)$ denote the subalgebra of $\A$ generated by $S$.
  	We write $\Sg^\A(a)$ for $\Sg^\A(\{a\})$.
  \end{notation}

  The category of sets and functions will be denoted by $\Set$, and the category of topological spaces and continuous maps by $\Top$.

\subsection{\texorpdfstring{$\LA$-algebras and $\LA$-spaces}{L-algebras and L-spaces}}

\begin{definition}
  An \emph{$\LA$-algebra} is an algebra isomorphic to an algebra of the form $\A \leq \LA^{X}$ for some set $X$. Equivalently, the class of $\LA$-algebras is $\Alg \assign \IOp \SOp \POp (\LA)$, where the class operators $\IOp$, $\SOp$, $\POp$ denote the closure of a class of algebras under isomorphic images, subalgebras and products.
\end{definition}

  If $\LA$ has no constants in its signature, we allow for the empty $\LA$-algebra. The product of the empty family of $\LA$-algebras is a singleton algebra. Accordingly, every singleton algebra is an $\LA$-algebra.

\begin{remark}
	The class $\Alg$ is a \emph{prevariety}, i.e.\ it is closed under the class operators $\IOp$, $\SOp$ and $\POp$. It is the prevariety \emph{generated} by $\LA$, i.e.\ the smallest prevariety containing $\LA$. If $\LA$ is finite, then every $\LA$-algebra is locally finite and $\Alg$ is in fact a \emph{quasivariety}, i.e.\ it is in addition closed under $\PUOp$ (ultraproducts).
\end{remark}

\begin{remark}
  In many cases of interest, the class of all $\LA$-algebras is in fact a variety, i.e.\ $\Alg = \VOp(\LA)$, where $\HOp$ denotes the closure under homomorphic images. This happens in particular if $\LA$ is a finite algebra that generates a congruence distributive variety (for example, due to having a lattice reduct) and every non-empty nontrivial subalgebra of $\LA$ is simple (this is a consequence of J\'onsson's lemma \cite[Thm.~IV.6.8]{BurrisSankappanavar1981}).
\end{remark}

  We shall treat $\Alg$ (or indeed any class of algebras) as a category where the morphisms are homomorphisms of algebras.

\begin{notation}
  We shall use the notation:
\begin{enumerateroman}
\item $\Btwo$ for the two-element Boolean algebra $0 < 1$,
\item $\Dtwo$ for the two-element bounded distributive lattice $0 < 1$,
\item $\MVchain$ for the standard MV-chain,
\item $\PMVchain$ for the standard positive MV-chain.
\end{enumerateroman}
\end{notation}

  The standard positive MV-chain is the reduct of the MV-chain $\MVchain$ in the signature $\{ \oplus, \odot, \vee, \wedge, 0, 1 \}$. \emph{Positive MV-algebras}, introduced in \cite{CabrerJipsenEtAl2019} and further studied in \cite{AbbadiniJipsenEtAl2022,Poiger2024,AbbadiniAglianoEtAl}, are the quasivariety generated by $\PMVchain$. The prevariety generated by $\PMVchain$ is the subclass of semisimple positive MV-algebras~\cite[Section~4]{AbbadiniAglianoEtAl}.

\begin{example}
  These algebras yield the following classes of $\LA$-algebras:
\begin{enumerateroman}
\item the variety of Boolean algebras for $\LA \assign \Btwo$,
\item the variety of bounded distributive lattices for $\LA \assign \Dtwo$,
\item the class of semisimple MV-algebras for $\LA \assign \MVchain$,
\item the class of semisimple positive MV-algebras for $\LA \assign \PMVchain$.
\end{enumerateroman}
\end{example}

  $\LA$-algebras were defined as the algebras of some $\LA$-valued functions on some set $X$ (under the pointwise operations), up to isomorphism. We can equivalently define them as the algebras of some continuous $\LA$-valued functions on some space.

\begin{definition}
  The continuous $\LA$-valued functions on a topological space $X$ (with $\LA$ topologically discrete) form an algebra
\begin{align*}
  \Cont(X, \LA) \assign \set{f \in \LA^{X}}{f\colon X \to \LA \text{ is continuous}} \leq \LA^{X}.
\end{align*}
This is true because, for every $n \in \N$, any function $\LA^n \to \LA$ is continuous, since $\LA$ is discrete (and so the interpretation of every function symbol is continuous, and so maps a tuple of continuous functions to a continuous function).
\end{definition}

\begin{definition}
  An \emph{$\LA$-representation} of an $\LA$-algebra $\A$ (on a set $X$) is an embedding ${\rho\colon \A \into \LA^{X}}$. A \emph{continuous $\LA$-representation} of $\A$ (on a topological space $X$) is an embedding $\rho\colon \A \into \Cont(X, \LA)$.
\end{definition}

\begin{remark}
  Only singleton $\LA$-algebras and the empty $\LA$-algebra (if it exists) have an $\LA$-representation over $X \assign \emptyset$. The $\LA$-algebras with an $\LA$-representation over a singleton set $X$ are up to isomorphism precisely the subalgebras of $\LA$.
\end{remark}

  The main theme of the present paper is the study of $\LA$-algebras through their representations. However, instead of working with $\LA$-representations as embeddings of an $\LA$-algebra into $\LA^{X}$ or $\Cont(X, \LA)$, it will generally be more convenient to work directly with a set or a space $X$ equipped with a subalgebra of $\LA^{X}$ or $\Cont(X, \LA)$.

\begin{definition}
\label{d:L-space}
  An \emph{$\LA$-set} is a pair $\X \assign \langle X, \A \rangle$ consisting of
\begin{enumerateroman}
\item a set $X$ (the \emph{underlying set} of $\X$),
\item an algebra $\A \leq \LA^{X}$ (the algebra of \emph{compatible functions} on $\X$).
\end{enumerateroman}
  Similarly, an \emph{$\LA$-space} is a pair $\X \assign \langle X, \A \rangle$ consisting of
\begin{enumerateroman}
\item a topological space $X$ (the \emph{underlying space} of $\X$),
\item an algebra $\A \leq \Cont(X, \LA)$ (the algebra of \emph{compatible functions} on $\X$).
\end{enumerateroman}
\end{definition}

\begin{notation}
  Given an $\LA$-set or an $\LA$-space $\X \assign \langle X, \A \rangle$, we use the notation $\Comp \X \assign \A$. We shall also simply write $X$, $Y$ for the underlying sets or spaces of the $\LA$-sets or $\LA$-spaces $\X$, $\Y$.
\end{notation}

\begin{definition}
  Given $\LA$-sets $\X$ and $\Y$, an \emph{$\LA$-map} ${\phi\colon \X \to \Y}$ is a map of sets ${\phi\colon X \to Y}$ which \emph{reflects compatibility}:
\begin{align*}
  g \in \Comp \Y \implies g \circ \phi \in \Comp \X.
\end{align*}
\end{definition}

  The category of $\LA$-sets and $\LA$-maps will be denoted by $\LSet$, and the category of $\LA$-spaces and continuous $\LA$-maps by $\LSpa$.

  The assignment $\X \mapsto \Comp \X$ extends to a contravariant functor, namely the \emph{compatible algebra functor} $\Comp\colon \LSet^{\op} \to \Alg$ ($\Comp\colon \LSpa^{\op} \to \Alg$), which takes a (continuous) $\LA$-map $\phi\colon \X \to \Y$ to the homomorphism
\begin{align*}
  \Comp \phi\colon \Comp \Y & \longto \Comp \X \\
  g & \longmapsto g \circ \phi.
\end{align*}
  The definition of an $\LA$-map is devised precisely so that $\Comp \phi$ is well-defined.

\begin{remark}
  The definitions of $\LA$-sets and $\LA$-maps are entirely analogous to the definitions of topological spaces and continuous maps, with the algebra $\LA$ taking the place of the two-element frame $\FrameTwo$ (the two-element chain $0 < 1$ with finite meets and arbitrary joins as primitive operations). A topological space is precisely a set $X$ equipped with an algebra $\A \leq \FrameTwo^{X}$, and given topological spaces represented as sets $X$ and $Y$ equipped with algebras $\A \leq \FrameTwo^{X}$ and $\B \leq \FrameTwo^{Y}$, a continuous map $\phi\colon X \to Y$ is precisely a map of sets such that $g \in \B$ implies $g \circ \phi \in \A$.
\end{remark}

  In other words, $\LA$-sets are sets with an $\LA$-topological structure instead of an ordinary topological structure, while $\LA$-spaces have both an $\LA$-topological structure and an ordinary topological structure. As we shall now see, in some cases the topological structure can be recovered from the $\LA$-topological structure.

\subsection{\texorpdfstring{Completely $\LA$-regular $\LA$-spaces}{Completely L-regular L-spaces}}

  The category of $\LA$-spaces comes with three natural forgetful functors: one forgets about the topology, one about the algebra of compatible functions, and one about both. These are, respectively, the underlying $\LA$-set functor, the underlying space functor, and the underlying set functor:
\begin{align*}
  & \ULSet\colon \LSpa \to \LSet, & & \UTop\colon \LSpa \to \Top, & & \abs{\cdot}\colon \LSpa \to \Set.
\end{align*}
  We first consider the underlying $\LA$-set functor $\ULSet$, which assigns to each $\LA$-space its $\LA$-set reduct. It has a left and a right adjoint, corresponding to the two natural ways in which an $\LA$-set $\X$ can be given a topological structure.

  The first option is to take $X$ to be a discrete topological space. The second option is to take the initial topology on $X$ with respect to $\Comp \X$, i.e.\ the topology generated by the sets $f^{-1}[\{ a \}]$ for $f \in \Comp \X$ and $a \in \A$. This yields the \emph{discretization} and the \emph{regularization} functors, respectively:
\begin{align*}
  & \Disc\colon \LSet \to \LSpa, & & \Reg\colon \LSet \to \LSpa.
\end{align*}
  Their action on maps is $\Disc \phi \assign \phi$ and $\Reg \phi \assign \phi$. ($\Reg \phi$ is continuous for each $\LA$-map $\phi\colon \X \to \Y$ because $\phi^{-1}[f^{-1} [\{ a \}]] = (f \circ \phi)^{-1} [\{ a \}]$ for $f \in \Comp \Y$ and~$a \in \LA$.)

\begin{definition}
  An $\LA$-space $\X$ is \emph{completely $\LA$-regular} if its topology is initial with respect to $\Comp \X$, or equivalently if the sets $f^{-1}[\{ a \}]$ for $f \in \Comp \X$ and $a \in \A$ form a basis for the underlying space of $X$.
\end{definition}

  That is, $\X$ is completely $\LA$-regular if $\X = \Reg \ULSet \X$. The topological structure of a completely $\LA$-regular $\LA$-space can thus be recovered from its $\LA$-set structure. The category of completely $\LA$-regular $\LA$-spaces (as a full subcategory of $\LSpa$) will be denoted by $\CRegLSpa$, and the category of discrete $\LA$-spaces by $\DiscLSpa$.

\begin{fact} \label{fact: creg adjunction}
  $\Disc$ is left adjoint to $\ULSet$, and $\Reg$ is right adjoint to $\ULSet$, the unit and counit in both cases being the identity maps. Consequently, $\ULSet$ and $\Disc$ form an equivalence between $\LSet$ and $\DiscLSpa$, and $\ULSet$ and $\Reg$ form an equivalence between $\LSet$ and $\CRegLSpa$.
\end{fact}

\begin{proof}
  We have an isomorphism $\LSet(\ULSet \X, \Y) \iso \LSpa(\X,  \Reg \Y)$ (natural in $\X$ and $\Y$) because each $\LA$-map from an $\LA$-space into a completely regular $\LA$-space is continuous. Similarly, we have an isomorphism $\LSet(\X, \ULSet \Y) \iso \LSpa(\Disc \X, \Y)$ (natural in $\X$ and $\Y$) because each $\LA$-map from a discrete $\LA$-space into an $\LA$-space is continuous.
\end{proof}

\subsection{\texorpdfstring{Separated and full $\LA$-spaces}{Separated and full L-spaces}}

  Before turning our attention to the underlying space functor $\UTop$, let us first introduce the $\LA$-space analogue of the $T_{0}$ separation axiom. Recall that a topological space is $T_0$ if each pair of distinct points $x \neq y$ is \emph{topologically distinguishable}: there is some open set $U$ such that either $x \in U$ and $y \notin U$ or $x \notin U$ and $y \in U$.

\begin{definition}
  The $\LA$-set or $\LA$-space $\X$ is \emph{separated} if for each $x \neq y$ in $X$ there is a compatible function $f \in \Comp \X$ such that $f_{x} \neq f_{y}$.
\end{definition}

  Equivalently, separation states that for each homomorphism $h\colon \Comp \X \to \LA$ there is at most one point $x \in X$ such that $h = \pi_{x}$. The property dual to separation will be called fullness.

\begin{definition}
  An $\LA$-space $\X$ is \emph{full} if for each homomorphism $h\colon \Comp \X \to \LA$ there is $x \in X$ such that $h(f) = f_{x}$ for all $f \in \Comp \X$.
\end{definition}

  The category of separated $\LA$-spaces, as a full subcategory of $\LSpa$, will be denoted by $\SepLSpa$. In a separated $\LA$-space $\X$ the discrete topology of $\LA$ forces certain topological properties to hold in $\X$.

\begin{definition}
  A topological space $X$ is called \emph{zero-dimensional} if it has a basis (or equivalently a subbasis) of clopen sets. It is a \emph{Stone space} if it is a compact Hausdorff zero-dimensional space, or, equivalently, a compact space in which clopens separate points.
\end{definition}

\begin{lemma} \label{lemma: properties of separated l-spaces} \label{l:separating-then}
~
	\begin{enumerateroman}
		\item \label{i:sep-is-Haus} 
		In each separated $\LA$-space, clopens separate points.
		\item \label{i:sep-of-compact-is-Stone}
		Each compact separated $\LA$-space is Stone.
		\item \label{i:sep-of-finite-is-discrete}
		Each finite separated $\LA$-space is discrete.
	\end{enumerateroman}
\end{lemma}

\begin{proof}
	The first claim holds because for distinct $x, y \in X$ there is a clopen $U \subseteq X$ that separates $x$ and $y$, namely $U \assign f^{-1} [\{ f_{x} \}]$ for any $f \in \Comp \X$ which separates $x$ and $y$. The other claims are immediate consequences.
\end{proof}

  Every $\LA$-space has a separated quotient:

\begin{definition}
\label{def: separated quotient}
  Each $\LA$-space $\X$ determines an equivalence relation $\theta_{\X}$ on $X$:
\begin{align*}
  \pair{x}{y} \in \theta_{\X} \text{ for } x, y \in X \iff f_{x} = f_{y} \text{ for each } f \in \Comp \X.
\end{align*}
  Each $f \in \Comp \X$ then determines a function
  \begin{align*}
  f / \theta_{\X} \colon X / \theta_{X} & \longrightarrow \LA, \\
  x / \theta_{\X} &\longmapsto f_{x},
\end{align*}
  which is continuous with respect to the quotient topology on $X / \theta_{\X}$.

  The \emph{separated quotient} of $\X$, denoted by $\Sep \X$, consists of the space $X / \theta_{\X}$ with the quotient topology and $\Comp \Sep \X \assign \set{f / \theta_{\X} \in \Cont(X / \theta_{\X}, \LA)}{f \in \Comp \X}$. The quotient map $\pi_{\X}\colon x \mapsto x / \theta_{\X}$ is a continuous $\LA$-map $\pi_{\X}\colon \X \to \Sep \X$.
\end{definition}

\begin{remark}
  The equivalence relation $\theta_{\X}$ is closed as a subset of $X \times X$: for each $a \in \alg{A}$ the relation $\theta_{a} \assign \set{\pair{x}{y} \in \A}{a_x = a_y}$ is clopen due to being the preimage of the (clopen) equality relation $\Delta_{\LA} \subseteq \LA \times \LA$ under the continuous function $a \times a \colon \A \times \A \to \LA \times \LA$, and $\theta_{\X} = \bigcap_{a \in \A} \theta_{a}$.
\end{remark}

  The separated quotient construction yields a functor $\Sep\colon \LSpa \to \SepLSpa$ mapping a continuous $\LA$-map $\phi\colon \X \to \Y$ to the continuous $\LA$-map
\begin{align*}
  \Sep \phi\colon \Sep \X & \longto \Sep \Y \\
  x / \theta_{\X} & \longmapsto \phi(x) / \theta_{\Y}.
\end{align*}
  This map is well-defined: if $\pair{x}{y} \in \theta_{\X}$, then $g_{\phi(x)} = (g \circ \phi)_{x} = (g \circ \phi)_{y} = g_{\phi(y)}$ for each $g \in \Comp \Y$, and so $\pair{\phi(x)}{\phi(y)} \in \theta_{\Y}$.

\begin{fact}
  The separated quotient functor $\Sep\colon \LSpa \to \SepLSpa$ is left adjoint to the inclusion functor $\SepLSpa \into \LSpa$, the unit being the quotient map~$\pi$.
\end{fact}

\begin{proof}
  Consider $\LA$-spaces $\X$ and $\Y$ with $\Y$ separated. Then we have an isomorphism $\LSpa(\X, \Y) \iso \LSpa(\Sep \X, \Y)$ (natural in $\X$ and $\Y$): clearly each continuous $\LA$-map $\Sep \X \to \Y$ yields a continuous $\LA$-map $\X \to \Y$ when precomposed with $\pi_{\X}$. Conversely, for each continuous $\LA$-map $\phi$ and all $x, y \in \X$, if $\pair{x}{y} \in \theta_{\X}$, then $(g \circ \phi)_{x} = (g \circ \phi)_{y}$ for every $g \in \Comp \Y$, so $g_{\phi(x)} = g_{\phi(y)}$ for every $g \in \Comp \Y$, and thus $\phi(x) = \phi(y)$ because $\Y$ is separated. The continuous $\LA$-map $\phi$ therefore factors as $\phi =  \psi \circ \pi_{\X}$, where $\psi\colon x / \theta_{\X} \mapsto \phi(x)$ is a continuous map because topologically $\Sep \X$ is a quotient of $\X$, and it is an $\LA$-map because $g \in \Comp \Y$ implies $g \circ \phi \in \Comp \X$ and thus $(g \circ \phi) / \theta_{\X} \in \Comp \Sep \X$. Moreover, this map $\psi$ is the unique map $\psi\colon \Sep \X \to \Y$ such that $\phi = \psi \circ \pi_{\X}$.
\end{proof}

\begin{lemma} \label{lemma: separated quotient preserves properties}
	Consider an $\LA$-space $\X$. Then:
	\begin{enumerateroman}
		\item \label{i:sep-quot-is-Haus}
		In $\Sep \X$, clopens separate points.
		\item \label{i:sep-quot-of-compact-is-Stone}
		If $\X$ is compact, then $\Sep \X$ is Stone.
		\item \label{i:sep-quot-of-compact-is-full}
		If $\X$ is full, then so is $\Sep \X$.
		\item \label{i:sep-quot-creg} If $\X$ is a completely $\LA$-regular $\LA$-space, then so is $\Sep \X$, and moreover $\theta_{\X}$ is the relation of topological indistinguishability.
	\end{enumerateroman}
\end{lemma}

\begin{proof}
	\eqref{i:sep-quot-is-Haus} follows from \cref{lemma: properties of separated l-spaces}.\eqref{i:sep-is-Haus}.
		
	\eqref{i:sep-quot-of-compact-is-Stone} follows from \eqref{i:sep-quot-is-Haus} because every quotient of a compact space is compact.

	\eqref{i:sep-quot-of-compact-is-full} holds because composing each homomorphism $h\colon \Comp \Sep \X \to \LA$ with the map $f \mapsto f / \theta_{\X}$ yields a homomorphism $\Comp \X \to \LA$, which has the form $\pi_{x}$ for some $x \in \X$, and so $h$ has the form $\pi_{x / \theta_{\X}}$.

	\eqref{i:sep-quot-creg} if $\X$ is a completely $\LA$-regular $\LA$-space $\X$, then
\begin{align*}
  \pair{x}{y} \in \theta_{\X} \iff & f_{x} = f_{y} \text{ for each } f \in \Comp \X \\
  \iff & \text{for each } f \in \Comp \X \text{ and } a \in \LA \\
   & x \in f^{-1} [\{ a \}] \text{ if and only if } y \in f^{-1} [\{ a \}], \\
  \iff & \text{$x$ and $y$ are topologically indistinguishable},
\end{align*}
  using the fact that the sets of the form $f^{-1} [\{ a \}]$ form a basis for $\X$. The rest of the claim holds because each basic open $f^{-1} [\{ a \}] / \theta_{\X}$ of $\Sep \X$ has the form $(f / \theta_{\X})^{-1} [\{ a \}]$ and $f / \theta_{\X} \in \Comp \Sep \X$.
\end{proof}

\section{\texorpdfstring{Duality for $\LA$-algebras}{Duality for L-algebras}}
\label{sec: duality for l-algebras}

  In this section, we establish a dual adjunction between $\LA$-algebras and $\LA$-spaces, and restrict it to a dual equivalence between $\Alg$ and a full subcategory of $\LSpa$, namely the category of completely $\LA$-regular full separated $\LA$-spaces. This duality is not in and of itself a very useful one, owing to the fact that being completely $\LA$-regular and being full are complicated conditions, but it will serve as a springboard towards the more useful dualities established later in this paper.

\subsection{\texorpdfstring{The spectrum of an $\LA$-algebra}{The spectrum of an L-algebra}}

  Our first task is to describe a left adjoint of the functor $\Comp\colon \LSpa^{\op} \to \Alg$ which assigns to each $\LA$-space its algebra of compatible functions. Given $\LA$-algebras $\A$ and $\B$, we use $\Alg(\A, \B)$ to denote the set of all homomorphisms from $\A$  to $\B$.

\begin{definition}
\label{d:spectrum}
  The \emph{spectrum} of an $\LA$-algebra $\A$ is the $\LA$-space
\begin{align*}
  \Spec \A \assign \langle \Alg(\A, \LA), \unit_{\A}[\A] \rangle,
\end{align*}
   where
\begin{enumerateroman}
\item $\Alg(\A, \LA) \subseteq \LA^{A}$ has the subspace topology of the product topology, and
\item $\unit_{\A}$ is the following embedding of $\LA$-algebras:
\begin{align*}
  \unit_{\A}\colon \A & \longinto \LA^{\Alg(\A, \LA)} \\
  a & \longmapsto (h \mapsto h(a)).
\end{align*}
\end{enumerateroman}
We will call $\unit_\A$ the \emph{canonical representation} of $\A$.
\end{definition}

\begin{remark} \label{rem: clopen subbasis}
  Because $\LA$ is topologically discrete, $\Alg(\A, \LA)$ is a zero-dimensional Hausdorff space. One subbasis of $\Alg(\A, \LA)$ consists of the clopen sets
\begin{align*}
  & U_{a \mapsto W} \assign \set{h\colon \A \to \LA}{h(a) \in W} & & \text{for } a \in \A \text{ and } W \subseteq \LA.
\end{align*}
  The complement of a set of the form $U_{a \mapsto W}$ is another set of the same form, namely $U_{a \mapsto (\LA - W)}$.
  Therefore, the sets $U_{a \mapsto W}$ are indeed clopen and form a subbasis for the closed sets in addition to being a subbasis for the opens. The continuity of $\unit_{\A}$ follows from the observation that $\unit_{\A}(a)^{-1}[W] = U_{a \mapsto W}$.
\end{remark}

  Each homomorphism of $\LA$-algebras $h\colon \A \to \B$ induces a continuous $\LA$-map
\begin{align*}
  \Spec h\colon \Spec \B & \longto \Spec \A, \\
  g & \longmapsto g \circ h.
\end{align*}
  Consequently, $\Spec$ extends to a functor $\Spec\colon \Alg \to \LSpa^{\op}$.

\begin{lemma} \label{lemma: spec is full separated}
  $\Spec \A$ is a completely $\LA$-regular full separated $\LA$-space for ${\A \hskip-0.25pt \in \hskip-0.25pt \Alg}$.
\end{lemma}

\begin{proof}
  Separation is clear from the definition of $\Spec \A$. To prove that $\Spec \A$ is full, consider a homomorphism $h\colon \Comp \Spec \A \to \LA$. Each $f \in \Comp \Spec \A$ has the form $\unit_{\A}(a)$ for some $a \in \A$, and so $h(f) = h(\unit_{\A}(a)) = (h \circ \unit_{\A})(a) = \unit_{\A}(a)(h \circ \unit_{\A}) = f_{x}$ for $x \assign h \circ \unit_{\A}$. The $\LA$-space $\Spec \A$ is $\LA$-regular because the sets $U_{a \mapsto b}$ for $a \in \A$ and $b \in \LA$ form a basis for $\Spec \A$, and $U_{a \mapsto b}$ has the form $f^{-1}[\{ b \}]$ for $f \assign\unit_{\A}(a)$.
\end{proof}

\begin{fact} \label{fact: unit is iso}
  For every $\LA$-algebra~$\A$, the map $\unit_{A}\colon\hskip-0.8pt \A \into \Comp \Spec \A$ is an isomorphism.
\end{fact}

\subsection{\texorpdfstring{$\LA$ as a dualizing object}{L as a dualizing object}}

  The dual adjunction between the category $\Alg$ of $\LA$-algebras and the category $\LSpa$ of $\LA$-spaces falls under the framework of \emph{concrete dualities} of Porst \& Tholen~\cite{PorstTholen1991}. That is, it is induced in a precise sense by a dualizing object which lives both in a category of algebras and a category of spaces. To make sense of this, we must first note that $\Alg$ and $\LSpa$ are \emph{concrete categories}: they are equipped with a faithful functor $\abs{\cdot}$ into the category $\Set$ of sets and functions, namely the underlying set functors. Showing that $\LA$ is a dualizing object involves equipping the underlying set $L$ of the algebra $\LA$ with the structure of an $\LA$-space $\LS$ and proving that the functors $\abs{\Spec -}$ and $\abs{\Comp -} = \LSpa(-, \LS)$ as well as the underlying set functors on $\Alg$ and $\LSpa$ are representable. The theory developed in~\cite{PorstTholen1991} then allows us (after a straightforward verification of two technical conditions) to conclude that $\Comp$ and $\Spec$ form a dual adjunction.

\begin{fact} \label{fact: free on 1}
  The subalgebra $\FA$ of $\LA^{L}$ generated by the identity map $\idmap_{\LA} \in \LA^{L}$ is a free $\LA$-algebra freely generated by the element $\idmap_{\LA}$.
\end{fact}

\begin{definition}
  The $\LA$-space $\LS$ consists of the discrete set $L$ with $\Comp \LS \assign \FA$.
\end{definition}

\begin{lemma} \label{lemma: dualizing is representable}
~ 
\begin{enumerateroman}
\item $\abs{\Spec -} = \Alg(-, \LA)$.
\item $\abs{\Comp -} = \LSpa (-, \LS)$.
\end{enumerateroman}
\end{lemma}

\begin{proof}
  It suffices to verify that $\abs{\Spec \A} = \Alg(\A, \LA)$ and $\abs{\Comp \X} = \LSpa(\X, \LS)$ for each $\A \in \Alg$ and $\X \in \LSpa$, since at the level of morphisms all four functors are defined by precomposition. The first equality holds by the definition of $\Spec$. To prove the second equality, consider an $\LA$-space $\X$ and a continuous map $f\colon X \to L$. Because $\Comp \LS$ is generated by $\idmap_{\LA}$, the map $f$ reflects compatibility if and only if $f \circ \idmap_{\LA} \in \Comp \X$, i.e.\ if and only if $f \in \Comp \X$.
\end{proof}

\begin{definition}
  The $\LA$-space $\FS$ has the singleton space $\{ \Singleton \}$ as its underlying space and $\Comp \FS \assign \LA^{\{ \Singleton \}}$.
\end{definition}

\begin{lemma} \label{lemma: underlying is representable}
~ 
\begin{enumerateroman}
\item $\Alg(\FA, -) = \abs{-}$ in
 the category $\Alg$.
\item $\LSpa(\FS, -) = \abs{-}$ in the category $\LSpa$.
\end{enumerateroman}
\end{lemma}

\begin{proof}
  Again it suffices to verify that $\Alg(\FA, \A) = \abs{\A}$ and $\LSpa(\FS, \X) = \abs{\X}$ for $\A \in \Alg$ and $\X \in \LSpa$. The first equality is an immediate consequence of the fact that $\FA$ is a free $1$-generated $\LA$-algebra (Fact~\ref{fact: free on 1}). The second equality follows from the fact that each map of sets $\{ \Singleton \} \to \X$ is an $\LA$-map.
\end{proof}

  The next two facts will not be needed in what follows. Nonetheless, they are worth stating because they clarify the relations between $\FA$, $\LA$, $\LS$, and $\Spec \LA$.\footnote{The fact that $\Spec \FA \iso \LS$ and $\Comp \FS \iso \LA$ in fact follows directly from the general theory of dualities induced by dualizing objects \cite[Prop.~1.2]{PorstTholen1991}. The reason why we explicitly state Facts~\ref{fact: spec fa iso ls} and~\ref{fact: ccomp fs iso la} is to record the maps witnessing these isomorphisms.}

\begin{fact} \label{fact: spec fa iso ls}
  The map $h \mapsto h(\idmap_{\LA})$ is an isomorphism $\Spec \FA \iso \LS$.
\end{fact}

\begin{proof}
  The map $\phi\colon h \mapsto h(\idmap_{\LA})$ is bijective because $\FA$ is the free $\LA$-algebra freely generated by $\idmap_{\LA}$. It is a homeomorphism because $\LA$ is discrete and so is $\Spec \FA$, since each singleton $\{ h \} \subseteq \Spec \FA$ has the form $U_{\idmap_{\LA} \mapsto h(\idmap_{\LA})}$. The map $\phi$ reflects compatibility: if $g \in \Comp \LS = \FA$, then there is some unary term $t(x)$ such that $g = t^{(\LA^{L})}(\idmap_{\LA})$, and so 
  \[
  g \circ \phi\colon h \mapsto (t^{(\LA^{L})}(\idmap_{\LA}))(h(\idmap_{\LA})) = t^{\LA}(h(\idmap_{\LA})) = h(t^{(\LA^{L})}(\idmap_{\LA})).
  \]
  That is, $g \circ \phi = \unit_{\FA}(t^{\FA}(\idmap_{\LA})) \in \Comp \Spec \FA$. The map $\phi^{-1}$ also reflects compatibility: if $g \in \Comp \Spec \FA$, then $g = \unit_{\FA}(t^{\FA}(\idmap_{\LA}))$ for some unary term $t(x)$, and so (by the above sequence of equalities) $g = t^{(\LA^{L})}(\idmap_{\LA}) \circ \phi$. Therefore, $g \circ \phi^{-1} = t^{(\LA^{L})}(\idmap_{\LA}) \in \Comp \LS$. The map $\phi$ is therefore an isomorphism.
\end{proof}

\begin{fact} \label{fact: ccomp fs iso la}
  The map $f \mapsto f_{\Singleton}$ is an isomorphism $\Comp \FS \iso \LA$.
\end{fact}

\subsection{\texorpdfstring{Duality for $\LA$-algebras}{Duality for L-algebras}}

  It remains to describe the counit of the adjunction between $\Alg$ and $\LSpa$.

\begin{fact} \label{fact: ev}
  Consider an $\LA$-space $\X$. The \emph{evaluation map}
\begin{align*}
  \ev_{\X}\colon \X & \longrightarrow \Spec \Comp \X \\
   x & \longmapsto \pi_{x}
\end{align*}
  is a continuous $\LA$-map. It is the unique $\LA$-map ${\ev_{\X}\colon \X \to \Spec \Comp \X}$ such that, for every $f \in \Comp \X$,
\begin{align*}
  \unit_{\Comp \X}(f) \circ \ev_{\X} = f.
\end{align*}
\end{fact}

\begin{proof}
  The map $\ev_{\X}$ reflects compatibility: for every $f \in \Comp \X$,
\begin{align*}
  (\unit_{\Comp \X}(f) \circ \ev_{\X})_{x} = (\unit_{\Comp \X}(f))(\pi_{x}) = \pi_{x}(f) = f_{x},
\end{align*}
  and therefore $\unit_{\Comp \X}(f) \circ \ev_{\X} = f \in \Comp \X$. The map $\ev_{\X}$ is continuous because, for each $a \in \X$ and $W \subseteq \LA$,
\begin{align*}
  \ev_{\X}^{-1}[U_{a \mapsto W}] = \set{x \in X}{\ev_{\X}(x)(a) \in W} = \set{x \in X}{a_{x} \in W} = a^{-1} [W].
\end{align*}
  Finally, the uniqueness claim holds because, for any $\LA$-algebra $\A$, each $h \in \Spec \A$ is uniquely determined by the map $a \mapsto \unit_{\A}(a)(h)$ for $a \in \A$.
\end{proof}

\begin{remark}
  The uniqueness condition in Fact~\ref{fact: ev} states precisely that the canonical representation $\unit_{\A}$ of an $\LA$-algebra $\A$ is in fact a terminal representation of $\A$, if we define a morphism $\phi\colon \rho_{1} \to \rho_{2}$ of representations $\rho_{1}\colon \A \into \Cont(X_{1}, \LA)$ and $\rho_{2}\colon \A \into \Cont(X_{2}, \LA)$ of $\A$ as a continuous map $\phi\colon X_{1} \to X_{2}$ such that $\rho_{1}(a) = \rho_{2}(a) \circ \phi$ for each $a \in \A$.
\end{remark}

\begin{fact} \label{fact: ev injective surjective} 
  Let $\X$ be an $\LA$-space. Then:
\begin{enumerateroman}
\item $\X$ is separated if and only if $\ev_{\X}$ is injective.
\item $\X$ is full if and only if $\ev_{\X}$ is surjective.
\item $\X$ is completely $\LA$-regular if and only if the topology of $\X$ is the initial topology with respect to $\ev_{\X}\colon \X \to \Spec \Comp \X$.
\end{enumerateroman}
\end{fact}

\begin{proof}
  The first two claims are immediate, and the last one follows from the fact that $f^{-1}[\{ a \}] = \ev_{\X}^{-1}[U_{f \mapsto a}]$ for each $f \in \Comp \X$ and $a \in \LA$.
\end{proof}

\begin{fact} \label{fact: ev iso}
  Let $\X$ be an $\LA$-space. Then:
\begin{enumerateroman}
\item If $\X$ is completely $\LA$-regular, full and separated, then $\ev_{\X}$ is an isomorphism.\label{i:ev:creg}
\item If $\X$ is compact, full and separated, then $\X$ is completely $\LA$-regular.\label{i:ev:compact}
\end{enumerateroman}
\end{fact}

\begin{proof}
  (\ref{i:ev:creg}): it is straightforward to see that a continuous $\LA$-map $\phi\colon \X \to \Y$ is an isomorphism in the category of $\LA$-spaces if and only if it is a homeomorphism of the underlying spaces such that $f \in \Comp \X$ implies $f \circ \phi^{-1} \in \Comp \Y$. The map $\ev_{\X}$ is a continuous bijective $\LA$-map by Facts~\ref{fact: ev} and~\ref{fact: ev injective surjective}. Moreover, if $f \in \Comp \X$, then $\unit_{\Comp \X}(f) \circ \ev_{\X} = f$ by Fact~\ref{fact: ev}, and so $f \circ \ev_{\X}^{-1} = \unit_{\Comp X}(f) \in \Comp \Spec \Comp \X$. Finally, $\ev_{\X}^{-1}$ is continuous (making $\ev_{\X}$ a homeomorphism) because by Fact~\ref{fact: creg adjunction} each $\LA$-map whose domain is a completely regular $\LA$-space is continuous. Thus, $\ev_{\X}$ is an isomorphism of $\LA$-spaces.

  (\ref{i:ev:compact}): $\Spec \Comp \X$ is separated (Lemma~\ref{lemma: spec is full separated}) and thus Hausdorff (Lemma~\ref{lemma: properties of separated l-spaces}). If $\X$ is compact, then the proof of (\ref{i:ev:creg}) again shows that $\ev_{\X}$ is an isomorphism. Because $\Spec \Comp \X$ is completely $\LA$-regular (Lemma~\ref{lemma: spec is full separated}), and so is~$\X$.
\end{proof}

  We are now ready to prove the most general form of our duality theorem.

\begin{theorem}[Duality theorem for $\LA$-algebras] \label{thm: duality}
  The functors $\Spec$ and $\Comp$ form a dual adjunction between the category $\Alg$ of $\LA$-algebras and the category $\LSpa$ of $\LA$-spaces with unit $\unit$ and counit $\ev$. This dual adjunction restricts to a dual equivalence between $\Alg$ and the full subcategory of completely $\LA$-regular full separated $\LA$-spaces.
\end{theorem}

\begin{proof}
  The contravariant functors $\abs{\Comp -}$ and $\abs{\Spec -}$ are both representable (Lemma~\ref{lemma: dualizing is representable}), and so are the covariant underlying set functors $\abs{\cdot}$ of the concrete categories $\Alg$ and $\LSpa$ (Lemma~\ref{lemma: underlying is representable}). By \cite[Thm.~1.7 \& footnote 4]{PorstTholen1991}, $\Comp$ and $\Spec$ thus form a dual adjunction with the given unit and counit, provided that they satisfy two additional conditions.

  The first condition states that if $h\colon \A \to \Comp \X$ is a function such that the map $h_{x}\colon a \mapsto h(a)(x)\colon \A \to \LA$ is a homomorphism for every $x \in X$, then $h$ is in fact a homomorphism. But this is true for each function $h\colon \A \to \LA^{X}$.

  The second condition states that if $\phi\colon \X \to \Spec \A$ is a function such that the map $\phi_{a}\colon x \mapsto \phi(x)(a)\colon \X \to \LS$ is a continuous $\LA$-map for each $a \in \A$, then $\phi$ is in fact a continuous $\LA$-map. But this holds because for each basic open $U_{a \mapsto W} \subseteq \Spec \A$ with $a \in \A$ and $W \subseteq \LA$ we have $\phi^{-1}[U_{a \mapsto W}] = \phi_{a}^{-1}[W]$.

  We now prove the second claim of the theorem. $\Spec \A$ for $\A \in \Alg$ is indeed a completely $\LA$-regular full separated $\LA$-space by Lemma~\ref{lemma: spec is full separated}. It remains to prove that $\unit$ and $\ev$ are isomorphisms on the given categories. But $\unit_{\A}$ is always an isomorphism of $\LA$-algebras by Fact~\ref{fact: unit is iso}, and $\ev_{\X}$ is an isomorphism of $\LA$-spaces for each completely $\LA$-regular full separated $\LA$-space $\X$ by Fact~\ref{fact: ev iso}. 
\end{proof}

  The above dual equivalence has two problematic parts that make it difficult to use: complete $\LA$-regularity and fullness. In the next two sections, we shall replace these by more amenable conditions. In the next section, we trade complete $\LA$-regularity for compactness by restricting to a subclass of $\LA$-algebras, so-called canonically finitely valued $\LA$-algebras. Afterwards, to remove the fullness condition, we impose some restrictions on $\LA$, the main one being a congruence distributivity assumption.

\section{\texorpdfstring{Duality for canonically finitely valued $\LA$-algebras}{Duality for canonically finitely valued L-algebras}}
\label{sec: finitely valued}

  We now restrict the spatial side of the duality between $\LA$-algebras and completely $\LA$-regular full separated $\LA$-spaces (Theorem~\ref{thm: duality}) to compact full separated $\LA$-spaces. This requires us to introduce a corresponding subclass of $\Alg$ on the algebraic side: so-called (canonically) finitely valued $\LA$-algebras.

  Recall that a continuous $\LA$-representation of an $\LA$-algebra $\A$ is an embedding $\rho\colon \A \into \Cont(X, \LA)$ where $X$ is a topological space. This is the same thing as an isomorphism $\A \iso \Comp \X$ where $\X$ is an $\LA$-space. Accordingly, it makes sense to talk of full, separated, compact, etc.\ representations.

\subsection{\texorpdfstring{Finitely valued $\LA$-algebras}{Finitely valued L-algebras}}

\begin{definition}
  An $\LA$-valued function $f$ on a set $X$ \emph{has finite range} if its image $f[X] \subseteq \LA$ is finite. The $\LA$-valued functions of finite range on~$X$ form an algebra
\begin{align*}
  \FinRng(X, \LA) \assign \set{f \in \LA^{X}}{f \text{ has finite range}} \leq \LA^{X}.
\end{align*}
\end{definition}

\begin{definition}
  An $\LA$-representation $\A \iso \Comp \X$ of an $\LA$-algebra $\A$ is \emph{finitely valued} if $\Comp \X \leq \FinRng (X, \LA)$. The \emph{canonical $\LA$-representation} of $\A$ is the isomorphism $\unit_{\A}\colon \A \into \Comp \Spec \A$.
\end{definition}

\begin{remark}
  If either $X$ or $\LA$ is finite, then $\FinRng(X, \LA) = \LA^{X}$.
\end{remark}

\begin{fact} \label{fact:compact-finitely-valued}
  If $X$ is a compact space, then $\Cont(X, \LA) \leq \FinRng(X, \LA)$.
\end{fact}

\begin{proof}
  The image of a compact set under a continuous map is compact, which in a discrete space means finite.
\end{proof}

\begin{remark}
  Each of the finitely many values $a_{1}, \dots, a_{n} \in \LA$ attained by a continuous $\LA$-valued function $f$ on a compact space $X$ is attained on a clopen set. Such functions $f$ are therefore in bijective correspondence with pairs consisting of a decomposition of $X$ into finitely many disjoint non-empty clopens $F_{1}, \dots, F_{n} \subseteq X$ and a corresponding tuple of distinct values $a_{1}, \dots, a_{n} \in \LA$.
\end{remark}

  The constructions $\FinRng(X, \LA)$ for a set $X$ and $\Cont(X, \LA)$ for a compact space $X$ are closely related via the \emph{Stone--\v{C}ech compactification} of $X$. This is the space $\Uf X$ of all ultrafilters on $X$ with a suitable Stone topology. The space $X$ embeds into $\Uf X$ via the map $\iota_{X}\colon X \hookrightarrow \Uf X$ sending each $x \in X$ to the corresponding principal ultrafilter. The universal property of the Stone--\v{C}ech compactification then states that each function $f\colon X \to Y$ into a compact Hausdorff space $Y$ lifts to a unique continuous function $f^{\sharp}\colon \Uf X \to Y$ such that $f = f^{\sharp} \circ \iota_{X}$. Conversely, each continuous function $g\colon \Uf X \to Y$ has a restriction $g^{\flat} \assign g \circ \iota_{X}\colon X \to Y$.
  
  Since any finite subset of $\LA$ is a compact Hausdorff space, we get the following.

\begin{fact} \label{fact: finrng cont iso}
  $\FinRng(X, \LA)$ and $\Cont(\Uf X, \LA)$ are isomorphic for each set $X$ via the maps $f \mapsto f^{\sharp}$ and $g \mapsto g^{\flat}$. Each algebra $\A \leq \FinRng(X, \LA)$ is therefore isomorphic to the algebra $\alg{A}^{\sharp} \assign [\A]^{\sharp} \leq \Cont(\Uf X, \LA)$.
\end{fact}

  The above isomorphism in effect allows us to evaluate $\LA$-valued functions of finite range on $X$ at ultrafilters: for $\U \in \Uf X$ and $I \subseteq \Uf X$ the projection maps
\begin{align*}
  & \pi_{\U}\colon \Cont(\Uf X, \LA) \to \LA, & & \pi_{I}\colon \Cont(\Uf X, \LA) \to \LA^{I}
\end{align*}
  have restrictions to $X$, namely
\begin{align*}
  & \pi^{\flat}_{\U} \assign \pi_{\U} \circ \iota_{X}\colon \FinRng(X, \LA) \to \LA, & & \pi^{\flat}_{I} \assign \pi_{I} \circ \iota_{X} \colon \FinRng(X, \LA) \to \LA^{I}.
\end{align*}
  These restrictions can be described more directly as
\begin{align*}
  & \pi^{\flat}_{\U}\colon f \mapsto a \iff f^{-1} [\{ a \}] \in \U, & & \pi^{\flat}_{I}(f)\colon \U \mapsto \pi_{\U}(f) \text{ for } \U \in I.
\end{align*}
  In an abuse of notation, we shall generally write $\pi_{\U}$ and $\pi_{I}$ for $\pi^{\flat}_{\U}$ and $\pi^{\flat}_{I}$.

\begin{fact} \label{f:equivalence-for-representations}
  The following are equivalent for each $\LA$-algebra $\A$:
\begin{enumerateroman}
\item \label{i:compact-repr}
$\A$ has a compact $\LA$-representation.

\item \label{i:separating-Stone-repr}
$\A$ has a separated Stone $\LA$-representation.

\item \label{i:finitely-valued-repr}
$\A$ has a finitely valued $\LA$-representation.

\item \label{i:separating-finitely-valued-repr}
$\A$ has a separated finitely valued $\LA$-representation.
\end{enumerateroman}
\end{fact}

\begin{proof}
  The implications \eqref{i:separating-finitely-valued-repr} $\Rightarrow$ \eqref{i:finitely-valued-repr} and \eqref{i:separating-Stone-repr} $\Rightarrow$ \eqref{i:compact-repr} are trivial. The implications \eqref{i:compact-repr} $\Rightarrow$ \eqref{i:finitely-valued-repr} and \eqref{i:separating-Stone-repr} $\Rightarrow$ \eqref{i:separating-finitely-valued-repr} hold because $\Comp(X, \LA) \leq \FinRng(X, \LA)$.
  Let us prove the implication \eqref{i:finitely-valued-repr} $\Rightarrow$ \eqref{i:separating-Stone-repr}.
  By Fact~\ref{fact: finrng cont iso}, each finitely valued $\LA$-representation yields a Stone $\LA$-representation.
  By \cref{lemma: separated quotient preserves properties}, the separated quotient of this Stone $\LA$-representation is also Stone.
\end{proof}

\begin{definition}
  A \emph{finitely valued $\LA$-algebra} is an $\LA$-algebra that has a finitely valued $\LA$-representation, or, equivalently, a compact $\LA$-representation. In other words, the class of all finitely valued $\LA$-algebras can be defined as $\Algfv \assign \IOp \SOp \PfrOp(\LA)$, where
\begin{align*}
  \PfrOp(\LA) \assign \set{\FinRng(X, \LA)}{X \text{ arbitrary set}},
\end{align*}
  or, equivalently, as $\Algfv \assign \IOp \SOp \PcOp(\LA)$, where
\begin{align*}
  \PcOp(\LA) \assign \set{\Cont(X, \LA)}{X \text{ compact space}}.
\end{align*}
\end{definition}

  Finitely valued $\LA$-algebras are thus defined much like $\LA$-algebras, except that arbitrary powers of $\LA$ are replaced by compact powers.

\begin{remark} \label{rem: singleton representation}
  Each subalgebra $\A \leq \LA$ is finitely valued, as witnessed by its \emph{singleton representation} $\X$ consisting of a singleton space $1$ with $\Comp \X \assign \A^{1}$.
\end{remark}

\begin{fact} \label{f:closure-under-operators}
  The class of all finitely valued $\LA$-algebras is closed under $\IOp$, $\SOp$, and~$\PfinOp$, i.e.\ under isomorphic images, subalgebras, and finite products.
\end{fact}

\begin{proof}
  The closure under $\PfinOp$ holds because $\FinRng(X_{1}, \LA) \times \dots \times \FinRng(X_{n}, \LA) \iso \FinRng(X, \LA)$, where $X$ is the disjoint union of the sets $X_{1}, \dots, X_{n}$.
  The closures under $\IOp$ and $\SOp$ are clear.
\end{proof}

  We say that an algebra $\A$ \emph{lies locally} in a class of algebras $\class{K}$ if each finitely generated subalgebra of $\A$ lies in $\class{K}$. This is equivalent to $\A$ being a directed union of algebras in $\class{K}$, provided that $\class{K}$ is closed under subalgebras.

\begin{fact} \label{fact: locally finitely pointed}
  Each finitely valued $\LA$-algebra $\A$ lies locally in $\IOp \SOp \PfinOp (\LA)$.
\end{fact}

\begin{proof}
  Let $\A \leq \FinRng(X, \LA)$ and let $\B$ be the subalgebra of $\A$ generated by finitely many elements $f_{1}, \dots, f_{n} \in \A$. Each of the functions $f_{i}$ decomposes $X$ into finitely many equivalence classes where $f_{i}$ is constant. Let $\sim$ be the equivalence relation on $X$ such that $x \sim y$ holds if and only if $f_{i}(x) = f_{i}(y)$ for all $i \in \{ 1, \dots, n \}$. Since $\B$ is generated by the function $f_{i}$, we also have that $x \sim y$ holds if and only if $g(x) = g(y)$ for all $g \in \B$. The map
  \begin{align*}
  	\rho \colon \alg{B} & \longinto \LA^{X /{\sim}}\\
  	h & \longmapsto ([x] \mapsto h(x))
  \end{align*}
  is therefore a representation of $\B$ on a finite set, and hence $\B \in \IOp \SOp \PfinOp(\LA)$.
\end{proof}

\begin{example} \label{example: finitely valued mv}
  Finitely valued $\LA$-algebras for $\LA \assign [0, 1]$ (the standard MV-chain) will be called \emph{finitely valued MV-algebras}. Finitely valued $\LA$-algebras for $\LA \assign \PMVchain$ (the standard positive MV-chain) will be called \emph{finitely valued positive MV-algebras}. We shall see later (Fact~\ref{fact: fv = ffv}) that in both of these cases the class of finitely valued $\LA$-algebras coincides with the class of all $\LA$-algebras which are locally in $\IOp \SOp \PfinOp(\LA)$.

  Cignoli and Marra~\cite[Thm.~3.1]{CignoliMarra2012} show that the class of finitely valued MV-algebras coincides with the class of \emph{locally weakly finite} MV-algebras, which are the MV-algebras lying locally in $\IOp \SOp \PfinOp([0, 1])$. We shall prove a universal algebraic generalization of this fact later in Fact~\ref{fact: fv = ffv}. Their MV-algebraic Stone duality is formulated for the class of locally weakly finite MV-algebras.

  They also observe that locally weakly finite MV-algebras are \emph{hyper-Archimedean}: for each $a \in \A$ there is some $n \geq 1$ such that $a^{n} \assign a \odot \cdots \odot a$ ($n$ times) is idempotent, or equivalently for each $a \in \A$ there is some $n \geq 1$ such that $n a \assign a \oplus \cdots \oplus a$ ($n$~times) is idempotent. This is because $[0, 1]$ is hyper-Archimedean and the property is preserved under $\IOp$, $\SOp$, $\PfinOp$ and directed unions. The same argument shows that locally weakly finite positive MV-algebras also satisfy both of these formulations of hyper-Archimedeanicity, although in that case they are no longer equivalent. (However, note that in positive MV-algebras being idempotent with respect to $\odot$ is still equivalent to being idempotent with respect to $\oplus$.) Cignoli and Marra moreover give an example of a hyper-Archimedean MV-algebra which is not finitely valued~\cite[Example~4.4]{CignoliMarra2012}.
\end{example}

\begin{example} \label{example: locally finite mv}
  Finitely valued $\LA$-algebras for $\LA \assign [0, 1]_{\Q}$ (the rational MV-chain) are precisely the locally finite MV-algebras. This was shown by Cignoli, Dubuc and Mundici~\cite[Thm.~5.1]{CignoliDubucMundici2004} in the course of setting up a Stone duality for this class of MV-algebras. We shall again prove a universal algebraic generalization of this fact later in Fact~\ref{fact: fv = ffv}.
\end{example}

\begin{remark}
  The class of finitely valued MV-algebras is not closed under arbitrary products: the standard MV-chain $\MVchain$ is a finitely valued MV-algebra but $\MVchain^{\N}$ is not finitely valued because it is not hyper-Archimedean (consider the function $f\colon k \mapsto \frac{1}{k+1}$). Similarly, it is not closed under ultraproducts: the non-semisimple Chang algebra embeds into an ultraproduct of $[0, 1]$.

  The same holds for the class of locally finite MV-algebras: the rational chain $[0, 1]_{\Q}$ is locally finite, but $([0, 1]_{\Q})^{\N}$ is not (as witnessed by the same function). Similarly, it is not closed under ultraproducts: the finite MV-chains $\Luk_{n}$ for $n \geq 1$ are locally finite MV-algebras but $[0, 1] \in \IOp \SOp \PUOp (\set{\Luk_{n}}{n \geq 1})$ is not.
\end{remark}

\subsection{\texorpdfstring{Canonically finitely valued $\LA$-algebras}{Canonically finitely valued L-algebras}}

  The definition of finitely valued $\LA$-algebras postulates the existence of \emph{some} finitely valued representation. We have seen that, among all representations of an $\LA$-algebra, there is always a canonical one. The following definition therefore naturally suggests itself.

\begin{definition}
  An $\LA$-algebra $\A$ is \emph{canonically finitely valued} if its canonical representation is finitely valued. That is, for every $a \in \A$ the following set is finite:
\begin{align*}
  \set{h(a) \in \LA}{h \in \Alg(\A, \LA)}.
\end{align*}
\end{definition}

  The terminology \emph{fully finitely valued} instead of \emph{canonically finitely valued} would also be justified, according to the next theorem.

\begin{theorem}[Equivalent descriptions of canonically finitely valued $\LA$-algebras] \label{thm: fully finitely valued}
  The following are equivalent for each $\LA$-algebra $\A$:
\begin{enumerateroman}
\item \label{i:ffvc-cfv} $\A$ is canonically finitely valued, i.e.\ the canonical representation of $\A$ is finitely valued.
\item \label{i:ffvc-cc} The canonical representation of $\A$ is compact, i.e. $\Spec \A$ is compact.
\item \label{i:ffvc-ffv} $\A$ has a full finitely valued representation.
\item \label{i:ffvc-fcr} $\A$ has a full compact representation.
\item \label{i:ffvc-efv} Each representation of $\A$ is finitely valued.
\end{enumerateroman}
  Moreover, if $\A$ is finitely generated, then the above conditions are equivalent to:
\begin{enumerateroman}
\setcounter{enumi}{5}
\item \label{i:ffvc-fin} $\Spec \A$ is finite.
\end{enumerateroman}
\end{theorem}

\begin{proof}	
	\eqref{i:ffvc-efv} $\Rightarrow$ \eqref{i:ffvc-cfv}: trivial. 
	
	\eqref{i:ffvc-cfv} $\Rightarrow$ \eqref{i:ffvc-efv}: in each representation $\rho$ of $\A$ the range of the function $\rho(a) \in \A$ is contained in the finite set $\set{h(a) \in \LA}{h \in \Alg(\A, \LA)}$.
	
	\eqref{i:ffvc-cc} $\Rightarrow$ \eqref{i:ffvc-fcr}: the canonical representation is full by Lemma~\ref{lemma: spec is full separated}. 
	  
	\eqref{i:ffvc-fcr} $\Rightarrow$ \eqref{i:ffvc-ffv}: each compact representation is finitely valued.
	  
	\eqref{i:ffvc-ffv} $\Rightarrow$ \eqref{i:ffvc-cfv}: let $\rho\colon \A \to \Cont(X, \LA)$ be a full finitely valued representation. Then, by fullness, $\set{h(a) \in \LA}{h \in \Alg(\A, \LA)} = \set{\pi_{x}(\rho(a)) \in \LA}{x \in X}$ for each $a \in \alg{A}$, and so $\unit_{\alg{A}}(a)$ is a function of finite range on $\Alg(\alg{A}, \LA)$.
	
	\eqref{i:ffvc-cfv} $\Rightarrow$ \eqref{i:ffvc-cc}: suppose that $\set{h(a) \in \LA}{h \in \Alg(\A, \LA)}$ is finite for each $a \in \A$, and let us prove that $\Spec \A = \Alg(\A, \LA)$ is compact.
	For $a\in \A$, we denote with $\pi_a\colon \LA^A\rightarrow \LA$ the projection onto the $a$-th coordinate (which is a continuous map). 
	The set $\Alg(\A, \LA)$ is the intersection of all sets of the form
	\[
		\{h\in \LA^A \mid \pi_{\tau_{\scriptscriptstyle\A}(a_1, \dots, a_n)}(h)= \tau_{\LA}(\pi_{a_1}(h), \dots, \pi_{a_n}(h))\}
	\]
	for $\tau$ ranging over all function symbols and $a_1, \dots, a_n$ (where $n$ is the arity of $\tau$) ranging over all elements of $\LA$.
	The function $\tau_{\LA} \colon \LA^n \to \LA$ is continuous because $\LA$ has the discrete topology, and so the function from $\LA^A$ to $\LA$ that maps $h$ to $\tau_{\LA}(\pi_{a_1}(h), \dots, \pi_{a_n}(h))$ is continuous. Since $\LA$ is Hausdorff and so has a closed diagonal, the set displayed above is closed.
	Therefore, $\Alg(\A, \LA)$ is a closed subset of $\LA^A$.
	Moreover, $\Alg(\A, \LA)$ is contained in the subset
\begin{align*}
  \prod_{a \in \A} \set{h(a) \in \LA}{h \in \Alg(\A, \LA)}
\end{align*}
   of $\LA^A$, which, being a product of finite discrete and hence compact spaces, is compact by Tychonoff's theorem.
	Therefore, $\Alg(\A, \LA)$ is a closed subset of a compact space, and hence it is compact. (In fact, this proof shows that it is a Stone space because a closed subset of a product of finite discrete spaces is a Stone space.)
	  
	This proves that the conditions \eqref{i:ffvc-ffv}--\eqref{i:ffvc-efv} are equivalent.
	
	Suppose now that $\A$ is finitely generated.

	\eqref{i:ffvc-fin} $\Rightarrow$ \eqref{i:ffvc-cfv}: trivial.
	  
	\eqref{i:ffvc-cfv} $\Rightarrow$ \eqref{i:ffvc-fin}: each homomorphism $h \in \Alg(\A, \LA)$ from an algebra $\alg{A}$ generated by elements $a_{1}, \dots, a_{n} \in \LA$ is uniquely determined by the values $h(a_{1}), \dots, h(a_{n}) \in \LA$. If there are only finitely many values for each $h(a_{i})$ when $h$ ranges over all homomorphisms $h \in \Alg(\A, \LA)$, then the set $\Alg(\alg{A}, \LA)$ is finite.
\end{proof}

\begin{remark}
  The last condition in the above theorem implies the preceding conditions: every $\LA$-algebra $\A$ with a finite spectrum is canonically finitely valued.
\end{remark}

\begin{remark}
  If $\LA$ is finite, then the classes of canonically finitely valued $\LA$-algebras, finitely valued $\LA$-algebras, and all $\LA$-algebras coincide.
\end{remark}

\begin{remark}
  Recall that every $\LA$-algebra $\A$ has a full representation (namely, the canonical representation), but this representation need not be compact. On the other hand, a finitely valued $\LA$-algebra always has a compact representation by \cref{f:equivalence-for-representations}, but this representation need not be full.
\end{remark}

  It may indeed happen that an $\LA$-algebra has some finitely valued $\LA$-representation but its canonical $\LA$-representation is not finitely valued, i.e.\ for some $a \in \alg{A}$ the function $\unit_{\alg{A}}(a)$ on $\Alg(\A, \LA)$ takes infinitely many values. In particular, $\LA$ itself is always a finitely valued $\LA$-algebra but it need not be canonically finitely valued. A trivial example of this occurs when $\LA$ is an infinite algebra in the empty signature. A~more interesting example is the following.

\begin{example} \label{ex: not canonically finite valued}
  Let $\LA$ be the reduct of the standard MV-chain $\MVchain$ in the signature $\{ \oplus, \wedge, \vee, 1, 0 \}$. Then for each real number $r \geq 1$ the truncated multiplication map $\mu_{r}\colon x \mapsto \min(rx, 1)$ is an endomorphism of $\LA$. But for each $a \in (0, 1)$ the set $\set{\mu_{r}(a) \in \MVchain}{r \geq 1}$ has the cardinality of the continuum, so the canonical representation of $\LA$ is not finitely valued. Nonetheless, $\LA$ is of course a finitely valued $\LA$-algebra. The same holds for the reduct of $\MVchain$ in the signature $\{ \odot, \wedge, \vee, 0, 1 \}$, taking instead the endomorphisms $\mu_{r}\colon x \mapsto \max(1 - r(1-x), 0)$.
\end{example}

  In the next section, we shall formulate a sufficient condition which will ensure that finitely valued and canonically finitely valued $\LA$-algebras coincide. It will in particular apply to the case of finitely valued (positive) MV-algebras.

\begin{fact} \label{fact: ffv closure properties}
  The class of all canonically finitely valued $\LA$-algebras is closed under relative homomorphic images (homomorphic images in $\Alg$) and directed unions.
\end{fact}

\newcommand{\succf}{\mathop{\mathrm{succ}}}
\newcommand{\zf}{\mathop{\mathrm{z}}}

\begin{example}
  In general, the class of $\LA$-algebras is not closed under subalgebras. Consider the algebra $\LA \coloneqq \Z \cup \{\infty\}$ with the unary successor operation $\succf$, taking $\succf(\infty) \assign \infty$, and the unary operation $\zf$ which is the identity map on $\Z$ and $\zf(\infty) \assign 0$. Then the only homomorphism $h\colon \LA \to \LA$ is the identity map: from $\succf(h(\infty))= f(h(\succf(\infty))) = h(\infty)$ we deduce that $h(\infty) = \infty$ and therefore $h(0) = h(\zf(\infty)) = \zf(h(\infty)) = \zf(\infty) = 0$. But now, using $\succf^{n}$ to denote the $n$-fold iteration of the function, $h(n) = h(\succf^{n}(0)) = \succf^{n}(h(0)) = \succf^{n}(0) = n$ for each $n \geq 1$. Likewise, $h(-n) = -n$ because $0 = h(0) = h(\succf^{n}(-n)) = \succf^{n}(h(-n))$. Therefore, $\LA$ is a canonically finitely valued $\LA$-algebra. However, the subalgebra $\Z$ is not canonically finitely valued, since for every integer $k$ the map $h\colon n \mapsto n+k$ is a homomorphism $h\colon \Z \to \LA$.

	A similar example is provided by taking $\LA$ to be the expansion of the algebra from the Example~\ref{ex: not canonically finite valued} by the constant $\frac{1}{2}$. Then each homomorphism $\LA \to \LA$ is the identity map on elements of the form $\frac{1}{2^{n}}$, and hence on elements of the form $\frac{k}{2^{n}}$. Because such elements are dense in $\LA$, the only homomorphism $\LA \to \LA$ is the identity map. On the other hand, let $\A$ be the subalgebra of $\LA$ with the universe $\{ 0 \} \cup [\frac{1}{2}, 1]$. Then any order-preserving map $f\colon \{ 0 \} \cup [\frac{1}{2}, 1] \to \{ 0 \} \cup [\frac{1}{2}, 1]$ which fixes the elements $0$, $\frac{1}{2}$, $1$ is a homomorphism, since $a \oplus b = 1$ for all $a, b \in \A$, unless $a = 0$ or $b = 0$. This shows again that canonically finitely valued $\LA$-algebras need not be closed under subalgebras.
\end{example}

\begin{example}
	In general, the class of $\LA$-algebras is not closed under products either.
	Let $\LA$ be an infinite set with no function symbol.
	The empty product (i.e.\ the singleton algebra) is not canonically finitely valued.
	
	As a bonus, we show that $\LA$-algebras are not closed under binary products.
	Let $\LA$ be an algebra consisting of an infinite set equipped for each $a \in \LA$ with a unary operation that maps everything to $a$.
	Then $\LA$ is a canonically finitely valued $\LA$-algebra because the unique homomorphism from $\LA$ to $\LA$ is the identity.
	However, $\LA \times \LA$ is not canonically finitely valued: any map $h\colon \LA \times \LA \to \LA$ such that $h(a, a) = a$ for all $a \in \LA$ is a homomorphism.
\end{example}

\subsection{\texorpdfstring{Duality for canonically finitely valued $\LA$-algebras}{Duality for canonically finitely valued L-algebras}}

  The duality theorem for $\LA$-algebras (Theorem~\ref{thm: duality}) now restricts to a duality theorem for canonically finitely valued $\LA$-algebras.

\begin{theorem}[Duality theorem for canonically finitely valued $\LA$-algebras] \label{thm: duality canonically fv}
  The functors $\Spec$ and $\Comp$ form a dual equivalence between canonically finitely valued $\LA$-algebras and compact full separated $\LA$-spaces, with unit $\unit$ and counit $\ev$.
\end{theorem}

\begin{proof}
  By Theorem~\ref{thm: fully finitely valued}, if $\A$ is a canonically finitely valued $\LA$-algebra, then $\Spec \A$ is compact, and if $\X$ is a compact full $\LA$-space, then $\Comp \X$ is canonically finitely valued. If $\X$ is a compact full separated $\LA$-space, then it is completely $\LA$-regular by Fact~\ref{fact: ev iso}.(\ref{i:ev:compact}), and therefore we indeed obtain a restriction of the duality of Theorem~\ref{thm: duality}.
\end{proof}

\section{\texorpdfstring{Duality for finitely valued $\LA$-algebras: congruence distributivity}{Duality for finitely valued L-algebras: congruence distributivity}}
\label{sec: jonsson}

  The duality for canonically finitely valued $\LA$-algebras obtained in the previous section (Theorem~\ref{thm: duality canonically fv}) is still not entirely satisfactory due to the fullness condition. The aim of this section, achieved in Theorem~\ref{thm: cd duality}, is to make this condition disappear by imposing some restrictions on $\LA$.

\subsection{The J\'{o}nsson property}

  The proof of our main result involves what we call the J\'{o}nsson property, which comes in three equivalent versions. After proving their equivalence, we shall refer to them indiscriminately as ``the J\'{o}nsson property''.

  We shall say that a homomorphism $h\colon \A \to \LA$ \emph{factors through} a homomorphism $g\colon \A \to \B$ if $h = f \circ g$ for some homomorphism $f\colon g[\A] \to \LA$. (Note that the domain of~$f$ is $g[\A]$, not $\B$.) This happens if and only if $\ker g \leq \ker h$.

\begin{definition}
\label{d:J-finite-covers}
  An $\LA$-algebra $\A$ is said to have the \emph{J\'{o}nsson property for finite covers} if for each finitely valued $\LA$-representation $\X$ of $\A$ and each finite cover $X_{1}, \dots, X_{n}$ of $X$, every homomorphism $h\colon \Comp \X \to \LA$ factors through some~$\pi_{X_{i}}$. By a \emph{finite cover} of a set $X$ we mean a finite family of sets $X_{1}, \dots, X_{n}$ such that $X_{1} \cup \dots \cup X_{n} = X$. Note that we allow for the case $n = 0$ here: the empty family of sets covers the empty set.
\end{definition}

\begin{definition}
\label{d:J-compact}
  An $\LA$-algebra $\A$ has the \emph{compact J\'{o}nsson property} if for each compact $\LA$-representation $\X$ of $\A$, every homomorphism $h\colon \Comp \X \to \LA$ factors through $\pi_{x}$ for some $x \in \X$.
\end{definition}

\begin{definition}
\label{d:J-ultrafilter}
  An $\LA$-algebra $\A$ has the \emph{ultrafilter J\'{o}nsson property} if for each finitely valued $\LA$-representation $\X$ of $\A$, each homomorphism ${h\colon \Comp \X \to \LA}$ factors through $\pi_{\U}$ for some ultrafilter $\U$ on~$X$.
\end{definition}

  This last condition calls for some explanation. Recall from the discussion surrounding Fact~\ref{fact: finrng cont iso} that each finitely valued function $f\colon X \to \LA$ extends to a continuous function $f^{\sharp}\colon \Uf X \to \LA$. What we mean by $h$ factoring through some $\pi_{\U}$ is that there is some ultrafilter $\U$ on $X$ such that $h(f) = \pi_{U}(f^{\sharp})$ for all $ f\in \Comp \X$.

\begin{remark} \label{remark: empty jonsson}
  The empty $\LA$-algebra, if it exists, fails to satisfy any version of the J\'{o}nsson property: consider its representation over an empty space. If the singleton $\LA$-algebra is a subalgebra of $\LA$, then also this algebra fails to satisfy any version of the J\'{o}nsson property: consider again its representation over an empty space. However, if the singleton $\LA$-algebra is not a subalgebra of $\LA$, then it vacuously satisfies all versions of the J\'{o}nsson property, since it has no homomorphism into $\LA$.
\end{remark}

\begin{notation}
  Given functions $f, g \in \Cont(X, \LA)$, we introduce the notation
\begin{align*}
  \equalizer{f}{g} \assign \set{x \in X}{f_{x} = g_{x}}.
\end{align*}
  This set is always clopen because $\LA$ is discrete and $\equalizer{f}{g} = \pair{f}{g}^{-1}[\Delta_{\LA}]$, where $\Delta_{\LA} \subseteq \LA \times \LA$ is the (clopen) equality relation on $\LA$.
\end{notation}

\begin{theorem}[Equivalence between J\'{o}nsson properties] \label{thm: jonsson equivalence}
  The following are equivalent for each $\LA$-algebra $\A$:
\begin{enumerateroman}
\item \label{i:Jp-fv}
$\A$ has the J\'{o}nsson property for finite covers.
\item \label{i:Jp-c}
$\A$ has the compact J\'{o}nsson property.
\item\label{i:Jp-u}
$\A$ has the ultrafilter J\'{o}nsson property.
\end{enumerateroman}
\end{theorem}

\begin{proof}
  \eqref{i:Jp-fv} $\Rightarrow$ \eqref{i:Jp-u}: suppose that $\A$ has the J\'{o}nsson property for finite covers. 
  Let $\A \leq \FinRng(X, \LA)$ be a finitely valued $\LA$-representation of $\A$, and let ${h\colon \A \to \LA}$ be a homomorphism.
  Let $\I$ be the set of all $I \subseteq X$ such that $h$ does not factor through $\pi_{I}$.
  For all $I_1, \dots, I_n \in \I$ we have $I_1 \cup \dots \cup I_n \neq X$ because otherwise, by the J\'{o}nsson property for finite covers, there would be $i \in \{1, \dots, n\}$ such that $h$ factors through $\pi_{I_i}$, contradicting $I_i \in \I$.
  Therefore, the ideal generated by $\I$ is proper.
  By the Ultrafilter Lemma, $\I$ is included in a prime ideal, whose complement is an ultrafilter $\U$.
  Since $\U$ is disjoint from $\I$, using the definition of $\I$ we get that $I \in \U$ implies that $h$ factors through $\pi_{I}$.
  
  \eqref{i:Jp-u} $\Rightarrow$ \eqref{i:Jp-fv}: suppose that $\A$ has the ultrafilter J\'{o}nsson property. 
  Let $\A \leq \FinRng(X, \LA)$ be a finitely valued $\LA$-representation of $\A$, let ${h\colon \A \to \LA}$ be a homomorphism and let $X_{1}, \dots, X_{n}$ be a finite cover of $X$.
  By the ultrafilter J\'{o}nsson property, $h$ factors through $\pi_{\U}$ for some $\U \in \Uf X$. Since $X_{1} \cup \dots \cup X_{n} = X$ and $\U$ is an ultrafilter, $X_{i} \in \U$ for some $i \in \{ 1, \dots, n \}$, so $\pi_{\U}$ factors through~$\pi_{X_{i}}$, and hence $h$ factors through $\pi_{X_i}$.

  \eqref{i:Jp-u} $\Rightarrow$ \eqref{i:Jp-c}: suppose that $\A$ has the ultrafilter J\'{o}nsson property.
 Let $\A \leq \Cont(X, \LA)$ be a compact $\LA$-representation of $\A$.
 By the ultrafilter J\'{o}nsson property $h$ factors through $\pi_{\U}$ for some ultrafilter $\U$ on $X$. Because $X$ is compact, $\U$ converges to some $x \in X$. That is, $x \in U$ implies $U \in \U$ for each open $U$.
  We claim that $h$ factors through $\pi_{x}$, i.e.\ that $\ker \pi_{x} \leq \ker h$. Suppose therefore that, for $f, g \in \A$, we have $\pair{f}{g} \in \ker \pi_{x}$, i.e.\ $f_{x} = g_{x}$. Taking $a \assign f_{x} = g_{x}$, the set $U \assign f^{-1} [\{ a \}] \cap g^{-1} [\{ a \}]$ is an open neighborhood of $x$, so $U \in \U$ and thus $\pi_{\U}$ factors through $\pi_{U}$. But $\pi_{U}(f) = \pi_{U}(g)$, so $\pi_{\U}(f) = \pi_{\U}(g)$, and so $h(f) = h(g)$.

  \eqref{i:Jp-c} $\Rightarrow$ \eqref{i:Jp-u}: suppose that $\A$ has the compact J\'{o}nsson property. 
  Let $\A \leq \FinRng(X, \LA)$ be a finitely valued $\LA$-representation of $\A$, and let ${h\colon \A \to \LA}$ be a homomorphism.
  Given the isomorphism between $\FinRng(X, \LA)$ and $\Cont(\Uf X, \LA)$ (Fact~\ref{fact: finrng cont iso}), the finitely valued representation $\A \leq \FinRng(X, \LA)$ induces a compact representation $\A \leq \Cont(\Uf X, \LA)$.
  By the compact J\'{o}nsson property applied to $\A \leq \Cont(\Uf X, \LA)$, the homomorphism $h \colon \A \to \LA$ factors through the projection $\pi_{\U}$ for some $\U \in \Uf X$, and thus $h$ factors through $\pi_{\U}\colon \FinRng(X, \LA) \to \LA$.
\end{proof}

\subsection{Relative congruence distributivity and the J\'{o}nsson property}

  To show that finitely valued $\LA$-algebras enjoy the J\'{o}nsson property, we will combine two assumptions: a lack of nontrivial partial endomorphisms of $\LA$ and the relative congruence distributivity of finitely valued $\LA$-algebras.

\begin{definition}
  A \emph{partial endomorphism of $\LA$} is a homomorphism $h\colon \A \to \LA$ where $\A \leq \LA$. We say that \emph{$\LA$ has only trivial partial endomorphisms} if each partial endomorphism of $\LA$ is an inclusion.
\end{definition}

\begin{remark} \label{rem: singleton}
	If $\LA$ is nontrivial and only has trivial partial endomorphisms, then it does not have any singleton subalgebras. Indeed, if it had one, say $\{a\}$, then the constant function $\LA \to \{a\}$ would be a homomorphism. This homomorphism would be an inclusion, since $\LA$ has only trivial partial endomorphisms, but this would contradict the nontriviality of $\LA$.
	Recall also that $\LA$ has the empty algebra as a subalgebra if and only if the signature contains no constant symbols.
\end{remark}

\begin{fact} \label{fact: no endomorphisms implies fsi}
  If $\LA$ has only trivial partial endomorphisms, then the only relative congruences on any subalgebra $\A$ of $\LA$ are the identity congruence $\Delta_\A$ and the total congruence $\nabla_\A$.
\end{fact}

\begin{proof}
  Let $\A$ be a subalgebra of $\LA$.
  
  The trivial and the total congruence on an algebra in a prevariety are always relative congruences with respect to the prevariety. This settles one inclusion.
  
   Let $\theta$ be a relative congruence on $\A$.
  Since $\theta$ is a relative congruence, $\alg{A}/\theta$ is an $\LA$-algebra and hence a subdirect product of subalgebras of $\LA$.
  Therefore, there is a family $(\theta_{i})_{i \in I}$ of relative congruences on $\A$ such that $\alg{A} / \psi_{j} \in \IOp \SOp (\LA)$ and $\theta = \bigcap_{i \in I} \theta_i$.
  Since $\alg{A}$ is a subalgebra of $\LA$, the assumption that $\LA$ has only trivial partial endomorphisms yields that $\theta_{i} = \Delta_{\A}$ for each $i \in I$.
  Therefore, either $\theta = \Delta_\A$ (if $I \neq \varnothing$) or $\theta = \nabla_\A$ (if $I = \varnothing$).
\end{proof}

  The following is a useful sufficient condition that ensures that $\LA$ has only trivial partial endomorphisms.

\begin{lemma}
	Suppose that $\LA$ has a bounded lattice reduct $\langle L; \wedge, \vee, 1, 0 \rangle$ and for all $a,b \in \LA$ with $a \nleq b$ there is a unary term $t(x)$ such that $t^{\LA}(a) = 1$ and $t^{\LA}(b) = 0$. Then $\LA$ has only trivial partial endomorphisms.
\end{lemma}

\begin{proof}
	Let $\A$ be a subalgebra of $\LA$ and $h \colon \A \to \LA$ a homomorphism.
	We claim that $h(a) = a$.
	By way of contradiction, suppose this is not the case.
	Then either $h(a) \nleq a$ or $a \nleq h(a)$.
	Let us consider the case where $h(a) \nleq a$.
	Then $t^{\LA}(h(a)) = 1$ and $t^{\LA}(a) = 0$ for some unary term $t$, and so $1 = t^{\LA}(h(a)) = h(t^{\A}(a)) = h(t^{\LA}(a)) = h(0) = 0$.
	But then $h(a) \leq 1 = 0 \leq a$, which contradicts the assumption that $h(a) \nleq a$.	
	The case $a \nleq h(a)$ is analogous.
\end{proof}

\begin{example} \label{ex: pmv chain has no partial endomorphisms}
  Such term functions exist in the standard positive MV-chain $\PMVchain$, and therefore also in the standard MV-chain $\MVchain$. Consequently, $\PMVchain$ and $\MVchain$ have only trivial partial endomorphisms.
\end{example}

\begin{proof}
  Consider $a \nleq b$ in $\PMVchain$. We define a sequence $f_{i}$ of term functions on $\PMVchain$ for $i \in \N$ and use it to determine sequences $a_{i}, b_{i} \in \PMVchain$ for $i \in \N$ as $a_{i+1} \assign f_{i}(a_{i})$ and $b_{i+1} = f_{i}(b_{i})$, with $a_{0} \assign a$ and $b_{0} \assign b$. If $a_{i}, b_{i} \leq \frac{1}{2}$, we take $f_{i}(x) \assign x \oplus x$. If $a_{i}, b_{i} > \frac{1}{2}$, we take $f_{i}(x) \assign x \odot x$. If $b_{i} \leq \frac{1}{2} \leq a_{i}$, we take $f_{i}(x) \assign x \oplus x$. Observe that $a_{i} \nleq b_{i}$ for each $i \in \N$, so the above three cases are mutually exclusive and no other cases can arise. In each step either $a_{i+1} - b_{i+1} = 2 (a_{i} - b_{i})$ or $a_{i+1} = 1$ and $b_{i+1} < 1$. Since the distance can only double finitely many times while remaining within $\PMVchain$, eventually we reach some $k \in \N$ where $a_{k} = 1$ and $b_{k} < 1$. But then there is some $n$ such that $(a_{k})^{n} = 0$. Taking $g(x) \assign x^{n}$, we obtain a term function $h \assign g \circ f_{k} \circ \dots \circ f_{0}$ such that $h(a) = 1$ and $h(b) = 0$.
\end{proof}

\begin{definition}
  Given a prevariety $\class{K}$, a \emph{$\class{K}$-relative} congruence on an algebra $\A$ is a congruence $\theta$ on $\alg{A}$ such that $\alg{A} / \theta \in \class{K}$. The $\class{K}$-relative congruences on $\alg{A}$ form a complete lattice $\Con_{\class{K}} \A$, where arbitrary meets are intersections. An algebra $\alg{A}$ is \emph{$\class{K}$-relatively congruence distributive} if $\Con_{\class{K}} \A$ is a distributive lattice.
\end{definition}

\begin{lemma} \label{lemma: rcd implies jonsson}
  Let $\LA$ be a nontrivial algebra with only trivial partial endomorphisms and without the empty subalgebra. Let $\class{K}$ be a prevariety containing $\LA$. Then each $\class{K}$-relatively congruence distributive finitely valued $\LA$-algebra has the J\'{o}nsson property.
\end{lemma}

\begin{proof}
  If $\LA \in \class{K}$, then $\class{K}$ contains all $\LA$-algebras. Consider a finitely valued representation $\alg{A} \leq \FinRng(X, \LA)$, a finite cover $X_{1}, \dots, X_{n}$ of $X$, and a homomorphism $h\colon \alg{A} \to \LA$.
  
  By Fact~\ref{fact: no endomorphisms implies fsi} the image $h[\A] \leq \LA$ has at most two congruences, so $\ker h$ is either the top element of $\Con_{\class{K}} \A$ or it is a coatom. In the latter case, $\ker h$ is finitely meet irreducible, and so by the distributivity of $\Con_{\class{K}} \A$ it is finitely meet prime; therefore, the inequality $\ker \pi_{X_{1}} \wedge \dots \wedge \ker \pi_{X_{n}} = \Delta_{\alg{A}} \leq \ker h$ implies that $\ker \pi_{X_{i}} \leq \ker h$ for some ${i \in \{ 1, \dots, n \}}$, and hence $h$ factors through~$\pi_{X_{i}}$.
  Suppose now that we are in the remaining case: $\ker h$ is the top element of $\Con_{\class{K}} \A$.
  If $n \geq 1$, then $\ker \pi_{X_{1}} \leq \ker h$ and so $h$ factors through~$\pi_{X_{1}}$.
  If $n = 0$, then $X = \emptyset$ and so $\alg{A}$ is either a singleton or the empty algebra; but then $h[\alg{A}]$ would be either the empty subalgebra of $\LA$ (impossible by hypothesis) or a singleton subalgebra (impossible by Remark~\ref{rem: singleton}).
\end{proof}

\begin{remark}
  The implication from congruence distributivity to the ultrafilter J\'{o}nsson property is a special case of J\'{o}nsson's Lemma~\cite[Lemma~II.4.3]{Jonsson1995}, which states that for each finitely meet prime congruence $\theta$ on a congruence distributive algebra $\A \leq \prod_{x \in X} \A_{x}$ there is an ultrafilter $\U$ on~$X$ such that $\equalizer{f}{g} \in \U$ implies $\pair{f}{g} \in \theta$ for all $f, g \in \A$. The proof of this implication essentially mimics the proof of J\'{o}nsson's lemma.
\end{remark}

\subsection{The CD Duality Theorem}

  We now finally turn to proving our main duality result. Firstly, we observe that congruence distributivity ensures that finitely valued and canonically finitely valued $\LA$-algebras coincide, provided that $\LA$ has only trivial partial endomorphisms.

\begin{fact}\label{fact: compact implies full}
  Let $\LA$ be a nontrivial algebra with only trivial partial endomorphisms and without the empty subalgebra. Suppose that all finitely valued $\LA$-algebras are $\class{K}$-relatively congruence distributive for some prevariety $\class{K}$ containing $\LA$. Then each compact $\LA$-space is full.
\end{fact}

\begin{proof}
  By Lemma~\ref{lemma: rcd implies jonsson}, finitely valued $\LA$-algebras have the J\'{o}nsson property. 
  Let $\X$ be a compact $\LA$-space, and let $h\colon \Comp \X \to \LA$ be a homomorphism.
  By the J\'{o}nsson property, there is $x \in X$ such that $h$ factors through $\pi_{x}$.
  Since $\LA$ has only trivial partial endomorphisms, $h = \pi_{x}$.
\end{proof}

\begin{fact} \label{fact: fv = ffv}
	Let $\LA$ be a nontrivial algebra with only trivial partial endomorphisms and without the empty subalgebra. Suppose that all finitely valued $\LA$-algebras are $\class{K}$-relatively congruence distributive for some prevariety $\class{K}$ containing $\LA$.
	Then the following classes coincide:
\begin{enumerateroman}
\item\label{i:fvla} finitely valued $\LA$-algebras,
\item\label{i:cfvla} canonically finitely valued $\LA$-algebras,
\item\label{i:lwfla} $\LA$-algebras which lie locally in $\IOp \SOp \PfinOp(\LA)$,
\item\label{i:lfs} $\LA$-algebras which lie locally in the class of $\LA$-algebras with a finite spectrum.
\end{enumerateroman}
  In particular, each locally finite $\LA$-algebra is finitely valued.
\end{fact}

\begin{proof}
	\eqref{i:fvla} $\Rightarrow$ \eqref{i:cfvla}:
	Consider a finitely valued $\LA$-algebra~$\A$.
	This means that $\A$ has a compact representation.
	By \cref{fact: compact implies full}, it is full.
	Thus $\A$ has a compact full representation, and so it is canonically finitely valued by Theorem~\ref{thm: fully finitely valued}.

	\eqref{i:cfvla} $\Rightarrow$ \eqref{i:fvla}: trivial.

	\eqref{i:fvla} $\Rightarrow$ \eqref{i:lwfla}: this was proved in Fact~\ref{fact: locally finitely pointed}.

	\eqref{i:lwfla} $\Rightarrow$ \eqref{i:fvla}: the class of finitely valued $\LA$-algebras contains $\LA$ and is closed under $\IOp$, $\SOp$ and $\PfinOp$ by Fact~\ref{f:closure-under-operators}. The class of canonically finitely valued $\LA$-algebras is closed under directed unions by Fact~\ref{fact: ffv closure properties}. Given the equivalence between \eqref{i:fvla} and \eqref{i:cfvla}, it follows that each $\LA$-algebra which is a directed union of $\LA$-algebras in $\IOp \SOp \PfinOp(\LA)$ is finitely valued.

	\eqref{i:lfs} $\Rightarrow$ \eqref{i:lwfla}: each $\LA$-algebra with a finite spectrum lies in $\IOp \SOp \PfinOp(\LA)$.

	\eqref{i:lwfla} $\Rightarrow$ \eqref{i:lfs}: each $\alg{A} \in \IOp \SOp \PfinOp(\LA)$ has a representation over a finite set $X$ (which is in particular a compact representation). By Fact~\ref{fact: compact implies full} this representation is full. Thus each homomorphism $h\colon \alg{A} \to \LA$ is equal to one of the finitely many projections $\pi_{x}$ for some $x \in X$.

  Finally, if $\A$ is a locally finite $\LA$-algebra, then it lies locally in the class of finite $\LA$-algebras, which is a subclass of $\IOp \SOp \PfinOp(\LA)$.
\end{proof}

We now show that under the hypotheses on $\LA$ in \cref{fact: fv = ffv}, the property (about elements of a given $\LA$-algebra $\A$) of being represented by a function of finite range is absolute in the sense that it does not depend on the $\LA$-representation of $\A$. This contrasts with Example~\ref{ex: not canonically finite valued}, where an element was represented by a function of finite range in one representation but by a function of infinite range in another one.

\begin{fact}
	Let $\LA$ be a nontrivial algebra with only trivial partial endomorphisms and without the empty subalgebra. Suppose that all finitely valued $\LA$-algebras are $\class{K}$-relatively congruence distributive for some prevariety $\class{K}$ containing $\LA$. Then the following are equivalent for each element $a \in \A$ of an $\LA$-algebra $\A$:
\begin{enumerateroman}
\item \label{i:for-some}
$\rho(a)$ has finite range for some representation $\rho$ of $\A$.
\item \label{i:for-each}
$\rho(a)$ has finite range for each representation $\rho$ of $\A$.
\item \label{i:canonical}
$\Alg(\Sg^{\A}(a), \LA)$ is finite.
\end{enumerateroman}
  Consequently, $\A$ is finitely valued if and only if $\Alg(\Sg^{\A}(a), \LA)$ is finite for each $a \in \alg{A}$. Moreover, $\alg{A}$ has a largest finitely valued subalgebra, which consists of the elements satisfying the above equivalent conditions.
\end{fact}

\begin{proof}
  \eqref{i:for-each} $\Rightarrow$ \eqref{i:for-some}: trivial. 
  
  \eqref{i:for-some} $\Rightarrow$ \eqref{i:canonical}: if a function $\rho(a)$ with $a \in \A$ has finite range in a representation $\rho$ of $\alg{A}$, then restricting $\rho$ to $\alg{B} \assign \Sg^{\A}(a)$ yields a finitely valued representation of the principal $\LA$-algebra $\alg{B}$. But then $\alg{B}$ is finitely valued, so it is canonically finitely valued by Fact~\ref{fact: fv = ffv}, and thus $\Alg(\alg{B}, \LA)$ is finite by Theorem~\ref{thm: fully finitely valued}.\eqref{i:ffvc-fin}.

  \eqref{i:canonical} $\Rightarrow$ \eqref{i:for-each}: given a representation $\rho\colon \A \to \LA^{X}$ the range of $\rho(a)$ consists of the set of values $(\pi_{x} \circ \rho)(a)$ where $x$ ranges over $X$, which is a subset of the set of values $h(a)$ where $h$ ranges over the finite set $\Alg(\alg{B}, \LA)$, where $\alg{B} \assign \Sg^{\A}(a)$.

  The elements of $\A$ that have finite range in every representation form a subalgebra $\alg{B} \leq \alg{A}$, since, for each $n$-ary operation $f$ in the signature, if $a_{1}, \dots, a_{n} \in \A$ have finite range in a given representation of $\A$, then so does $f^{\A}(a_{1}, \dots, a_{n})$. Clearly $\alg{B}$ is finitely valued, as witnessed by the restriction of any representation of $\A$ to $\alg{B}$. Finally, if $\alg{C}$ is a finitely valued subalgebra of $\alg{A}$, then $\Alg(\Sg^{\alg{C}}(c), \LA)$ is finite for each $c \in \alg{C}$, but $\Sg^{\alg{C}}(c) = \Sg^{\alg{A}}(c)$, so $\Alg(\Sg^{\alg{A}}(c), \LA)$ is finite and $c \in \alg{B}$, and therefore $\alg{C} \leq \alg{B}$.
\end{proof}

  The following is the main duality result of the paper.
  The crucial improvement with respect to the cheaper duality in \cref{thm: duality canonically fv} is that we no longer need to include the fullness condition on $\LA$-spaces explicitly.

\begin{theorem}[CD Duality Theorem for finitely valued $\LA$-algebras] \label{thm: cd duality}
  Let $\LA$ be a nontrivial algebra with only trivial partial endomorphisms and without the empty subalgebra. If all finitely valued $\LA$-algebras are $\class{K}$-relatively congruence distributive for some prevariety $\class{K}$ containing $\LA$ (for instance, if $\VOp(\LA)$ is congruence distributive), then $\Spec$ and $\Comp$ form a dual equivalence between finitely valued $\LA$-algebras and compact separated $\LA$-spaces.
\end{theorem}

\begin{proof}
  This is the duality of Theorem~\ref{thm: duality canonically fv}, taking into account that finitely valued and canonically finitely valued $\LA$-algebras now coincide by Fact~\ref{fact: fv = ffv}, as do compact and compact full $\LA$-spaces by Fact~\ref{fact: compact implies full}.
\end{proof}

  The assumption that the empty algebra is not a subalgebra of $\LA$ can be removed, at the cost of slightly complicating the statement of the duality. Namely, the empty space with empty $\Comp \X$ needs to be excluded (but not the empty space with $\Comp \X$ a singleton). This is because $\X$ is not full if $\X$ is the empty space with $\Comp \X$ empty. On the other hand, if $\X$ is the empty space with $\Comp \X$ a singleton, then $\X$ is full if and only if $\LA$ has no singleton subalgebras.

  In addition to the above duality, which applies if all finitely valued $\LA$-algebras are $\class{K}$-relatively congruence distributive for some prevariety $\class{K}$ containing $\LA$, we can study congruences of finitely valued $\class{K}$-relatively congruence distributive $\LA$-algebras without assuming that all finitely valued $\LA$-algebras are congruence distributive. The following theorem tells us that if $\Con_{\class{K}} \alg{A}$ is distributive for $\class{K} \assign \Alg$, then there is a very concrete reason for this, namely the existence of an isomorphism between $\Con_{\class{K}} \alg{A}$ and the lattice of closed, or equivalently compact, subsets of the spectrum $\Alg(\A, \LA)$ of $\A$.

\begin{theorem}[Representation of relative congruences] \label{thm: representation of congruences}
	Let $\LA$ be a nontrivial algebra with only trivial partial endomorphisms and without the empty subalgebra. Let $\A$ be an $\Alg$-relatively congruence distributive finitely valued $\LA$-algebra.
  Given a compact separated representation of $\A$ on a space $X$, such as the canonical representation of $\A$ on its spectrum $\Alg(\A, \LA)$, the lattice $\Con_{\Alg} \A$ is anti-isomorphic to the lattice $\CS(X)$ of closed subsets of $X$ via the map
\begin{align*}
	\CS(X) & \longrightarrow \Con_{\Alg} \A\\
  	Y & \longmapsto \ker \pi_{Y}.
\end{align*}
\end{theorem}

\begin{proof}
	Let $\Theta \colon \CS(X) \to \Con_{\Alg} \A$ be the map in the statement.
  Note that, for every closed subset $Y$ of $X$, $\Theta(Y)$ is indeed a relative congruence of $\A$ because $\pi_{Y}$ is a homomorphism into an $\LA$-algebra. Clearly, for any closed subsets $Y$ and $Z$ of $X$, the inclusion $Z \subseteq Y$ implies $\Theta(Y) \leq \Theta(Z)$. 

	On the other hand, suppose $\Theta(Y) \leq \Theta(Z)$ and let us prove $Z \subseteq Y$. 
	Let $z \in Z$, and let us prove $z \in Y$.
	Since $\Theta(Y) \leq \Theta(Z) \leq \ker \pi_{z}$, by the Homomorphism Theorem there is a unique homomorphism $h\colon \pi_Y[\A] \to \LA$ making the following diagram commute:
	\[
	\begin{tikzcd}
		\A \arrow[swap, two heads]{d}{\pi_Y} \arrow{r}{\pi_z}& \LA\\
		\pi_Y[\A] \arrow[dashed,swap]{ru}{h}
	\end{tikzcd}
	\]
	
  Now consider the representation $\pi_Y[\A] \leq \Cont(Y, \LA)$, i.e.\ the restriction of the representation $\A \leq \Cont(X, \LA)$ to~$Y$. This is a compact separated representation because $\A \leq \Cont(X, \LA)$ is a separated representation and $Y$ is compact (due to being a closed subset of a compact space). Under our assumptions, each compact representation is full (Fact~\ref{fact: compact implies full}) and so $h = \pi_{y}$ for some $y \in Y$. 
  Therefore, for every $f \in \A$ we have $\pi_{z}(f) = h(\pi_Y(f)) = \pi_y(f)$, i.e.\ $f_z = f_y$.
  Since $\A$ is separating, $z = y$, and so $z \in Y$.
\end{proof}

\section{\texorpdfstring{Duality for $\LA$-constrained spaces: near unanimity}{Duality for L-constrained spaces: near unanimity}} \label{sec: nu}

  The main duality result of the paper (the CD Duality Theorem~\ref{thm: cd duality}) was proved in the previous section. Nonetheless, the reader may well feel some dissatisfaction with this theorem. After all, it is not at all clear how to obtain something like the Priestley duality for bounded distributive lattices or the Cignoli--Marra duality for finitely valued MV-algebras as special cases of this duality.

  To this end, the remainder of the paper is devoted to providing a categorical isomorphism between the category of $\LA$-spaces that featured in the duality theorem and a category of more tangible spaces that we call $\LA$-constrained spaces. These consist of a topological space $X$ together with some relational structure. They specialize directly to Priestley spaces and the spaces of Cignoli and Marra.

  The benefit of this division of labor is that for some results, such as the representation of relative congruences (Theorem~\ref{thm: representation of congruences}), the more concrete category of spaces is simply not needed. This allows us to prove them under weaker hypotheses, but more importantly using cleaner proofs. A second benefit is that it naturally leads one to adopt the perspective that dual spaces of algebras consist of two structures which are quite alike, namely topological and $\LA$-topological structure. Contrast this with Priestley spaces, which combine two superficially quite different kinds of structure, namely topological structure and order structure. This perspective immediately makes certain notions, such as complete $\LA$-regularity, more visible.

  What will enable us to represent $\LA$-spaces more concretely is the Baker--Pixley property: given a separated compact representation of an $\LA$-algebra $\A$ on a space $X$, we can detect whether a function $f \in \Cont(X, \LA)$ belongs to $\alg{A}$ just by looking at $f$ locally, meaning by looking at its restrictions to subsets of cardinality at most $k$ for some fixed finite $k$.
  Like the J\'{o}nsson property, it again comes in three equivalent forms: the finite, the compact, and the ultrafilter Baker--Pixley property (Theorem~\ref{thm: bp equivalence}). A convenient sufficient condition which ensures that $\LA$ has the Baker--Pixley property is the presence of a so-called near unanimity term.

\begin{remark}
  In this section, instead of talking about an $\LA$-space $\X$ and its algebra of compatible functions $\Comp \X$, it will be more convenient to talk about a topological space $X$ and an algebra $\A \leq \Cont(X, \LA)$. When we talk about e.g.\ a separated compact representation $\A \leq \Cont(X, \LA)$, what we mean is that the inclusion map $\A \into \Cont(X, \LA)$ is a separated compact representation, which in turn means that $\X \assign \langle X, \A \rangle$ is a separated compact $\LA$-space.
\end{remark}

\subsection{The Baker--Pixley property}

\begin{notation}
  Given a natural number $k$, we use the following notation:
\begin{align*}
  I \subseteq_{k} X \iff I \subseteq X \text{ and } I \text{ has cardinality at most } k.
\end{align*}
\end{notation}

\begin{definition}
	Let $X$ be a set, $\A$ a subalgebra of $\LA^{X}$, and $f \in \LA^X$ a function. We say that $f$ is \emph{$k$-interpolated} by $\A$ for $k \geq 1$ if $\pi_{I}(f) \in \pi_{I}[\A]$ for all $I \subseteq_{k} X$, i.e.\ if for each $I \subseteq_{k} X$ there is some $g \in \alg{A}$ such that $\restrict{f}{I} = \restrict{g}{I}$.
\end{definition}

\begin{definition}
\label{d:BP-finite}
	An algebra $\LA$ has the \emph{finite $k$-ary Baker--Pixley property} for $k \geq 1$ if for each finite separated representation $\A \leq \LA^{X}$ and each function $f \in \LA^{X}$
\begin{align*}
  f \in \A \iff f \text{ is $k$-interpolated by } \A.
\end{align*}
\end{definition}

\begin{definition}
\label{d:BP-compact}
	An algebra $\LA$ has the \emph{compact $k$-ary Baker--Pixley property} for $k \geq 1$ if for each compact separated representation $\A \leq \Cont(X, \LA)$ and each function $f \in \Cont(X, \LA)$
\begin{align*}
  f \in \A \iff f \text{ is $k$-interpolated by } \A.
\end{align*}
\end{definition}

\begin{definition}\label{d:ultrafilter-interpolated}
  A function $f \in \FinRng(X, \LA)$ is \emph{ultrafilter $k$-interpolated} for $k \geq 1$ by an algebra $\A \leq \FinRng(X, \LA)$ if for all $I \subseteq_{k} \Uf X$ we have $\pi_{I}(f) \in \pi_{I}[\A]$.
\end{definition}

\begin{definition}
\label{d:BP-ultrafilter}
  An algebra $\LA$ has the \emph{ultrafilter $k$-ary Baker--Pixley property} for $k \geq 1$ if for each finitely valued separated representation $\A \leq \FinRng(X, \LA)$ and each function $f \in \FinRng(X, \LA)$
\begin{align*}
  f \in \A \iff &f \text{ is ultrafilter $k$-interpolated by } \A.
\end{align*}
\end{definition}

  The left-to-right implications in \cref{d:BP-finite,d:BP-compact,d:BP-ultrafilter} above hold trivially.
  
\begin{remark}
	If an algebra has the $k$-ary Baker--Pixley property (in any of the above forms), then it has the $j$-ary Baker--Pixley property for each $j \geq k$.
\end{remark}

\begin{remark}\label{r:no-need-for-separation}
  In case $k \geq 2$, the restriction to representations that are separated can be removed from each of the above formulations of the Baker--Pixley property, with the resulting conditions remaining equivalent to the original ones.

  To see this, consider a general representation $\A \leq \LA^{X}$ corresponding to an $\LA$-space $\X$ and suppose that a function $f \in \LA^{X}$ is $k$-interpolated by $\A$. Recall the definition of the equivalence relation $\theta_{\X}$ on $X$ from Definition~\ref{def: separated quotient} and the discussion of separated quotients there. The functions $g / \theta_{\X} \in \LA^{X / \theta_{\X}}$ for $g \in \A$ are all well-defined continuous functions. The same holds for $f$ because it is $2$-interpolated by $\A$: if $\pair{x}{y} \in \theta_{\X}$ for $x, y \in X$, then $f_{x} = f_{y}$. The following argument therefore goes through: because $f$ is $k$-interpolated by $\A$, it follows that $f / \theta_{\X} \in \LA^{X / \theta_{\X}}$ is $k$-interpolated by $\set{g / \theta_{\X} \in \LA^{X / \theta_{\X}}}{g \in \A}$, so by the Baker--Pixley property $f / \theta_{\X} = g / \theta_{\X}$ for some $g \in \A$, and therefore $f \in \A$.
\end{remark}

\begin{remark}
  In case $k = 1$, the restriction to representations that are separated cannot be simply removed (consider the diagonal subalgebra of $\LA^{2}$). However, we can replace it with a restriction to functions $f \in \LA^{X}$ which \emph{separate at most as much as~$\A$}: if $\pair{x}{y} \in \theta_{\X}$ (that is, if $g_x = g_y$ for all $g \in \A$), then $f_{x} = f_{y}$. This restriction will make the argument in the previous remark work. Conversely, if a representation $\A \leq \LA^{X}$ is separated, then each $f \in \LA^{X}$ separates at most as much as~$\A$.
\end{remark}

  We now show that these properties are equivalent.

\begin{theorem}[Equivalence between Baker--Pixley properties] \label{thm: bp equivalence}
  The following are equivalent for $k \geq 1$:
\begin{enumerateroman}
\item \label{i:finite-BP}
$\LA$ has the finite $k$-ary Baker--Pixley property.
\item \label{i:compact-BP}
$\LA$ has the compact $k$-ary Baker--Pixley property.
\item \label{i:ultrafilter-BP}
$\LA$ has the ultrafilter $k$-ary Baker--Pixley property.
\end{enumerateroman}
\end{theorem}

\begin{proof}
	We write the proof for $k \geq 2$; at the end of the proof we will mention how to treat also the case $k = 1$.

	\eqref{i:finite-BP} $\Rightarrow$ \eqref{i:compact-BP}: consider a compact representation $\A \leq \Cont(X, \LA)$ and a function $f \in \Cont(X, \LA)$ that is $k$-interpolated by $\alg{A}$. We show that $f \in \A$.
	
	Observe that for all continuous functions $g,h \colon X \to \LA$ the set
\begin{align*}
  \set{\pair{x}{y} \in X \times X}{g(x) = a, \, h(y) = b}
\end{align*}
   is clopen.
   Therefore for each $g \in \A$ the set $\equalizer{f}{g} \subseteq X$ is clopen in~$X$ and hence $\equalizer{f}{g}^{k}$ is clopen in~$X^{k}$. Moreover, the family $\set{\equalizer{f}{g}^{k}}{g \in \A}$ is a cover of~$X^{k}$, since $f$ is $k$-interpolated by~$\alg{A}$. Because $X$ is compact, so is $X^{k}$, and hence this open cover has a finite subcover. That is, there are functions $g_{1}, \dots, g_{n} \in \A$ such that for all $x_{1}, \dots, x_{k} \in X$ there is some $j \in \{1, \dots, n\}$ with $\langle x_{1}, \dots, x_{l} \rangle \in \equalizer{f}{g_{j}}^{k}$, which means that $f(x_i) = g_j(x_i)$ for all $i \in \{1, \dots, k\}$.
   Without loss of generality, we can suppose $n \geq 1$, because, since $f$ is interpolated by $\A$, there is $h \in \A$ such that $\restrict{f}{\varnothing} = \restrict{h}{\varnothing}$ (note that this equality is void). Therefore, we can throw any such $h$ in the set $\{g_1, \dots, g_n\}$, if the latter is empty.

  Let $\sim$ be the equivalence relation on $X$ such that
\begin{align*}
  x \sim y \iff f(x) = f(y) \text{ and } g_{j}(x) = g_{j}(x) \text{ for all }j \in \{ 1, \dots, n \}.
\end{align*}
(It could be shown that the condition ``$f(x) = f(y)$'' is redundant.)
  This equivalence relation has finitely many equivalence classes, since each of the finitely many functions $f$ and $g_{j}$ has finite range.

  Let $\alg{C}$ be the subalgebra of $\alg{A}$ generated by $\{g_{1}, \dots, g_n\} \subseteq \A$.
  We now show that $f \in \alg{C}$, from which it will follow that ${f \in \A}$.
  By definition of $\alg{C}$ and $\sim$, we have that $x \sim y$ implies $g(x) = g(y)$ for all $g \in \alg{C}$.
  The following is therefore a well-defined finitely valued representation of $\alg{C}$ on $X / {\sim}$:
  \begin{align*}
  	\rho \colon \alg{C} & \longinto \LA^{X /{\sim}}\\
  	g & \longmapsto ([x] \mapsto g(x)).
  \end{align*}

	We prove that the function $\rho(f)$ is $k$-interpolated by $\rho[\alg{C}]$: for all $x_1, \dots, x_k \in X$ there is some $g_{j}$ such that, for all $i \in \{1, \dots, k\}$, $f(x_i) = g_j(x_i)$, and so $(\rho(f))([x_i]) = f(x_i) = g_j(x_i) = (\rho(g_j))([x_i])$. Moreover, since $n \geq 1$, we have $\restrict{f}{\varnothing} = \restrict{g_1}{\varnothing}$.
	This proves that $\rho(f)$ is $k$-interpolated by $\rho[\alg{C}]$.

	Since $X /{\sim}$ is finite, the finite $k$-ary Baker--Pixley property implies that $\rho(f) \in \rho[\alg{C}]$.
	Since $\rho$ is injective, this implies $f \in \alg{C}$. From $f \in \alg{C} \leq \A$ we deduce $f \in \A$.

  \eqref{i:compact-BP} $\Rightarrow$ \eqref{i:ultrafilter-BP}: consider a representation $\A \leq \FinRng(X, \LA)$ and a function $f \in \FinRng(X, \LA)$ that is ultrafilter $k$-interpolated by $\alg{A}$.
  Recall from \cref{fact: finrng cont iso} that $\A$ has a canonical representation $\alg{A}^{\sharp} \leq \Cont(\Uf X, \LA)$ inside $\Cont(\Uf X, \LA)$.

	We claim that the function $f^\sharp$ is $k$-interpolated on $\Uf X$ by $\A^{\sharp}$.
	Consider $\U_1, \dots, \U_l \in \Uf X$ with $l \leq k$.
	Since $f \in \FinRng(X, \LA)$ is ultrafilter $k$-interpolated by $\A$, there is some $g \in \A$ such that $f^\sharp(\U_i) = g^\sharp(\U_j)$ for all $i \in \{1, \dots, l\}$.	
  This proves our claim that $f^{\sharp}$ is $k$-interpolated on $\Uf X$ by $\A^{\sharp}$.
  By the compact $k$-ary Baker--Pixley property, $f^\sharp \in \A^\sharp$, and thus $f \in \alg{A}$.

  \eqref{i:ultrafilter-BP} $\Rightarrow$ \eqref{i:finite-BP}: if $X$ is finite, then $\LA^{X} = \FinRng(X, \LA)$. Moreover, each ultrafilter on $X$ is then principal, and so being ultrafilter $k$-interpolated on $X$ by an algebra~$\A$ is equivalent to being $k$-interpolated on $X$ by $\alg{A}$.

  This concludes the proof of the theorem. For the interested reader, we also include a direct proof of the implication \eqref{i:ultrafilter-BP} $\Rightarrow$ \eqref{i:compact-BP}: consider a compact representation $\A \leq \Cont(X, \LA)$ and a function $f \in \Cont(X, \LA)$ that is $k$-interpolated by $\alg{A}$. We shall prove $f \in \A$. By the ultrafilter $k$-ary Baker--Pixley property, it suffices to show that $f$ is ultrafilter $k$-interpolated by $\alg{A}$. Consider therefore a set of ultrafilters $I = \{ \U_{1}, \dots, \U_{l} \} \subseteq_{k} \Uf X$. Each ultrafilter $\U_{i}$ on $X$ converges to some $x_{i} \in X$. That is, $x_{i} \in U$ implies $U \in \U_{i}$ for each open $U$. Let $J \assign \{ x_{1}, \dots, x_{l} \}$. Because $f$ is $k$-interpolated by $\alg{A}$, there is some $g \in \A$ such that $\pi_{J}(f) = \pi_{J}(g)$. We claim that $\pi_{I}(f) = \pi_{I}(g)$. To this end, it suffices to prove that $\pi_{\U_{i}}(f) = \pi_{\U_{i}}(g)$ for all $i \in \{ 1, \dots, l \}$, i.e.\ that $\equalizer{f}{g} \in \U_{i}$. But this holds because $x_{i} \in \equalizer{f}{g}$, the set $\equalizer{f}{g}$ is open, and $\U_{i}$ converges to~$x_{i}$.
  
 	Finally, as promised, we mention how the above proof can be adapted to cover also the case $k = 1$.
	All changes are quite straightforward except for those pertaining the proof of \eqref{i:finite-BP} $\Rightarrow$ \eqref{i:compact-BP}.
  In this case, in which $f$ is assumed to separate at most as much as~$\A$, one uses compactness of $\{\pair{x}{y} \in X \times X \mid f(x) \neq f(y)\}$ to show the existence of a finite list of elements $h_1, \dots, h_p$ such that, for all $x, y \in X$ with $f(x) \neq f(y)$, there is $t \in \{1, \dots, p\}$ such that $h_t(x) \neq h_t(y)$.
  Then in the definition of $x \sim y$ one also imposes that for all $t \in \{1, \dots, p\}$ we have $h_t(x) = h_t(y)$, and one includes $h_1, \dots, h_p$ among the generators of $\alg{C}$ in the definition of $\alg{C}$.
  In this way one ensures that $\rho(f)$ separates at most as much as $\rho[\alg{C}]$. (These changes are harmless for $k \geq 2$.)
\end{proof}

\subsection{Near unanimity terms and the Baker--Pixley property}

  For $k \geq 2$ the $k$-ary Baker--Pixley property can be witnessed syntactically by a so-called near unanimity term.

\begin{definition} \label{def: NU term}
An $n$-ary \emph{near unanimity term} on an algebra $\A$ for $n \geq 3$ is a term $\maj(x_{1}, \dots, x_{n})$ such that for all $a, a_{1}, \dots, a_{n} \in \A$
\begin{align*}
	\text{if $a_{i} = a$ for all but at most one $i \in \{ 1, \dots, n \}$, then $\maj^{\A}(a_{1}, \dots, a_{n}) = a$.}
\end{align*}
\end{definition}

\begin{remark}
	If $\maj(x_{1}, \dots, x_{n})$ is an $n$-ary near unanimity term, then for every $q \geq n$ the term $\maj'(x_1, \dots, x_q) \coloneqq \maj(x_{1}, \dots, x_{n})$ is a $q$-ary near unanimity term.
\end{remark}

  The simplest, i.e.\ ternary, case of a near unanimity term is called a \emph{majority term}. By definition, this is a ternary term $\maj(x_{1}, x_{2}, x_{3})$ such that
\begin{align*}
	\maj^{\A}(a, a, b) = \maj^{\A}(a, b, a) = \maj^{\A}(b, a, a) = a.
\end{align*}
  Each algebra with a lattice reduct has a majority term, namely
\begin{align*}
	\maj(x_{1}, x_{2}, x_{3}) \assign (x_{1} \wedge x_{2}) \vee (x_{2} \wedge x_{3}) \vee (x_{3} \wedge x_{1}).
\end{align*}

  Huhn~\cite[Satz~2.1]{Huhn1983} showed that an algebra $\LA$ has a $(k+1)$-ary near unanimity term if and only if the variety $\VOp(\LA)$ generated by $\LA$ has the so-called $k$-ary Chinese remainder property. The Chinese remainder property can more generally be defined relative to a prevariety $\class{K}$ if we replace arbitrary congruences by $\class{K}$-congruences.

\begin{definition}
  An algebra $\A$ has the \emph{$k$-ary Chinese remainder property} (\emph{relative} to a prevariety $\class{K}$) if, given $a_{1}, \dots, a_{n} \in \A$ and $\theta_{1}, \dots, \theta_{n} \in \Con \A$ ($\theta_{1}, \dots, \theta_{n} \in \Con_{\class{K}} \A$), the system of $n$ congruence equations $x \equiv a_{i} \mod{\theta_{i}}$, i.e.\ $\pair{x}{a_{i}} \in \theta_{i}$, has a solution $x \in \alg{A}$ whenever each subsystem of at most $k$ congruence equations has a solution. A prevariety $\class{K}$ has the \emph{(relative) $k$-ary Chinese remainder property} if each algebra in $\class{K}$ has the $k$-ary Chinese remainder property (relative to $\class{K}$).
\end{definition}

  More precisely, Huhn~\cite{Huhn1983} showed that each algebra with a $(k+1)$-ary near unanimity term has the $k$-ary Chinese remainder property, and conversely that if the free algebra $\alg{F}(k+1)$ over $k+1$ generators in a variety has the Chinese remainder property, then the variety has a $(k+1)$-ary near unanimity term. Because free algebras in $\VOp(\LA)$ in fact lie in $\Alg \assign \IOp \SOp \POp(\LA)$ and Huhn's proof only applies the Chinese remainder property to relative congruences of $\alg{F}(k+1)$ with respect to $\Alg$, namely the kernels of certain homomorphisms $\alg{F}(k+1) \to \alg{F}(2)$, we can modify Huhn's theorem as~follows.

\begin{theorem}[\cite{Huhn1983}]
  An algebra $\LA$ has a $(k+1)$-ary near unanimity term if and only if $\Alg \assign \IOp \SOp \POp(\LA)$ has the $k$-ary relative Chinese remainder property.
\end{theorem}

  Baker \& Pixley~\cite[Thm.~2.1, (2)$\Leftrightarrow$(3)]{BakerPixley1975} then observed that for each variety the Chinese remainder property is equivalent to the finite $k$-ary Baker--Pixley property.\footnote{Although the paper~\cite{Huhn1983} was published later than~\cite{BakerPixley1975}, Baker \& Pixley reference a preprint of Huhn containing the relevant result.} Here by the finite $k$-ary Baker--Pixley property for a \emph{class of algebras} $\class{K}$ (as opposed to an algebra~$\LA$) for $k \geq 2$ we mean the following condition: given a finite family of algebras $(\A_{x})_{x \in X}$ in $\class{K}$ and an algebra $\alg{A} \leq \prod_{x \in X} \A_{x}$, if $f \in \prod_{x \in X} \A_{x}$ is $k$-interpolated by $\alg{A}$, then $f \in \A$. This reduces to our previous definition of the Baker--Pixley property in case $\class{K} \assign \{ \LA \}$. 

  The proof of Baker \& Pixley can again be relativized to prevarieties: each prevariety has the relative Chinese remainder property if and only if it has the finite $k$-ary Baker--Pixley property. Combined with the previous theorem, this yields the following equivalence.

\begin{theorem} \label{thm: bp and nu}
  For any $k \geq 2$, $\Alg$ has the finite $k$-ary Baker--Pixley property if and only if $\LA$ has a $(k+1)$-ary near unanimity term. In particular, if $\LA$ has a $(k+1)$-ary near unanimity term, then $\LA$ has the finite $k$-ary Baker--Pixley property.
\end{theorem}

  Mitschke~\cite{Mitschke1978} observed that from each near unanimity term on $\A$ one can construct J\'{o}nsson terms witnessing the congruence distributivity of $\VOp(\A)$. An alternative semantic proof due to Fried~\cite[Lemma~1.2.12]{KaarliPixley2001} establishes the congruence distributivity of $\VOp(\alg{A})$ directly.

\begin{theorem}[\cite{Mitschke1978}] \label{thm: NU implies CD}
	If $\LA$ has a $(k+1)$-ary near unanimity term for some $k \geq 2$, then $\VOp(\A)$ is congruence distributive.
\end{theorem}

  We have seen that, for any $k \geq 2$, we have a convenient way to prove the $k$-ary Baker--Pixley property for $\LA$: namely, by exhibiting a ($k +1)$-ary near unanimity term on $\LA$.
What about the case of $k = 1$?

\begin{definition}
  A \emph{subdiagonal} of $\LA \times \LA$ is the diagonal algebra $\DiagAlg_{\C} \leq \C \times \C$ of some $\C \leq \LA$. A \emph{product subalgebra} of $\LA \times \LA$ is an algebra of the form $\C_{1} \times \C_{2}$ for some $\C_{1}, \C_{2} \leq \LA$. If each subalgebra of $\LA \times \LA$ has one of these two forms, we say that $\LA \times \LA$ \emph{only has subdiagonal or product subalgebras}.
\end{definition}

\begin{fact} \label{f:unary-in-terms-of-binary}
	The algebra $\LA$ has the unary Baker--Pixley property if and only if it has the binary Baker--Pixley property (for example, due to having a majority term) and moreover $\LA \times \LA$ only has subdiagonal or product subalgebras. 
\end{fact}

\begin{proof}
	Suppose that $\LA$ has the unary Baker--Pixley property.
	Then it clearly has the binary Baker--Pixley property.
	Moreover, let $\A$ be a subalgebra of $\LA \times \LA$, and let $\C_1$ and $\C_2$ be the images of the projections of $\A$ onto the two coordinates.
	If there are $\pair{a}{b} \in \A$ such that $a \neq b$, then each $\pair{c}{d} \in \LA \times \LA$ separates at most as much as $\A$, and so $\pair{c}{d} \in \A$ if and only if $c \in \C_1$ and $d \in \C_2$, i.e.\ $\A = \C_1 \times \C_2$.
	On the other hand, if $a = b$ for all $\pair{a}{b} \in \A$, then $\A$ is a subdiagonal.

	For the converse direction, suppose that $\LA$ has the binary Baker--Pixley property and that each subalgebra of $\LA \times \LA$ is either a subdiagonal or a product subalgebra.
	Let us prove that $\LA$ has the unary Baker--Pixley property.
	Let $\A$ be a subalgebra of $\LA^X$ for some finite set $X$, and let $f \in \LA^X$ be such that $f$ is $1$-interpolated by $\A$ and separates at most as much as $\A$.
	To prove that $f \in \A$, it is enough to prove that $f$ is $2$-interpolated by $\A$.
	Consider $x \neq y$ in $X$.
	Let $\B$, $\C_1$ and $\C_2$ be the images of $\A$ under the projections
	\begin{align*}
		\pi_{x,y} \colon \LA^{X} & \longrightarrow \LA \times \LA & \pi_{x} \colon \LA^{X} & \longrightarrow \LA & \pi_{y} \colon \LA^{X} & \longrightarrow \LA\\
		g & \longmapsto \pair{g_x}{g_y} & g & \longmapsto g_x & g & \longmapsto g_y,
	\end{align*}
	respectively.
	Since $f$ is $1$-interpolated by $\A$, $f_x \in \C_1$ and $f_y \in \C_2$.
	If there are $\pair{a}{b} \in \B$ with $a \neq b$, then $\B$ is not a subdiagonal and so it is a product subalgebra, namely $\B = \C_1 \times \C_2$; then, $\pair{f_x}{f_y} \in \C_1 \times \C_2 = \B$, and so $f$ coincides on $\{x,y\}$ with some function in $\A$.
	Otherwise, for all $\pair{a}{b} \in \B$ we have $a = b$; therefore, since $f$ separates at most as much as $\A$, $f_x = f_y$; since $f$ is $1$-interpolated by $\A$, there is $g \in \A$ such that $f_x = g_x$, and then also $f_y = f_x = g_x = g_y$.
	This proves that $f$ is $2$-interpolated by $\A$.
\end{proof}

\begin{example}
  The standard MV-chain $\MVchain$ has the unary Baker--Pixley property, as does the two-element Boolean algebra $\Btwo$. In contrast, the standard positive MV-chain $\PMVchain$ and the two-element bounded distributive lattice $\Dtwo$ lack this property, as witnessed by the subalgebra $\{ \pair{0}{0}, \pair{0}{1}, \pair{1}{1} \}$, which is neither a subdiagonal nor a product subalgebra.
\end{example}

 In case $\LA$ is \emph{finite} and has a majority term, the unary Baker--Pixley property is equivalent to the \emph{semi\-primality} of $\LA$, a well-known universal algebraic property formulated in the realm of finite algebras~\cite[Section~3.4]{KaarliPixley2001}. Among its many equivalent characterizations, a prominent one is that each function that preserves the subalgebras of $\LA^{2}$ is a term function of $\LA$. However, the equivalences between the different characterizations of semiprimal algebras substantially rely on their finiteness. To avoid confusion as to which of these equivalent characterizations we have in mind, we therefore do not use the term ``semiprimal'' here.

\subsection{\texorpdfstring{$\LA$-constrained spaces}{L-constrained spaces}}
\label{subsec: constrained spaces}

  We now work towards the second main result of the paper (the Baker--Pixley Representation Theorem~\ref{thm: bp representation}), which is a categorical isomorphism between the category of compact separated $\LA$-spaces and a category of more tangible spaces called $k$-ary $\LA$-constrained spaces. We introduce $k$-ary $\LA$-constrained spaces in this subsection and prove the Baker--Pixley Representation Theorem in the next.

\begin{definition} \label{def: constraint}  \label{def: l-set}
	A \emph{constraint} on a set $X$ is a subalgebra of $\LA^{I}$ for some finite set $I \subseteq X$. A \emph{family of $k$-ary constraints} on $X$ consists of a constraint $\alg{A}_{I} \leq \LA^I$ for each $I \subseteq_{k} X$. Such a family is \emph{subdirect} if, for all $J \subseteq I \subseteq_{k} X$,
	\begin{align*}
		\A_{J} = \pi_{J}[\A_{I}],
	\end{align*}
	or, equivalently, if for each $g \in \LA^{J}$
	\begin{align*}
		g \in \A_J \iff g = \restrict{f}{J} \text{ for some } f \in \A_{I}.
	\end{align*}
\end{definition}

\begin{remark} \label{rem: constraints}
  If the set $X$ has cardinality at most $k$, then each subdirect family of $k$-ary constraints on $X$ can be reconstructed from the constraint $\A_{X}$. Otherwise, the family can be reconstructed from constraints $\A_{I}$ with $I$ of cardinality $k$.
\end{remark}

  There are several equivalent ways of presenting subdirect families of $k$-ary constraints on a non-empty set~$X$, and we shall switch between these according to what is most convenient for the purpose at hand:
\begin{enumerateroman}
\item as a family of algebras $\A_{I} \leq \LA^{I}$ for $I \subseteq_{k} X$,
\item as a family of algebras $\A_{\tuple{x}} \leq \LA^{k}$ for $\tuple{x} \in X^{k}$,
\item\label{i:family-of-relations} as a family of functions $\chi_{\A}^{\tuple{a}}\colon X^{k} \to \{ 0, 1 \}$ for $\tuple{a} \in \LA^{k}$,
\item as a function $\chi_{\A}\colon X^{k} \times \LA^{k} \to \{ 0, 1 \}$.
\end{enumerateroman}
  Notice that condition \eqref{i:family-of-relations} makes it clear that a subdirect family of $k$-ary constraints can be thought of as a family of $k$-ary predicates on $X$ indexed by $\tuple{a} \in \LA^{k}$.

  To describe $\LA$-valued functions on finite subsets of $X$ we use the notation
\begin{align*}
  & (x_{1} \mapsto a_{1}, \dots, x_{n} \mapsto a_{n}) & & \text{for } x_{1}, \dots, x_{n} \in X \text{ and } a_{1}, \dots, a_{n} \in \LA.
\end{align*}
  This denotes the unique function $f\colon \{ x_{1}, \dots, x_{n} \} \to \LA$, if it exists, such that $f\colon x_{i} \mapsto a_{i}$ for all $i \in \{ 1, \dots, n \}$. The function $(x_{1} \mapsto a_{1}, \dots, x_{n} \mapsto a_{n})$ exists if and only if $x_{i} = x_{j}$ implies $a_{i} = a_{j}$ for all $i, j \in \{ 1, \dots, n \}$.

  The following equivalence translates between the first two presentations:
\begin{align*}
  \langle a_{1}, \dots, a_{k} \rangle \in \A_{x_{1}, \dots, x_{k}} \iff (x_{1} \mapsto a_{1}, \dots, x_{k} \mapsto a_{k}) \in \A_{\{ x_{1}, \dots, x_{k} \}},
\end{align*}
  where the right-hand side means that the function exists and moreover belongs to $\A_{\{ x_{1}, \dots, x_{k} \}}$. Every family of algebras $\alg{A}_{\tuple{x}} \leq \LA^{k}$ for $\tuple{x} \in X^{k}$ can in turn be represented by its \emph{characteristic function}
  \begin{align*}
  \chi_{\A}\colon X^{k} \times \LA^{k} & \longrightarrow \{ 0, 1 \}\\
  (\tuple{x}, \tuple{a}) & \longmapsto 
  \begin{cases}
  1 & \text{if } \tuple{a} \in \A_{\tuple{x}},\\
  0 & \text{otherwise}.\\
  \end{cases}
\end{align*}
  This function can be decomposed into a family of functions $\chi_{\A}^{\tuple{a}}\colon X^{k} \to \{ 0, 1 \}$ indexed by $\tuple{a} \in \LA^{k}$, namely $\chi_{\A}^{\tuple{a}}\colon \tuple{x} \mapsto \chi_{\A}(\tuple{x}, \tuple{a})$. This turns a set equipped with a family of constraints into a structure in the sense of model theory for a relational signature that contains one $k$-ary predicate for each $k$-tuple $\tuple{a} \in \LA^{k}$.

\begin{remark} \label{r:empty}
  The above equivalence between different presentations of a subdirect family of constraints was stated for the case of non-empty $X$. What happens if $X$ is empty? This depends on whether the signature of $\LA$ contains a constant. A subdirect family $(\alg{A}_{I})_{I \subseteq_{k} X}$ of constraints on $X \assign \emptyset$ consists of a single constraint $\alg{A}_{\emptyset}$, which is either the empty $\LA$-algebra (if it exists) or the singleton $\LA$-algebra. If the empty $\LA$-algebra does not exist, the above equivalence works also for $X = \emptyset$. However, it fails if the empty $\LA$-algebra exists: the unique $\LA$-valued function on $\emptyset$ will count as compatible with the constraint if $\alg{A}_{\emptyset}$ is the singleton $\LA$-algebra but not if $\alg{A}_{\emptyset}$ is the empty $\LA$-algebra. In contrast, the corresponding family of $\LA$-algebras $\alg{A}_{\tuple{x}} \leq \LA^{k}$ for ${\tuple{x} \in X^{k}}$ is always the empty family if $X = \emptyset$.
\end{remark}

\begin{definition}
  A subdirect family of $k$-ary constraints on a topological space $X$ is \emph{continuous} if:
\begin{enumerateroman}
	\item $\A_{I} \leq \Cont(I, \LA)$ for each $I \subseteq_{k} X$, and
	\item \label{i:open} $X_{\tuple{a}} \assign \set{\tuple{x} \in X^{k}}{\tuple{a} \in \A_{\tuple{x}}}$ is open for each $\tuple{a} \in \LA^k$.
\end{enumerateroman}
\end{definition}

\begin{remark} \label{r:remark-on-continuity}
  Recall that the condition \eqref{i:open} above is equivalent to the continuity of $\chi_{\A}\colon X^{k} \times \LA^{k} \to \{ 0, 1 \}$ with respect to the Sierpi\'{n}ski topology on $\{ 0, 1 \}$ with the opens $\emptyset, \{ 1 \}, \{ 0, 1 \}$.
  Notice that the condition \eqref{i:open} extends to all $l \leq k$: in each continuous family of constraints on $X$, for any $l \leq k$ the set $\set{\tuple{x} \in X^{l}}{\tuple{a} \in \A_{\tuple{x}}}$ is open for each $\tuple{a} \in \LA^l$. This is because the set for~$l \geq 1$ is the preimage of the set for~$k$ under the continuous map $\Delta\colon \langle x_{1}, \dots, x_{l} \rangle \to \langle x_{1}, \dots, x_{l}, x_{1}, \dots, x_{1} \rangle$, and for $l = 0$ the set is open because $X^{0}$ is a discrete space.
\end{remark}

  The definition of a continuous subdirect family of constraints can also be stated in terms of the so-called \emph{Scott topology} on the subalgebras of a given algebra, as in the work for MV-algebras in \cite{CignoliMarra2012}. This is the topology on the set $\Sub(\alg{A})$ of subalgebras of an algebra $\alg{A}$ generated by the sets of the form
\begin{align*}
	\{\alg{B} \in \Sub(\A)\mid \C \leq \B\}
\end{align*}
for $\C$ ranging over all finitely generated subalgebras.\footnote{This is an instance of the more general notion of the Scott topology of a poset with directed suprema, applied to the poset $\Sub(\alg{A})$ ordered by inclusion. We do not need the general definition of the Scott topology here, but we refer the interested reader to~\cite[Sec.~II]{GierzHofmannEtAl2003} for the notion of the Scott topology on a directed complete partially ordered set, to~\cite[Cor.~II-1.15]{GierzHofmannEtAl2003} for the fact that the Scott topology on any algebraic domain (such as $\Sub(\alg{A})$) is generated by the principal upsets of compact elements, and to~\cite{BurrisSankappanavar1981} for the fact that the compact elements of $\Sub(\alg{A})$ are precisely the finitely generated subalgebras.}

\begin{fact} \label{fact: scott topology}
	The following conditions are equivalent for a subdirect family of constraints $(\alg{A}_{I})_{I \subseteq_{k} X}$ on a topological space $X$:
	\begin{enumerateroman}
		\item \label{i:continuous-contraints}
		$(\alg{A}_{I})_{I \subseteq_{k} X}$ is continuous as a family of constraints.
		\item \label{i:continuous-Scott}
		The map $X^k \to \Sub(\LA^k)$ that maps $\tuple{x}$ to $\A_{\tuple{x}}$ is continuous with respect to the Scott topology of $ \Sub(\LA^k)$.
	\end{enumerateroman}
\end{fact}

\begin{proof}
	Since it suffices to check the continuity of a function on a basis, \eqref{i:continuous-Scott} holds if and only if the following set is open for each finite $S \subseteq \LA^k$, with $\alg{B}$ denoting the subalgebra generated by $S$:
\begin{align*}
		\{\tuple{x} \in X^k \mid \alg{B} \subseteq \A_{\tuple{x}}\} = \{\tuple{x} \in X^k \mid S \subseteq \A_{\tuple{x}}\} = \bigcap_{\tuple{a} \in S} \{\tuple{x} \in X^k  \mid \tuple{a} \in \A_{\tuple{x}}\}.
\end{align*}
	Since the intersection is finite, this is equivalent to the condition that, for all $\tuple{a} \in \LA^k$, the set $\set{\tuple{x} \in X^{k}}{\tuple{a} \in \A_{\tuple{x}}}$ is open, i.e.\ \eqref{i:continuous-contraints}.
\end{proof}

\begin{definition} \label{d:constrained} \label{def: l-map}
  A \emph{$k$-ary $\LA$-constrained space} for $k \geq 2$ is a pair
\begin{align*}
  \X \assign \langle X, (\A_{I})_{I \subseteq_{k} X} \rangle
\end{align*}
  which consists of a topological space $X$ (called the \emph{underlying space} of $\X$) and a continuous subdirect family of constraints $(\A_{I})_{I \subseteq_{k} X}$ on $X$.
\end{definition}

\begin{notation}
  Whenever we write about $\LA$-constrained spaces $\X$ and $\Y$, we assume that $\X \assign \langle X, (\alg{A}_{I})_{I \subseteq_{k} X} \rangle$ and $\Y \assign \langle Y, (\alg{B}_{J})_{J \subseteq_{k} Y} \rangle$, i.e.\ we use the letter $\A$ for constraints on $\X$ and the letter $\alg{B}$ for constraints on $\Y$. \end{notation}

\begin{definition}
	A $k$-ary $\LA$-constrained space $\X$ for $k \geq 2$ is \emph{separated} if for all $x \neq y$ in $X$ there is some $\pair{a}{b} \in \A_{x, y}$ such that $a \neq b$.
\end{definition}

\begin{definition}
  A \emph{(continuous) $\LA$-constrained map} $\phi\colon \X \to \Y$ between $k$-ary $\LA$-constrained spaces $\X$ and $\Y$ is a (continuous) map $\phi\colon X \to Y$ such that for each $I \subseteq_{k} X$
\begin{align*}
  g \in \alg{B}_{\phi[I]} \implies g \circ \restrict{\phi}{I} \in \A_{I}.
\end{align*}
\end{definition}

  The category of $\LA$-constrained spaces and continuous $\LA$-constrained maps will be denoted by $\CLSpa$.

  In case $\LA \times \LA$ has only subdiagonal and product subalgebras, binary $\LA$-constrained spaces can be presented as what we call unary $\LA$-constrained spaces. This step is not strictly speaking necessary: the reader can simply ignore all definitions pertaining to unary $\LA$-constrained spaces and think of them as being exactly the same thing as binary $\LA$-constrained spaces in case $\LA$ only has subdiagonal and product subalgebras. However, unary $\LA$-spaces have an advantage when it comes to describing concrete $\LA$-constrained spaces: picturing a unary family of algebras plus possibly an equivalence relation is easier than picturing a binary family of algebras.

\begin{definition}
  A \emph{unary $\LA$-constrained space} is a triple
\begin{align*}
  \X \assign \langle X, (\A_{I})_{I \subseteq_{1} X}, \approx_{\X} \rangle
\end{align*}
  that consists~of a topological space $X$, a continuous subdirect family of unary constraints $(\A_{I})_{I \subseteq_{1} X}$ on $X$ and an equivalence relation $\approx_{\X}$ on $X$ closed as a subset of $X^{2}$ such that $x \not\approx_{\X} y$ implies that there are $a \in \A_{x}$ and $b \in \A_{y}$ with $a \neq b$.
\end{definition}

  Note that the condition ``$x \not\approx_{\X} y$ implies that there are $a \in \A_{x}$ and $b \in \A_{y}$ with $a \neq b$'' is automatically satisfied when $\LA$ has at least two distinct constant terms.

\begin{definition}
  A unary $\LA$-constrained space $\X$ is \emph{separated} if $\approx_{\X}$ is the equality relation on~$X$.
\end{definition}

\begin{definition}
  A \emph{(continuous) $\LA$-constrained map} $\phi\colon \X \to \Y$ between unary $\LA$-constrained spaces $\X$ and $\Y$ is a (continuous) map $\phi\colon X \to Y$ such that for each $I \subseteq_{1} X$
\begin{align*}
  g \in \alg{B}_{\phi[I]} \implies g \circ \restrict{\phi}{I} \in \A_{I},
\end{align*}
  and moreover for each $x, y \in \X$
\begin{align*}
  x \approx_{\X} y \implies \phi(x) \approx_{\Y} \phi(y).
\end{align*}
\end{definition}

\begin{remark} \label{l:approx-is-closed}
  One could make the definition of $k$-ary $\LA$-constrained spaces uniform for all $k$. To do so, one would equip a $k$-ary $\LA$-constrained space $\X$ for $k \geq 2$ with the equivalence relation $\approx_{\X}$ defined by $x \approx_{\X} y$ if and only if for all $f \in \A_{\{x,y\}}$ we have $f_x = f_y$.
Note that $\approx_{\X}$ is indeed a closed relation: $x \not\approx_{\X} y$ for $x, y \in \X$ if and only if there are distinct $a, b \in \LA$ such that $\pair{a}{b} \in \A_{x, y}$, which means that $\approx_{\X}$ is the complement of the open subset $\bigcup \set{X_{a, b}}{a, b \in \LA \text{ with } a \neq b}$ of $X^{2}$.
\end{remark}

\begin{fact} \label{fact: binary equals unary}
  Suppose that $\LA \times \LA$ has only subdiagonal and product subalgebras. Then the category of (separated) unary $\LA$-constrained spaces and unary $\LA$-constrained maps is isomorphic to the category of (separated) binary $\LA$-constrained spaces and binary $\LA$-constrained maps, taking in one direction
\begin{align*}
  \A_{x, y} \assign \begin{cases} \A_{x} \times \A_{y} \text{ if } x \not\approx y, \\ \DiagAlg_{\A_{x}} \text{ if } x \approx y, \end{cases}
\end{align*}
  and in the other direction
\begin{align*}
  x \approx y \iff \A_{x} = \A_{y} \text{ and } \A_{x, y} = \DiagAlg_{\A_{x}}.
\end{align*}
\end{fact}

\begin{proof}
  This construction applied to a binary $\LA$-constrained space $\X$ indeed yields a unary $\LA$-constrained space. It was already observed in~\cref{l:approx-is-closed} that $\approx_{\X}$ is closed. Moreover, $x \not\approx_{\X} y$ implies that $\alg{A}_{x, y}$ is not a subdiagonal of $\LA^{2}$, hence there is some $\pair{a}{b} \in \A_{x, y}$ with $a \neq b$. But then $a \in \A_{x}$ and $b \in \A_{y}$.

  Conversely, this construction applied to a unary $\LA$-constrained space $\X$ yields a binary $\LA$-constrained space. The subdirectness of the family of constraints $\A_{x, y}$ is immediate from its definition. If $x \approx y$ for $x, y \in \X$, then $\A_{\{ x, y \}} \leq \LA^{\{ x, y \}}$ is the diagonal subalgebra, so indeed $\A_{\{ x, y\}} \leq \Cont(\{ x, y \}, \LA)$. If $x \not\approx y$, then the subspace $\{ x, y \}$ is discrete (since $\{\pair{z}{z'} \in X^2 \mid z \not\approx z'\}$ is open), and therefore again $\A_{\{ x, y \}} \leq \Cont(\{ x, y \}, \LA)$. Finally, to prove that $\A_{x, y}$ is a continuous family of constraints, it remains to show that, for all $a,b \in \LA$, the subset $X_{\langle a, b \rangle} = \set{\langle x, y \rangle \in X^{2}}{\langle a, b \rangle \in \A_{x, y}}$ of $X^2$ is open. Let $a, b \in \LA$. If $a \neq b$, we have $\langle x, y \rangle \in X_{\langle a, b \rangle}$ if and only if $x \not\approx y$ and $a \in \A_{x}$ and $b \in \A_{y}$. On the other hand, if $a = b$ we have $\langle x, y \rangle \in X_{\langle a, b \rangle}$ if and only if $a \in \A_{x}$ and $b \in \A_{y}$. Because $\approx$ is closed in $X^2$ and $X_{a}$ and $X_{b}$ are open, in both cases $X_{\langle a, b \rangle}$ is open. 

  The fact that the above constructions are inverse, i.e.\ that applying one after the other to a unary or a binary $\LA$-constrained space $\X$ again yields $\X$, is immediate (using the fact that $\LA \times \LA$ has only subdiagonal and product subalgebras), as is the fact that these constructions preserve the property of being separated.

  It only remains to observe that the continuous $\LA$-constrained maps between the unary $\LA$-constrained spaces are the same functions as the continuous $\LA$-constrained maps between the corresponding binary $\LA$-constrained spaces. Clearly if $\phi$ is an $\LA$-constrained as a map between binary spaces, then it is $\LA$-constrained as a map between unary spaces, since the construction does not change the (at most) unary constraints. Conversely, suppose that $\phi\colon \X \to \Y$ is $\LA$-constrained as a map between unary spaces. We show that it is $\LA$-constrained as a map between binary spaces. To prove this, consider $\{ x, y \} \subseteq X$ and $g \in \B_{\{ \phi(x), \phi(y) \}}$. Then $(g \circ \phi)_{x} \in \A_{x}$ and $(g \circ \phi)_{y} \in \A_{y}$ because $\phi$ is $\LA$-constrained as a map between unary spaces. If $x \not\approx y$, it follows that $\langle (g \circ \phi)_{x}, (g \circ \phi)_{y} \rangle \in \A_{x, y} = \A_{x} \times \A_{y}$. On the other hand, if $x \approx y$, then $\phi(x) \approx \phi(y)$, so $g_{\phi(x)} = g_{\phi(y)}$ and again $\langle (g \circ \phi)_{x}, (g \circ \phi)_{y} \rangle \in \A_{x, y} = \DiagAlg_{\A_{x}}$.
\end{proof}

\begin{lemma} \hfill
	\begin{enumerateroman}
		\item \label{i:consep-is-Haus} 
		Each separated $k$-ary $\LA$-constrained space is Hausdorff.
		\item \label{i:consep-of-finite-is-discrete}
		Each finite separated $k$-ary $\LA$-constrained space is discrete.
	\end{enumerateroman}
\end{lemma}

\begin{proof}
	The relation $\approx_{\X}$ is closed (also in the case $k \geq 2$, by Remark~\ref{l:approx-is-closed}), and in the separated case $\approx_{\X}$ is the equality relation; this proves \eqref{i:consep-is-Haus}. Claim~\eqref{i:consep-of-finite-is-discrete} is an immediate consequence of \eqref{i:consep-is-Haus}.
\end{proof}

\begin{example} \label{ex: L-sets for L=2}
  Consider the case of $\LA \assign \Dtwo$ (the two-element bounded distributive lattice $0 < 1$). Then $\LA^{2}$ has exactly four subalgebras: the full algebra $\TotAlg \assign \Dtwo^{2}$ and its three proper subalgebras:
\begin{align*}
  & \DiagAlg \assign \{ \pair{0}{0}, \pair{1}{1} \}, & & \LeftAlg \assign \{ \pair{0}{0}, \pair{0}{1}, \pair{1}{1} \}, & & \RightAlg \assign \{ \pair{0}{0}, \pair{1}{0}, \pair{1}{1} \}.
\end{align*}
  Subdirect families of constraints on a set $X$ are precisely the families $(\A_{x, y})_{x, y \in X}$ that satisfy the conditions
\begin{align*}
  & \A_{x, x} = \DiagAlg, & & \A_{x, y} = \LeftAlg \iff \A_{y, x} = \RightAlg, & & \A_{x, y} = \TotAlg \iff \A_{y, x} = \TotAlg.
\end{align*}
  The first condition holds because $\Dtwo$ has no proper subalgebras and thus $\A_{x} = \Dtwo$ for each $x \in X$. The other two conditions hold because
\begin{align*}
  \pair{a}{b} \in \A_{x, y} \iff \pair{b}{a} \in \A_{y, x}.
\end{align*}
  Together, these conditions tell us that subdirect families of constraints for $\LA \assign \Dtwo$ correspond precisely to reflexive binary relations $\leqq$ on $X$ via the equivalence
\begin{align*}
  x \leqq y \iff \pair{1}{0} \notin \A_{x, y}.
\end{align*}
  More explicitly,
\begin{enumerateroman}
\item $\A_{x, y} = \DiagAlg$ if and only if $x \leqq y$ and $y \leqq x$,
\item $\A_{x, y} = \LeftAlg$ if and only if $x \leqq y$ and $y \nleqq x$,
\item $\A_{x, y} = \RightAlg$ if and only if $x \nleqq y$ and $y \leqq x$,
\item $\A_{x, y} = \TotAlg$ if and only if $x \nleqq y$ and $y \nleqq x$.
\end{enumerateroman}
  The continuous subdirect families of constraints on a topological space $X$ then correspond precisely to the closed reflexive binary relations $\leqq$ on $X$, i.e.\ the closed subsets of $X \times X$ containing the diagonal. Being separated corresponds to anti\-symmetry: if $x \leqq y$ and $y \leqq x$, then $x = y$.
\end{example}

\subsection{\texorpdfstring{Relating $\LA$-spaces and $\LA$-constrained spaces}{Relating L-spaces and L-constrained spaces}}

  Each $\LA$-constrained space $\X$ naturally induces an $\LA$-space which consists of the continuous functions on $X$ compatible with the constraints of $\X$.

\begin{definition}
\label{def: compatible functions}
	Consider a ($k$-ary) $\LA$-constrained space $\X$ and a subset $J \subseteq X$, for $k \geq 2$. An $\LA$-valued function $f \colon {J \to \LA}$ is \emph{compatible with a constraint} $\A_{I}$ for $I \subseteq_{k} J$ if ${\restrict{f}{I} \in \A_{I}}$. It is a \emph{compatible local function} if it is compatible with $\alg{A}_{I}$ for each $I \subseteq_{k} J$. In case $I = X$, we call it a \emph{compatible global function on $\X$}, or simply a \emph{compatible function on $\X$}.
\end{definition}

\begin{remark}
  In the framework of unary $\LA$-constrained spaces, we further require that a compatible local function $f$ satisfies $f_x = f_y$ whenever $x \approx y$.
\end{remark}

\begin{remark}
	For each $I \subseteq_{k} X$
	\begin{align*}
			f\colon I \to \LA \text{ is a compatible local function} \iff f \in \A_{I}.
		\end{align*}
\end{remark}

\begin{example}
  In the case of $\LA \assign \Dtwo$, a function $f\colon J \to \LA$ is compatible if and only if $x \leqq y$ implies $f_{x} \leq f_{y}$ for each $x, y \in J$, i.e.\ if it is order-preserving.
\end{example}

\begin{definition} \label{def: ccomp}
  Given an $\LA$-constrained space $\X$, the continuous compatible $\LA$-valued functions on $\X$ form an algebra $\CComp \X \leq \Cont(X, \LA)$. Endowing the space $X$ with this algebra yields the $\LA$-space $\Func \X \assign \langle X, \CComp \X \rangle$.
\end{definition}

  The assignment $\X \mapsto \CComp \X$ extends to a functor $\CComp \CLSpa^{\op} \to \Alg$ that to each $\LA$-constrained map $\phi\colon \X \to Y$ assigns the homomorphism
\begin{align*}
  \CComp \phi\colon \CComp \Y & \longto \CComp \X \\
  g & \longmapsto g \circ \phi.
\end{align*}
  As with $\LA$-maps, the definition of an $\LA$-constrained map is devised precisely so that $\CComp \phi$ is well-defined. It follows that the assignment $\X \mapsto \Func \X$ also extends to a functor $\Func\colon \CLSpa \to \LSpa$ that takes $\Func \phi \assign \phi$.

  Conversely, each $\LA$-space $\X$ induces an $\LA$-constrained space with the global extension property $\Cons \X$ that consists of the space $X$ and the constraints
\begin{align*}
  & \A_{I} \assign \pi_{I}[\Comp \X] \text{ for } I \subseteq_{k} X.
\end{align*}
  That is, $\A_{I} \leq \LA^{I}$ consists precisely of the continuous functions ${g\colon I \to \LA}$ that extend to some $f \in \Comp \X$.
  if $k = 1$, we further equip $\Cons \X$ with the equivalence relation defined by $x \approx y$ if and only if for all $f \in \Comp \X$ $f_x = f_y$. Moreover, each $\LA$-map $\phi\colon \X \to \Y$ is an $\LA$-constrained map $\Cons \phi \assign \phi\colon \Cons \X \to \Cons \Y$. This yields a functor $\Cons\colon \LSpa \to \CLSpa$.

  The characteristic property of $\LA$-constrained spaces induced in this way from some $\LA$-space is the following.

\begin{definition}
  For $k \geq 2$, a $k$-ary $\LA$-constrained space $\X$ enjoys the \emph{global extension property} if $\A_{I} = \pi_{I}[\CComp \X]$ for each $I \subseteq_{k} X$, i.e.\ for each $g \in \A_{I}$ there is $f \in \CComp \X$ such that $g = \restrict{f}{I}$.
\end{definition}

\begin{remark}
  In the framework of unary $\LA$-constrained spaces, the global extension property further requires that, if $x \not\approx_{\X} y$ for $x, y \in \X$, then there is some $f \in \CComp \X$ such that $f_x \neq f_y$.
\end{remark}

  Observe that each function $f \in \CComp \X$ is determined by specifying its finite range $a_{1}, \dots, a_{n} \in \LA$ and a corresponding tuple of clopens $X_{i} \assign f^{-1}[ \{ a_{i} \}]$, which constitute a clopen decomposition of $\X$. Therefore, postulating the existence of a global extension $f$ is equivalent to postulating the existence of a finite tuple of values in $\LA$ and clopen subsets of $X$ satisfying suitable constraints. In particular, we shall see that the global extension property corresponds to the Priestley separation property in case $\LA \assign \Dtwo$.

\begin{remark} \label{r:simple-if-global}
	The $k$-ary $\LA$-constrained spaces with the global extension property for $k \geq 2$ are precisely the topological spaces $X$ with a family of constraints $\alg{A}_{I} \leq \LA^{I}$ for $I \subseteq_{k} X$ such that each $g \in \A_{I}$ extends to some continuous function $f$ that is compatible with all constraints. The other required properties of $\alg{A}_{I}$ in an $\LA$-constrained space follow from this: such a family is clearly subdirect and the set $X_{\tuple{a}} \assign \set{\tuple{x} \in X^{k}}{\tuple{a} \in \A_{\tuple{x}}}$ is then open for all $\tuple{a} = {\langle a_{1}, \dots, a_{k} \rangle \in \LA^{k}}$. To see this, consider some $\tuple{x} = \langle x_{1}, \dots, x_{k} \rangle \in X^{k}$ in this set and let $I \assign \{ x_{1}, \dots, x_{k} \}$. By the global extension property there is some continuous $f \in \Comp \X$ such that $\restrict{f}{I}\colon x_{1} \mapsto a_{1}, \dots, x_{k} \mapsto a_{k}$. Thus $f^{-1}[\{a_{1}\}] \times \dots \times f^{-1}[\{a_{k}\}]$ is an open neighborhood of $\tuple{x}$ in $X^{k}$ contained in~$X_{\tuple{a}}$.
\end{remark}

  Assuming the Baker--Pixley property, the global extension property precisely characterizes the compact $\LA$-constrained spaces which arise from an $\LA$-space.

\begin{definition}
  Compact separated $k$-ary $\LA$-constrained spaces with the global extension property will be called \emph{$k$-ary $\LA$-Priestley spaces}.
\end{definition}

\begin{theorem}[Baker--Pixley Representation Theorem] \label{thm: bp representation}
  Suppose that $\LA$ has the $k$-ary Baker--Pixley property for ${k \geq 1}$.
  Then the functors $\Cons$ and $\Func$ form an isomorphism between the categories of compact $\LA$-spaces and of compact $k$-ary $\LA$-constrained spaces with the global extension property, which restricts to an isomorphism between the categories of compact separated $\LA$-spaces and of $k$-ary $\LA$-Priestley spaces.
\end{theorem}

\begin{proof}
  The global extension property for an $\LA$-constrained space $\X$ ensures that $\Cons \Func \X = \X$. Conversely, the (compact) $k$-ary Baker--Pixley property on a compact $\LA$-space $\Y$ ensures that $\Func \Cons \Y = \Y$. Clearly $\Cons \Y$ is separated as an $\LA$-constrained space if $\Y$ is separated as an $\LA$-space, and conversely $\Func \X$ is separated as an $\LA$-space if $\X$ has the global extension property and is separated as an $\LA$-constrained space.

  It only remains to show that $\CLSpa(\Cons \X, \Cons \Y) = \LSpa(\X, \Y)$ for compact $\LA$-spaces $\X$ and $\Y$. Each $\LA$-map $\phi\colon \X \to \Y$ is clearly an $\LA$-constrained map $\phi\colon \Cons \X \to \Cons \Y$. Conversely, consider a continuous $\LA$-constrained map $\phi\colon \Cons \X \to \Cons Y$. To prove that it is an $\LA$-map, consider $g \in \Comp \Y$. Then for each $I \subseteq_{k} X$ we have $\restrict{g}{\phi[I]} \circ \restrict{\phi}{I} = \restrict{f}{I}$ for some $f \in \Comp \X$. But $\restrict{g}{\phi[I]} \circ \restrict{\phi}{I} = \restrict{(g \circ \phi)}{I}$, and so the continuous $\LA$-valued function $g \circ \phi$ on $\X$ is $k$-interpolated by $\Comp \X$.
  (Note also that, if $k = 1$, then for all $x_1, x_2 \in X$ with $(g\circ \phi)(x_1) \neq (g \circ \phi)(x_2)$ we have $\phi(x_1) \not\approx_\Y \phi(x_2)$, and therefore $x_1 \not\approx_\X x_2$, and hence $h(x_1) \neq h(x_2)$ for some $h \in \Comp \X$.
  Therefore, $g \circ \phi$ separates at most as much as $\Comp \X$.)
  The Baker--Pixley property now implies that $g \circ \phi \in \Comp \X$.
\end{proof}

\begin{example} \label{example: priestley}
  In the case of $\LA \assign \Dtwo$ and $k \assign 2$, the global extension property states that ${\leqq}$ is a pre-order and for each $x \nleqq y$ there is a clopen upset $U$ with $x \in U$ and $y \notin U$. This is precisely the Priestley separation property. Recalling that separation corresponds to the antisymmetry of ${\leqq}$, the compact separated $\LA$-spaces are therefore precisely Priestley spaces.
\end{example}

\subsection{The NU Duality Theorem}

Let us now summarize the situation so far. 

	Suppose that $\LA$ is a nontrivial algebra with only trivial partial endomorphisms and without the empty subalgebra. If finitely valued $\LA$-algebras are relatively congruence distributive with respect to some prevariety containing $\LA$, then the CD Duality Theorem (Theorem~\ref{thm: cd duality}) provides the following categorical duality:
\begin{gather*}
  \text{finitely valued $\LA$-algebras} \\ \mathrel{\cong^{\op}} \\ \text{compact separated $\LA$-spaces}.
\end{gather*}
	
	 The Baker--Pixley Representation Theorem (Theorem~\ref{thm: bp representation}) in turn yields the following categorical equivalence for any algebra $\LA$ with the $k$-ary Baker--Pixley property for $k \geq 1$:
	\begin{gather*}
	  \text{compact separated $\LA$-spaces} \\ \cong \\ \text{$\LA$-Priestley spaces} \\ \text{( = compact separated $k$-ary $\LA$-constrained spaces} \\ \text{with the global extension property).}
	\end{gather*}

  Finally, for $k \geq 2$, the presence of a near unanimity term of arity $k+1$ on an algebra $\LA$ ensures that $\HOp\SOp\POp(\LA)$ is a congruence distributive variety (Theorem~\ref{thm: NU implies CD}) and that $\LA$ has the $k$-ary Baker--Pixley property (Theorem~\ref{thm: bp and nu}). Of course, the variety $\HOp \SOp \POp(\LA)$ contains all finitely valued $\LA$-algebras.

  Putting all of this together now yields the following theorem. 

\begin{theorem}[NU Duality Theorem for finitely valued $\LA$-algebras] \label{thm: nu duality}
  Let $\LA$ be a nontrivial algebra with only trivial partial endomorphisms and without the empty subalgebra. If $\LA$ has a $(k+1)$-ary near unanimity term for $k \geq 2$, then there is a dual equivalence between the category of finitely valued $\LA$-algebras and the category of $k$-ary $\LA$-Priestley spaces (that is, compact separated $k$-ary $\LA$-constrained spaces with the global extension property).
\end{theorem}

\begin{example}[Priestley duality for finitely valued positive MV-algebras] \label{ex: duality for pmv}
  Let us now work through what the above duality tells us in the case of $\LA \assign \PMVchain$, which we claim to be our concrete motivating example in the introduction. We already verified that the nontrivial algebra $\PMVchain$, which lacks the empty subalgebra, has no partial endomorphisms in Example~\ref{ex: pmv chain has no partial endomorphisms}. Because $\PMVchain$ has a lattice reduct, it has a majority term, so the NU Duality Theorem applies with $k \assign 2$.
  
  An $\LA$-Priestley space for $\LA \assign \PMVchain$ is now a compact space $X$ equipped with a separated family of algebras $\A_{I} \leq \PMVchain^{I}$ for $I \subseteq_{2} X$ satisfying the global extension property. Separatedness means that if $x \neq y$, then $\A_{\{ x, y \}}$ is not a subdiagonal of $\PMVchain^{I}$. The global extension property means that if $(x \mapsto a, y \mapsto b) \in \A_{x, y}$, then there is a continuous compatible function $f\colon X \to \PMVchain$ with $f_{x} = a$ and $f_{y} = b$.

  As in Example~\ref{ex: L-sets for L=2}, which described $\LA$-Priestley spaces in more concrete terms for $\LA \assign \Dtwo$, here too the algebra $\PMVchain^{2}$ has two important subalgebras, namely
\begin{align*}
  \LeftAlg & \assign \set{\pair{a}{b} \in \PMVchain^{2}}{a \leq b}, & \RightAlg & \assign \set{\pair{a}{b} \in \PMVchain^{2}}{a \geq b}.
\end{align*}

  We claim that each subalgebra $\C$ of $\PMVchain^{2}$ is either a product algebra, i.e.\ it has the form $\C_{1} \times \C_{2}$ for some $\C_{1}, \C_{2} \leq \PMVchain$ or a subalgebra of $\LeftAlg$ or $\RightAlg$. To see this, suppose that $\pair{a}{b}, \pair{c}{d} \in \C$ with $a \nleq b$ and $c \ngeq d$. Then by Example~\ref{ex: pmv chain has no partial endomorphisms} there is a unary term $t(x)$ such that $t^{\PMVchain}(a) = 1$ and $t^{\PMVchain}(b) = 0$, and likewise a unary term $u(x)$ such that $u^{\PMVchain}(c) = 0$ and $u^{\PMVchain}(d) = 1$. It follows that $\pair{1}{0} = t^{\PMVchain^{2}}(\pair{a}{b}) \in \C$ and $\pair{0}{1} = u^{\PMVchain^{2}}(\pair{c}{d}) \in \C$. Now take $\C_{1} \assign \pi_{1}[\C]$ and $\C_{2} \assign \pi_{2}[\C]$. Clearly $\C \leq \C_{1} \times \C_{2}$. To prove that $\C_{1} \times \C_{2} \leq \C$, suppose that $p \in \C_{1}$ and $s \in \C_{2}$. Then there are $q, s \in \PMVchain$ with $\pair{p}{q}, \pair{r}{s} \in \C$, and $\pair{p}{s} = (\pair{1}{0} \wedge \pair{p}{q}) \vee (\pair{0}{1} \wedge \pair{r}{s}) \in \C$.

  Switching to the notation $\A_{x, y}$ per the discussion following Remark~\ref{rem: constraints}, we can define a binary relation $\leq$ on $X$ as follows:
\begin{align*}
  x \leq y \iff \A_{x,y} \leq \LeftAlg \iff f_{x} \leq f_{y} \text{ for all compatible functions $f$},
\end{align*}
  where the second equivalence holds by the global extension property. We claim that this is a partial order. It is a reflexive relation because $\A_{x,x}$ is a subdiagonal of $\PMVchain^{2}$. It is antisymmetric because if $x \leq y \leq x$, then $\A_{x, y}$ is a subalgebra of both $\LeftAlg$ and $\RightAlg$, so it is a subdiagonal, and by separation it follows that $x = y$. Finally, it is transitive because if $x \leq y \leq z$ then $f_{x} \leq f_{y} \leq f_{z}$ for all compatible functions $f$ and so $x \leq z$.

  We can therefore recover an $\LA$-Priestley space $\X$ from the following data:
\begin{enumerateroman}
\item the family $(\A_{x})_{x \in X}$ of subalgebras of $\PMVchain$,
\item the partial order $\leq$ on $X$,
\item the family $(\A_{x,y})_{x \leq y}$ of subalgebras of $\LeftAlg$.
\end{enumerateroman}
  The compatible functions of $\X$ are precisely the continuous functions $f\colon X \to \PMVchain$ such that $u \leq v$ in $X$ implies $\pair{f_{u}}{f_{v}} \in \A_{u,v}$.

  Conversely, the above data determines an $\LA$-Priestley space if and only if
\begin{enumerateroman}
\item\label{i:mvp-subdirect} the inclusion $\A_{x, y} \leq \A_{x} \times \A_{y}$ is subdirect for $x \leq y$,
\item\label{i:mvp-diagonal} the algebra $\A_{x, y}$ is a subdiagonal if and only if $x = y$, and
\item\label{i:mvp-extension} for each $\pair{a}{b} \in \A_{x, y}$ there is a continuous compatible (in the above sense) function $f\colon X \to \PMVchain$ such that $f_{x} = a$ and $f_{y} = b$.
\end{enumerateroman}
  Item~\eqref{i:mvp-subdirect} ensures that the families $(\A_{x})_{x \in X}$ and $(\A_{x, y})_{x \leq y}$ determine a subdirect family of binary constraints $(\A_{I})_{I \subseteq_{2} X}$. Item~\eqref{i:mvp-diagonal} ensures that this family is separated. Finally, item~\eqref{i:mvp-extension} ensures the global extension property.

  Let us call a structure satisfying the above three conditions an \emph{MV-Priestley space}. A morphism of MV-Priestley spaces $\phi\colon \X \to \Y$ (whose families of algebras are denoted by $\A$ and $\B$ respectively) will be a continuous order-preserving map $\phi$ such that $\B_{\phi(x)} \leq \A_{x}$ for all $x \in \X$ and $\B_{\phi(x),\phi(y)} \leq \A_{x,y}$ for all $x \leq y$ in $\X$. These are precisely the conditions ensuring that $\phi$ is a continuous $\LA$-map between the associated $\LA$-Priestley spaces.

  The NU Duality Theorem in case $\LA \assign \PMVchain$ thus yields a duality between the category of finitely valued MV-algebras and the category of MV-Priestley spaces.

  In case where $\LA$ is not the full algebra $\PMVchain$ but only a finite subalgebra of $\PMVchain$, i.e.\ in case $\LA$ is the positive MV-algebra reduct of a finite MV-chain, we in effect recover the duality recently 
formulated by Poiger~\cite{Poiger2024}.
\end{example}

\subsection{\texorpdfstring{Duality for $\LA$-constrained spaces: the unary case}{Duality for L-constrained spaces: unary case}}
\label{sec: unary case}

  The global extension property is a brute force principle: it directly postulates that the local constraints fit together globally without saying how or why. In some cases, this is the best we can do. The Priestley separation axiom is, after all, also a brute force principle of precisely this sort which, although clothed in the trappings of clopen upsets, directly postulates the existence of a global continuous compatible function into~$\Dtwo$. In at least two special cases, however, we can do better: in the case in which $\LA$ has the unary Baker--Pixley property and in the case of finite spaces.

  Recall that a separated unary $\LA$-constrained is a topological space $X$ with a continuous subdirect family of unary constraints $(\A_{I})_{I \subseteq_{1} X}$ such that the equality relation is closed (that is, $X$ is Hausdorff) and for all $x \neq y$ in $X$ there are $a \in \A_{x}$ and $b \in \A_{y}$ with $a \neq b$.

\begin{theorem} \label{thm: extension property unary case}
  A compact separated unary $\LA$-constrained space $\X$ has the global extension property if and only if it is topologically a Stone space.
\end{theorem}

\begin{proof}
  Every compact separated $\LA$-constrained space with the global extension property is isomorphic to an $\LA$-constrained space of the form $\Cons \X$ for some compact separated $\LA$-space $\X$ by Theorem~\ref{thm: bp representation}, and therefore it is a Stone space by Lemma~\ref{lemma: properties of separated l-spaces}.
  
  Conversely, consider a compact separated $\LA$-constrained space $\X$ which is Stone. We need to show that it has the global extension property.
    
  Let $x \in X$, let $a \in \A_{x}$, and let us prove that there is some $f \in \CComp \X$ such that $f_{x} = a$. Given $a \in \A_{x}$, the constraint $\alg{A}_{x}$ is non-empty, and so, by subdirectness, $\A_\varnothing$ is nonempty, and so, by subdirectness, every $\alg{A}_{y}$ with $y \in \X$ is non-empty.
  In particular, for each $y \in \X$ we can choose some $a_{y} \in \alg{A}_{y}$, taking $a_{x} \assign a$. Because the family of constraints is continuous, each $y \in \X$ has an open neighborhood $U_{y}$ such that $a_{y} \in \A_{z}$ for all $z \in U_{y}$. Because $\X$ is Stone, we may choose $U_{y}$ clopen. The family $(U_{y})_{y \in \X}$ forms an open cover of $\X$, so by compactness some finite subfamily of $(U_{y})_{y \in \X}$ covers $\X$. We may assume that this subfamily contains $U_{x}$. Because each $U_{y}$ is clopen, we can transform this finite subfamily into a cover of $\X$ by disjoint clopens $V$ such that for each of these clopens $V$ there is some $y \in \X$ with $V \subseteq U_{y}$. We can moreover choose this family of disjoint clopens so that it contains $U_{x}$. Then any function $f$ which takes the constant value $a$ on $U_{x}$ and which on each of the other sets $V$ takes a constant value $a_{y}$ for some $y \in \X$ with $V \subseteq U_{y}$ is the desired $f \in \CComp \X$ with $f_{x} = a$.  This proves that every partial function on a subset of cardinality $1$ can be extended globally.

  Let now $* \in \alg{A}_\varnothing$.
  Then, by subdirectness, for every $x \in \X$ the set $\alg{A}_x$ is nonempty, and we can then play the same trick as above to extend $*$ to some $f \in \CComp \X$. (Alternatively, one can proceed by cases: if $\X = \varnothing$, then $*$ is the desired function in $\CComp \X$, and otherwise one picks a designated $x \in \X$, for which there would exist $a \in \A_x$, and piggybacks on the case above.)
  
  Let now $x \not\approx y$.
  Then, by definition of unary $\LA$-constrained space, there are $a \in \A_{\{x\}}$ and $b \in \A_{\{y\}}$ with $a \neq b$.
  Then we can again play a similar trick to the one above to produce $f \in \CComp \X$ such that $f_x = a$ and $f_y = b$.
  
  This proves the global extension property.
\end{proof}

\begin{remark}
	We give an example of a compact separated unary $\LA$-constrained space that is not topologically a Stone space (equivalently, that lacks the global extension property).
	Take $\LA$ to be the two-element boolean algebra $\{0,1\}$, $X$ to be the unit interval $[0,1]$ with the Euclidean topology, $\A_x \coloneqq \LA$ for each $x \in [0,1]$ (and $\A_\varnothing$ as the singleton $\LA^\varnothing$), with $\approx$ as the equality relation.

	The compatible functions are then the continuous maps $f\colon [0, 1] \to \{ 0, 1 \}$, but each such map is constant.
	Thus, even though for each $x \in [0, 1]$ and $a \in \A_{x}$ there is some continuous compatible $f$ with $f_{x} = a$, the further requirement in the global extension property that $x \not\approx_{\X} y$ implies the existence of some continuous compatible $f$ with $f_{x} \neq f_{y}$ is violated.
\end{remark}

\begin{remark} \label{r:Stone-separated-unary}
	If $\LA$ has two distinct constant terms, then a Stone separated unary $\LA$-constrained space $\X$ is determined by 
	\begin{enumerateroman}
	
		\item
		a Stone space $X$, and
		
		\item
		for each $x \in X$, a subalgebra $\A_x$ of $\LA$
	\end{enumerateroman}
	such that for all $a \in \LA$ the set $\{x \in X \mid a \in \A_x\}$ is open. The constraint $\A_{\emptyset}$ does not need to be specified, since it is the singleton algebra if $\LA$ has a constant term.

	If $\LA$ has a constant but not necessarily two distinct constant terms, we need to further require that for all $x \neq y$ in $X$ there are $a \in \A_x$ and $b \in \A_y$ with $a \neq b$.

  Finally, if the algebra $\LA$ has no constants, then in addition to the openness and separation conditions above, we need to specify
	\begin{enumerateroman}	 \setcounter{enumi}{2}
		\item
		a subalgebra $\A_\varnothing$ of the singleton $\LA^\varnothing$
	\end{enumerateroman}
	such that, for all $x \in X$, $\A_\varnothing = \varnothing$ if and only if $\A_x = \varnothing$.	
	
	A continuous $\LA$-constrained map from $\X$ to $\Y$ then amounts to a continuous map $\phi \colon X \to Y$ such that $\B_{\phi(x)} \subseteq \A_{x}$ for every $x \in X$ and $\B_{\varnothing} \subseteq \A_{\varnothing}$ (a condition we can ignore if $\LA$ has at least a constant symbol). In other words, the map $\phi$ is \emph{constraint-decreasing}. This definition of unary $\LA$-constrained spaces and continuous $\LA$-constrained maps specializes precisely to the definitions of Cignoli and Marra~\cite{CignoliMarra2012} for the case of $\LA \assign [0, 1]$.
\end{remark}

\begin{theorem}[Baker--Pixley Representation Theorem: unary case] \label{t:BP-unary}
  If $\LA$ has the unary Baker--Pixley property, then the category of compact separated $\LA$-spaces is isomorphic to the category of Stone separated unary $\LA$-constrained spaces.
\end{theorem}

\begin{proof}
	By \cref{thm: bp representation,thm: extension property unary case}.
\end{proof}

  Let us again summarize the situation so far in the unary case.

	Suppose that $\LA$ is a nontrivial algebra with only trivial partial endomorphisms and without the empty subalgebra. If finitely valued $\LA$-algebras are relatively congruence distributive with respect to some prevariety containing $\LA$, then, by Theorem~\ref{thm: cd duality},
\begin{gather*}
  \text{finitely valued $\LA$-algebras} \\ \mathrel{\cong^{\op}} \\ \text{compact separated $\LA$-spaces}.
\end{gather*}
	
	The unary case of the Baker--Pixley Representation Theorem (\cref{t:BP-unary}) yields the following categorical equivalence for any algebra $\LA$ with the unary Baker--Pixley property:
	\begin{gather*}
	  \text{compact separated $\LA$-spaces} \\ \cong \\ \text{Stone separated unary $\LA$-constrained spaces}.
	\end{gather*}
  
  Finally, the conditions that $\LA$ has a majority term and that $\LA \times \LA$ only has subdiagonal or product subalgebras ensure some of the key requirements above, namely that $\HOp \SOp \POp(\LA)$ is a congruence distributive variety (\cref{thm: bp and nu}) and that $\LA$ has the unary Baker--Pixley property (\cref{f:unary-in-terms-of-binary}). Putting all of this together now yields the following theorem. (We refer to \cref{r:Stone-separated-unary} for a simple description of Stone separated unary $\LA$-constrained spaces.)
 
 \begin{theorem}[NU Duality Theorem: the unary case] \label{thm: nu duality unary case}
 	Let $\LA$ be a nontrivial algebra with only trivial partial endomorphisms and without the empty subalgebra.
	If $\LA$ has a majority term and $\LA \times \LA$ only has subdiagonal and product subalgebras, then there is a dual equivalence between the category of finitely valued $\LA$-algebras and the category of Stone separated unary $\LA$-constrained spaces.
 \end{theorem}
 
 \begin{proof}
 	A proof precedes the statement. Alternatively, \cref{thm: nu duality,thm: extension property unary case} and~Facts~\ref{f:unary-in-terms-of-binary} and~\ref{fact: binary equals unary} provide a proof that leverages the results on the binary case.
 \end{proof}

\begin{example}
  Consider the case of the standard MV-chain $\LA \assign \MVchain$. The algebra $\MVchain$ has the binary Baker--Pixley property because it has a lattice reduct and therefore a majority term. Moreover, each subalgebra of $\MVchain^{2}$ is either a subdiagonal or it has the form $\A_{1} \times \A_{2}$ for some $\alg{A}_{1}, \A_{2} \leq \MVchain$, and hence $\MVchain$ has the unary Baker--Pixley property.
   Therefore, compact separated $\LA$-constrained spaces for $\LA \assign \MVchain$ are precisely the Stone separated unary $\LA$-constrained spaces.

  Recalling Examples~\ref{example: finitely valued mv} and~\ref{example: locally finite mv} and Fact~\ref{fact: scott topology}, our duality therefore yields precisely the duality of Cignoli \& Marra~\cite{CignoliMarra2012} for finitely valued MV-algebras in case $\LA \assign \MVchain$ and the duality of Cignoli, Dubuc \& Mundici~\cite{CignoliDubucMundici2004} in case $\LA \assign [0, 1]_{\Q}$ (the rational MV-chain).
\end{example}

\subsection{\texorpdfstring{Duality for $\LA$-constrained spaces: the case of finite spectrum}{Duality for L-constrained spaces: the case of finite spectrum}}

  In this last subsection, we show that for finite $\LA$-constrained spaces the global extension property is equivalent to what we call the local extension property, which, unlike the global extension property, is a first-order condition. For example, in the case of $\LA \assign \Dtwo$ the Priestley separation axiom reduces in the finite case to the transitivity of the reflexive and antisymmetric relation $\leqq$.
  
  In what follows we assume that $k \geq 2$.

\begin{definition}
  A $k$-ary $\LA$-constrained space $\X$ has the \emph{$n$-ary local extension property} if for each $I \subseteq_{n} X$, each $j \in X$, and each continuous compatible $g\colon I \to \LA$ there is some continuous compatible $f\colon I \cup \{ j \} \to \LA$ such that $\restrict{f}{I} = g$.
\end{definition}

\begin{remark} \label{rem: local extension}
  The $n$-ary local extension property implies the $m$-ary local extension property for each $m \leq n$. The global and local extension properties for a $k$-ary $\LA$-constrained space $\X$ are related as follows. The global extension property implies the $k$-ary local extension property. Conversely, if $\X$ has finite cardinality $n+1$, then the $n$-ary local extension property for $\X$ implies the global extension property.
\end{remark}

\begin{remark}
  In contrast to the global extension property, the $n$-ary local extension property is a first-order condition: for all points $x_{1}, \dots, x_{n}, y \in \X$ and all values $a_{1}, \dots, a_{n} \in \LA$, if the function $(x_{1} \mapsto a_{1}, \dots, x_{n} \mapsto a_{n})$ is compatible, then there is some $b \in \LA$ such that the function $(x_{1} \mapsto a_{1}, \dots, x_{n} \mapsto a_{n}, y \mapsto b)$ is compatible. The compatibility of $(x_{1} \mapsto a_{1}, \dots, x_{n} \mapsto a_{n})$ further translates into a first-order condition if we view $\X$ as a relational structure with a $k$-ary predicate indexed by tuples $\tuple{a} \in \LA^{k}$ (and a nullary predicate to cover the case of $\alg{A}_{\emptyset}$ in case $X = \emptyset$) and allow for existential quantification over these indices.
\end{remark}

\begin{example} \label{example: priestley local extension}
  In case $\LA \assign \Dtwo$, the binary local extension property states that for all $x, y, z \in \X$
\begin{align*}
  \pair{a}{b} \in \alg{A}_{x, y} \implies \pair{a}{c} \in \alg{A}_{x, z} \text{ and } \pair{c}{b} \in \alg{A}_{z, y} \text{ for some } c \in \Dtwo.
\end{align*}
Since $\pair{0}{0}, \pair{1}{1} \in \alg{A}_{u, v}$ for all $u, v \in X$, and moreover $\pair{a}{b} \in \alg{A}_{x, y}$ if and only if $\pair{b}{a} \in \alg{A}_{y, x}$, this implication is equivalent to the claim that if $\pair{1}{0} \in \alg{A}_{x, y}$, then for each $z \in X$ either $\pair{1}{0} \in \alg{A}_{x, z}$ or $\pair{1}{0} \in \alg{A}_{z, y}$. That is, if $x \nleq y$, then $x \nleq z$ or $z \nleq y$, which is simply the transitivity of $\leq$ stated contrapositively.
\end{example}

\begin{definition}
	Let $\maj$ be a $(k+1)$-ary near unanimity term on $\LA$ for some $l \geq 2$.
	A subset $M$ of $\LA$ is \emph{convex with respect to $\maj$} if for all $a_{1}, \dots, a_{k+1} \in \LA$
	\begin{align*}
		\text{$a_{i} \in M$ for all but at most one $i \in \{ 1, \dots, k+1 \}$ $\implies$ $\maj(a_{1}, \dots, a_{l+1}) \in M$.}
	\end{align*}
\end{definition}

  Notice that the definition of a near unanimity term states precisely that each singleton subset of $\LA$ is convex.

\begin{remark}
	The terminology stems from the fact that a sublattice $M$ of an algebra $\LA$ with a lattice reduct is order-convex if and only if it is convex with respect to the majority term
	\[
	\maj(x_{1}, x_{2}, x_{3}) \assign (x_{1} \wedge x_{2}) \vee (x_{2} \wedge x_{3}) \vee (x_{3} \wedge x_{1}).
	\]
	Indeed, suppose that $M$ is order-convex, and let $x_1, x_2 \in M$.
	Then
	\[
	\maj(x_{1}, x_{2}, x_{3}) = (x_{1} \wedge x_{2}) \vee (x_{2} \wedge x_{3}) \vee (x_{3} \wedge x_{1}) \leq x_1 \vee x_2 \vee x_1 = x_1 \vee x_2 \in M,
	\]
	and
	\[
	\maj(x_{1}, x_{2}, x_{3}) = (x_{1} \wedge x_{2}) \vee (x_{2} \wedge x_{3}) \vee (x_{3} \wedge x_{1}) \geq x_1 \wedge x_2 \in M,
	\]
	and thus $\maj(x_{1}, x_{2}, x_{3}) \in M$.
	Thus, $M$ is convex with respect to $\maj$.
	Conversely, suppose that $M$ is convex with respect to $\maj$.
	Let $x_1,x_3 \in M$ and $x_1 \leq x_2 \leq x_3$.
	Then
	\[
	x_2 = x_1 \vee x_2 \vee x_1 = (x_{1} \wedge x_{2}) \vee (x_{2} \wedge x_{3}) \vee (x_{3} \wedge x_{1}) = \maj(x_{1}, x_{2}, x_{3}) \in M,
	\]
	and so $M$ is order-convex.
\end{remark}

\begin{lemma} \label{lemma: convex intersection property}
  Let $M_{1}, \dots, M_{n}$ be convex subsets of $\LA$ with respect to some $(k+1)$-ary near unanimity term $\maj$. If $\bigcap_{i \in I} M_{i}$ is non-empty for all $I \subseteq_{k} \{ 1, \dots, n \}$, then $M_{1} \cap \dots \cap M_{n}$ is non-empty.
\end{lemma}

\begin{proof}
  We prove the claim by induction on $n$. For $n \leq k$ the claim holds trivially. Now suppose that the claim holds for some $n \geq k$ and consider convex subsets $M_{1}, \dots, M_{n+1}$ of $\A$. By the inductive hypothesis for each $i \in \{ 1, \dots, n+1 \}$ there is some $a_{i} \in \A$ such that $a_{i} \in M_{j}$ for each $j \neq i$. Take $a \assign \maj(a_{1}, \dots, a_{k+1})$. By convexity, $a \in M_{j}$ for all $j \in \{ 1, \dots, n+1 \}$.
\end{proof}

  Below we use the notation $(f, y \mapsto a) \assign f \cup \pair{y}{a}$ for the extension of a function $f$ with $y \notin \dom f$ by taking $y \mapsto a$.

\begin{lemma} \label{lemma: possible extensions are convex}
  Let $\maj$ be a $(k+1)$-ary near unanimity term on~$\LA$ and let $\X$ be an $\LA$-constrained space. Consider $I \subseteq_{k-1} X$ and $y \in X - I$. Then, for each continuous compatible function $f\colon I \to \LA$, the set
\begin{align*}
  M_{f, y} \assign \set{a \in \A_{y}}{(f, y \mapsto a) \in \A_{I \cup \{ y \}}}
\end{align*}
  is a convex subset of $\A_{y}$ with respect to $\maj$.
\end{lemma}

\begin{proof}
  Let $J \assign I \cup \{ y \}$. Take $a_{1}, \dots, a_{k} \in M_{f, y} \subseteq \A_{y}$ and $a_{k+1} \in \A_{y}$. That is, $(f, y \mapsto a_{i}) \in \A_{J}$ for all $i \in \{ 1, \dots, k \}$. Since $\A_{y} = \pi_{J \to y}[\A_{J}]$, there is some continuous compatible $g\colon I \to \LA$ such that $(g, y \mapsto a_{k+1}) \in \A_{J}$. Applying the near unanimity term $\maj$ to $(f, y \mapsto a_{1}), \dots, (f, y \mapsto a_{n}), (g, y \mapsto a_{k+1})$ yields a function $(f, y \mapsto a) \in \alg{A}_{J}$ for $a \assign \maj(a_{1}, \dots, a_{k+1})$. Thus $a \in M_{f, y}$.
\end{proof}

\begin{theorem}[From local extension to global extension] \label{t:local-to-global}
  Suppose that $\LA$ has a near unanimity term of arity $k+1$. Then each finite $k$-ary $\LA$-constrained space with the $k(k-1)$-ary local extension property has the global extension property.
\end{theorem}

\begin{proof}
  We prove by induction over $n$ that if a $k$-ary $\LA$-space $\X$ has the $k(k-1)$-ary local extension property, then it has the $n$-ary local extension property for each $n \geq 0$. If $\X$ is finite, the global extension property then follows by Remark~\ref{rem: local extension}.

  The case of $n \leq k(k-1)$ is covered by the $k(k-1)$-ary local extension property. Now suppose that $\X$ has the $n$-ary local extension property for some $n \geq k(k-1)$, and let us prove that it has the ($n+1$)-ary local extension property. Consider $J \subseteq_{n+1} X$, a compatible function $g\colon J \to \LA$, and some $y \in X - J$. For each $I \subseteq_{k-1} J$ the set $M_{I} \assign \set{a \in \A_{y}}{(\restrict{g}{I}, y \mapsto a) \in \A_{I \cup \{ y \}}}$ is a convex subset of $\alg{A}_{y}$ by Lemma~\ref{lemma: possible extensions are convex}. Given $k$ sets $I_{1}, \dots, I_{l} \subseteq_{k-1} J$ the intersection $M_{I_{1}} \cap \dots \cap M_{I_{l}}$ is non-empty by the $k(k-1)$-ary local extension property. By Lemma~\ref{lemma: convex intersection property} it follows that $M \assign \bigcap \set{M_{I}}{I \subseteq_{k-1} J}$ is non-empty. But for each $a \in M$ the function $(g, y \mapsto a)$ is compatible.
\end{proof}

\begin{theorem}[NU Duality Theorem: the case of finite spectrum]\label{thm: nu duality finite case}
  Let $\LA$ be a nontrivial algebra with a near unanimity term of arity $k+1$, with only trivial partial endomorphisms and that lacks the empty subalgebra. Then the category of $\LA$-algebras with a finite spectrum is equivalent to the category of finite separated $k$-ary $\LA$-constrained spaces with the $k (k - 1)$-ary local extension property.
\end{theorem}

\begin{proof}
  This follows from the above theorem and Theorem~\ref{thm: nu duality}.
\end{proof}

The special case of $k \assign 2$, which in particular applies if $\LA$ has a nontrivial bounded lattice reduct and no trivial partial endomorphisms, is the following.

\begin{corollary}
	Let $\LA$ be a nontrivial algebra with a majority term which only has trivial partial endomorphisms and lacks the empty subalgebra. Then the category of $\LA$-algebras with a finite spectrum is equivalent to the category of finite separated binary $\LA$-constrained spaces with the binary local extension property.
\end{corollary}

\begin{example}[Priestley duality for positive MV-algebras with finite spectrum] \label{example: pmv local extension}
  In Example~\ref{ex: duality for pmv}, we discussed at length how $\LA$-Priestley spaces can be understood in concrete terms in the case of $\LA \assign \PMVchain$ as MV-Priestley spaces: compact spaces equipped with a family $(\A_{x})_{x \in X}$ of subalgebras of $\PMVchain$, a partial order $\leq$, and a family $(\A_{x,y})_{x \leq y}$ of subalgebras of $\PMVchain^{2}$ satisfying certain conditions, including a global extension property. To obtain the restriction of this duality to positive MV-algebras with finite spectrum, it suffices to replace the global extension by the local extension property in the following form: for all $x, y, z \in \X$
\begin{align*}
  \pair{a}{b} \in \B_{x, y} \implies \pair{a}{c} \in \B_{x, z} \text{ and } \pair{c}{b} \in \B_{z, y} \text{ for some } c \in \PMVchain,
\end{align*}
  where
\begin{align*}
  \B_{x, y} \assign \begin{cases} & \A_{x, y} \text{ in case } x \leq y, \\ &  \set{\pair{b}{a} \in \LA^{2}}{\pair{a}{b} \in \A_{y, x}} \text{ in case } y < x, \\ & \A_{x} \times \A_{y} \text{ in case } x \parallel y \text{ (that is, $x \nleq y \nleq x$).}\end{cases}
\end{align*}
  More explicitly, it suffices to consider the following cases:
\begin{enumerateroman}
\item in case $x < z < y$,
\begin{align*}
  \pair{a}{b} \in \A_{x, y} \implies \pair{a}{c} \in \A_{x, z} \text{ and } \pair{c}{b} \in \A_{z, y} \text{ for some } c \in \PMVchain,
\end{align*}
\item in case $x < y < z$,
\begin{align*}
  \pair{a}{b} \in \A_{x, y} \implies \pair{a}{c} \in \A_{x, z} \text{ and } \pair{b}{c} \in \A_{y, z} \text{ for some } c \in \PMVchain,
\end{align*}
\item in case $z < y < x$,
\begin{align*}
  \pair{b}{a} \in \A_{y, x} \implies \pair{c}{a} \in \A_{z, x} \text{ and } \pair{c}{b} \in \A_{z, y} \text{ for some } c \in \PMVchain,
\end{align*}
\item in case $x, y < z$ and $x \parallel y$,
\begin{align*}
  a \in \A_{x} \text{ and } b \in \A_{y} \implies \pair{a}{c} \in \A_{x, z} \text{ and } \pair{b}{c} \in \A_{y, z} \text{ for some } c \in \PMVchain,
\end{align*}
\item in case $z < x, y$ and $x \parallel y$,
\begin{align*}
  a \in \A_{x} \text{ and } b \in \A_{y} \implies \pair{c}{a} \in \A_{z, x} \text{ and } \pair{c}{b} \in \A_{z, y} \text{ for some } c \in \PMVchain.
\end{align*}
\end{enumerateroman}
\end{example}

\section*{Acknowledgments}
The first author's research was funded by UK Research and Innovation (UKRI) under the UK government’s Horizon Europe funding guarantee (grant number EP/Y015029/1, Project ``DCPOS''). The ``Horizon Europe guarantee'' scheme provides funding to researchers and innovators who were unable to receive their Horizon Europe funding (in this case, a Marie Skłodowska-Curie Actions (MSCA) grant) while the UK was in the process of associating. The second author's work was funded by the grant 2021 BP 00212 of the grant agency AGAUR of the Generalitat de Catalunya.

\newcommand{\etalchar}[1]{$^{#1}$}

\end{document}